%% file: msps_core.tex
\documentclass[]{siamart190516}

\definecolor{dkgreen}{rgb}{0,0.4,0}
\definecolor{dkred}{rgb}{0.8,0.0,0}

\newcommand{\MK}[1]{{\color{dkred}{#1}}}

\input{msps_shared}

\begin{document}
	\maketitle
	
	\begin{abstract}
		Any autonomous nonlinear dynamical system can be viewed as a superposition of infinitely many linear processes, through the so-called Koopman mode decomposition.
		Its data-driven approximation- Dynamic Mode Decomposition (DMD)- has been extensively developed and deployed across a plethora of fields. 
		In this work, we study the effect of subtracting the temporal mean on the DMD approximation, for observables possessing only a finite number of Koopman modes.
		
		Pre-processing time-sequential training data by removing the temporal mean has been a point of contention in the Companion matrix formulation of DMD. 
		This stems from the potential of said pre-processing to render DMD equivalent to a temporal Discrete Fourier Transform (DFT). 
		We prove that this equivalence is impossible when the training data is linearly consistent and the order of the DMD model exceeds the number of Koopman modes. 
		Since model order and training set size are synonymous in this variant of DMD, the parity of DMD and DFT can, therefore, be indicative of inadequate training data.
		
	\end{abstract}
	
	

	\tableofcontents
	\section{Introduction}\label{s:Introduction} 
    Models of nonlinear phenomena are particularly useful when they strike a balance between simplicity and generalizability.    
    Such balanced models can be constructed in a principled and rigorous fashion through the Koopman operator \cite{koopman1931hamiltonian,mezic2004comparison,mezic2005spectral}.
    This object is typically infinite-dimensional, but is always linear.
    Hence, the eigenfunctions of the Koopman operator can be used to build generalizable linear models of non-linear dynamics.
    Consequently, algorithms for computing these spectral quantities have been the focus of sustained research.
    A substantial fraction of these efforts concerns the so-called Dynamic Mode Decomposition (DMD) \cite{ rowley2009spectral, schmid2010dynamic, williams2015data}.

    Procedurally, DMD is simply a linear fit followed by an eigen-decomposition.
    Hence, it has been applied to data from a plethora of natural and engineered phenomena; see for example the survey \cite{brunton2021modern}.
    In parallel, theoretical investigations have probed the effect of time delays, choice of observables and convergence to Koopman spectral quantities- see \cite{korda2018convergence} as well as the reviews \cite{schmid2022dynamic} and \cite{otto2021koopman} for a detailed discussion.
    This work continues the conversation on the pragmatic issue of removing the temporal mean, and its theoretical implications for DMD.

	The infinite-time average, when it exists, is the leading term of the Koopman mode decomposition \cite{mezic2005spectral}. 
    Its removal brings focus on the temporally varying parts of the process, and has been adopted in a diverse set of analyses \cite{brunton2016extracting, proctor2015discovering, sarmast2014mutual, avila2020data}. 
	Nonetheless, in practice, this step may render DMD equivalent to a temporal Discrete Fourier Transform (DFT) \cite{chen2012variants}. 
	The negative consequences of such an equivalence include the misrepresentation of dissipative dynamics and a bijection between the trajectory length and the DMD-based estimates of the Koopman eigenvalues. 

    The relationship between mean-subtracted DMD and DFT is substantially clarified in a subsequent work by Hirsch and co-workers \cite{hirsh2019centering}.
    Although their primary contribution is an \emph{alternative} to temporal mean removal, they also study the original pre-processing step in itself.
    First, they point out that the original work \cite{chen2012variants} only provides a \emph{sufficient} condition for the equivalence of mean-subtracted DMD and DFT.
    Then, assuming the observables used in DMD possess only a finite number of Koopman modes, they prove that the sufficiency condition is not satisfied when certain sub-matrices of the training set are connected by a linear map (linear consistency)
    and the training set size exceeds the number of Koopman modes (over-sampling).
    Finally, they provide numerical evidence that suggests linear consistency and over-sampling can actually negate DMD-DFT equivalence, i.e., be a sufficient condition for non-equivalence.   

    We begin by developing a \emph{necessary and sufficient condition} for the equivalence of mean-subtracted DMD and DFT (\cref{thm:musubNDFT}).
    Using this geometric condition, we construct counterexamples to show that the sufficient condition in \cite{chen2012variants} is not a necessary condition (\cref{thm:XmsnotLI_DMD_DFTequiv}).
    
    Then, we proceed to our primary contribution in \cref{thm:DMDDFT_Oversampled} - linear consistency\footnote{An appropriate number of time delays \cite{brunton2017chaos,arbabi2017ergodic} can ensure linear consistency.} and over-sampling prevent the equivalence of mean-subtracted DMD and DFT.
    This result can be paraphrased as follows:
   	\begin{theorem}\label{thm:MainResult__Coarsely}
    	Suppose we select a finite number of functions from the span of \(r\) Koopman eigen-functions with distinct eigenvalues. 
    	Say this collection of functions, aka dictionary, is used to observe the underlying discrete-time dynamical system at \(n+1\) sequential time instances, starting at an initial state that does not lie on the zero level set of any of the \(r\) eigenfunctions.
    	Let the resulting mean-subtracted DMD eigenvalues be denoted as \(\msubDMDEvals\).
 
    	If the time sequential observations are linearly consistent (\cref{defn:Linear_Consistency}) and there are more DMD eigenvalues \((\undelayedDMDorder)\) than Koopman eigenvalues \((\finitekissdim)\), then, mean-subtracted DMD is not a temporal DFT.
		\begin{displaymath}
    		\textrm{Linear consistency}~~\&~~ n \geq \finitekissdim + 1 \implies \msubDMDEvals ~~\neq~~ \{ z \neq 1~|~z^{\undelayedDMDorder+1}=1 \}.
    	\end{displaymath}
    \end{theorem}
    
    By the contrapositive, DMD-DFT equivalence can be used to decide if more data is needed to construct an accurate DMD model.
    Under the stronger condition of Koopman invariant observables, we also show the converse: DMD-DFT equivalence can result from a paucity of data (\cref{cor:DMDFT_Undersampled}).

    This paper is organized as follows: 
    \Cref{s:Notation} details the notation adopted in this work. 
    \Cref{s:Preliminaries} introduces the Koopman operator (\ref{ss:prelims_Koopman}) and Dynamic Mode Decomposition (\ref{ss:prelims_DMD}) before reviewing the effect of mean subtraction (\ref{ss:prelims_effect_of_msub}).
	Drawing on these, \Cref{s:New_MsubDMD_DFT} explores the relationship between mean subtracted DMD and DFT, resulting in an a posteriori check that alerts to data insufficiency. 
	\Cref{s:Numerics} sees the numerical verification of all major results on  linear time invariant systems (LTIs), the Van der Pol oscillator and the lid-driven cavity at various flow regimes. 
	\Cref{s:Conclusions} summarizes the findings and points out relevant questions that remain unanswered.
	\section{Notation}\label{s:Notation}
	The notation is mostly standard for our primary objects of interest: (finite) sets, vectors and matrices. 
	The elements of an object are written with subscripts that start at 0 or 1.
	The operation \(\#(\MK{\cdot}) \) returns the number of elements in the argument, as illustrated below.
	Sets are denoted with upper-case letters and are enclosed by curly braces. For example,
	\begin{displaymath}
	L ~:=~ \{\lambda_i\}_{i=1}^{r} ~=~ \{\lambda_1,\lambda_2,\cdots, \lambda_r\} ~\implies ~ \#(L) = r.
	\end{displaymath}
	
	Vectors are typically in \(\mathbb{C}^m\), labelled with bold-faced lower-case letters and represented as columns:
	\begin{displaymath}
	\myvec{v} ~=~ [v_i ]_{i=1}^{m} ~=~ \begin{bmatrix}
	v_1 \\ v_2 \\ \vdots\\ v_m
	\end{bmatrix} ~\implies~ \#(\myvec{v}) = m.
	\end{displaymath}
	The \(i\)-th standard basis vector is denoted as  \(\myvec{e}_i\):
	\begin{displaymath}
		\myvec{e}_1 ~=~
		\begin{bmatrix}
			1 \\ 0 \\ \vdots \\ 0
		\end{bmatrix},\quad
		\myvec{e}_2 ~=~
		\begin{bmatrix}
			0 \\ 1 \\ \vdots \\ 0
		\end{bmatrix}\quad \textrm{and}\quad
	\myvec{e}_m ~=~
	\begin{bmatrix}
		0 \\ 0 \\ \vdots \\ 1
	\end{bmatrix}.
	\end{displaymath}
	Constant vectors (those with only one unique coefficient) are represented by their value in boldface and dimension as a subscript.
	\begin{displaymath}
	\mathbf{1}_m = [1]_{i=1}^{m}.
	\end{displaymath}
	
	Matrices are in \( \mathbb{C}^{m \times n}\), denoted by bold-faced upper-case letters and written as ordered sets of columns.
	\begin{displaymath}
	\mymat{A} = [\myvec{a}_j ]_{j=1}^{n} = [a_{ij}]_{i,j=1}^{m,n} = \begin{bmatrix}
	\vline & \vline & &\vline\\
	\myvec{a}_1 & \myvec{a}_2 & \cdots & \myvec{a}_n \\
	\vline & \vline & &\vline
	\end{bmatrix} .
	\end{displaymath}
	The (Hermitian) transpose of \(\mymat{A}\) is written as (\(\mymat{A}^{H}\)) \(\mymat{A}^T\).
	The letter \(\mymat{I}\) is reserved for the identity matrix.
	\begin{displaymath}
	\mymat{I}_{n} ~=~ [\myvec{e}_j]_{j=1}^{n} ~=:~ [\delta_{ij}]_{i,j=1}^n.
	\end{displaymath}
	Constant matrices are written similar to their vector counterparts:
	\begin{displaymath}
	\mathbf{0}_{m \times n} = [\mathbf{0}_m]_{j=1}^{n} = [0]_{i,j=1}^{m,n}.
	\end{displaymath}
	The expression \(\mymat{T}[\myvec{c}]\) represents the companion matrix with 1 on the first sub-diagonal and \(\myvec{c}\) as the last column. 
    \begin{displaymath}
        \myvec{c}~\in~\mathbb{C}^n~\implies~\mymat{T}[\myvec{c}]~=~
        \begin{bmatrix} 
          & & & & c_1 \\
        1 & & & & c_2 \\
          & 1 & & & c_3 \\
          & & \ddots  & & \vdots \\
          & & & 1& c_n \\
        \end{bmatrix}.
    \end{displaymath}
    
	Sub-spaces are written in script. For instance, \(\mathcal{R}(\mymat{Z})\) and \(\mathcal{N}(\mymat{Z})\) denote the column-space and null-space of \(\mymat{Z}\) respectively.
	Orthogonal projectors play a crucial role in this work. 
	Under the standard inner product on \(\mathbb{C}^m\), \(\mathcal{P}_{\mathcal{W}}\) denotes the orthogonal projector onto the subspace \(\mathcal{W}\).
	If we let \(\mymat{A}^\dagger\) denote the Moore-Penrose pseudo-inverse of \(\mymat{A}\), we get:
	\begin{displaymath}
	\mathcal{P}_{\mathcal{R}(\mymat{A})} ~=~\mymat{A} \mymat{A}^\dagger, \quad \mathcal{P}_{\mathcal{N}(\mymat{A}^H)} ~=~ \mymat{I} - \mymat{A} \mymat{A}^\dagger.
	\end{displaymath}
	
	\section{Preliminaries}\label{s:Preliminaries}
    We provide an overview of the ideas needed to understand and build upon existing work regarding the effect of mean-subtraction on DMD.
    We begin with a brief introduction to the Koopman operator (\Cref{ss:prelims_Koopman}) and its approximation via the Dynamic Mode Decomposition (\Cref{ss:prelims_DMD}).
    Then, we consider augmenting DMD with a pre-processing step of mean-subtraction and discuss the concomitant issue of DMD-DFT equivalence (\Cref{ss:prelims_effect_of_msub}).    
    Subsequently, to facilitate a more detailed exposition, we introduce several technical constructs, including the notions of linear consistency (\cref{defn:Linear_Consistency}) and over-sampling (\cref{defn:SamplingRegimes}).
    Finally, we review the most relevant theoretical contribution from \cite{hirsh2019centering} (\cref{thm:ChenSuffMayNotHold}), highlight its shortcomings and set the stage for its resolution.
	\subsection{The Koopman operator}\label{ss:prelims_Koopman}
	Consider an autonomous discrete time dynamical system in \(\mathbb{R}^p\).
	\begin{equation}\label{eq:DS4SDMD}
	\state^{+} = \Gamma (\state) , \; \state \in \mathbb{R}^p.
	\end{equation}
	The map \(\Gamma\) describes how states (individual elements of \(\mathbb{R}^\ssdim\)) evolve in time- an admittedly geometric description.
	A complementary perspective is embodied by the so-called Koopman operator which describes the evolution of \emph{functions} on the state-space (aka observables).
	\begin{definition}[Koopman operator \cite{koopman1931hamiltonian}]\label{defn:Koopman_operator}
		Suppose \(\mathcal{H}\) denotes a vector space of complex-valued functions on \(\mathbb{R}^\ssdim\).
		If \(\mathcal{H}\) is closed under composition with \(\Gamma\) i.e.,
		\begin{displaymath}
			\forall~~\psi \in \mathcal{H},\quad \psi \circ \Gamma \in \mathcal{H},
		\end{displaymath}
		then, the Koopman operator, \(U\), associated to (\ref{eq:DS4SDMD}) is defined as follows:
	\begin{equation}\label{eq:Koopman_definition}
		\begin{aligned}
			U : \mathcal{H} &\rightarrow \mathcal{H} \\
			U \circ \psi &:= \psi \circ \Gamma.
		\end{aligned}
	\end{equation}
	\end{definition}	
	The function-space \(\mathcal{H}\) is usually infinite-dimensional.
	Even though this complexity is inherited by the Koopman operator, the saving grace is its' linearity:
	\begin{displaymath}
		\begin{aligned}
			U \circ (\alpha_1 \psi_1 + \alpha_2 \psi_2)
			~&=~
			(\alpha_1 \psi_1 + \alpha_2 \psi_2) \circ \Gamma \\
			~&=~
			\alpha_1~ \EmphasiseReduction{\psi_1 \circ \Gamma} ~+~ \alpha_2~ \EmphasiseReduction{\psi_2 \circ \Gamma} \\
			~&=~
			\alpha_1~ U \circ \psi_1 ~+~ \alpha_2~ U \circ \psi_2.
		\end{aligned}
	\end{displaymath}
	Pushing this further, we can look for eigen-functions i.e. functions that are merely scaled under the action of \koopman.
	\begin{definition}[Koopman eigen-functions (KEFs)]\label{defn:Koopman_Eigenfunctions}
		The function \(\phi\) is said to be an eigen-function of the Koopman operator \(U\) with eigen-value \(\lambda\) if and only if the following holds:
		\begin{equation}
			U \circ \phi  = \lambda ~\phi.
		\end{equation}
	\end{definition}
 	A most pertinent property is that linearity allows the eigen-functions\footnote{We assume that the choice of \(\mathcal{H}\) is such that non-trivial eigen-functions exist.} of \(U\) to describe the evolution of any function in their span.
 	We can formalize such an expansion as follows:
	\begin{definition}[Koopman Mode Decomposition (KMD) \cite{mezic2004comparison,mezic2005spectral}]\label{defn:Koopman_Mode_Decomposition}
		Suppose the function \(f\) lies in the span of the Koopman eigen-functions \(\{\phi_i\}_{i \in \mathcal{I}}\) with eigen-values \(\{\lambda_i\}_{i \in \mathcal{I}}\).
		Then, there exist numbers \(\{c_i\}_{i\in \mathcal{I}}\) such that
		\begin{displaymath}
			f ~=~ \sum_{i \in \mathcal{I}} c_i~\phi_i.
		\end{displaymath}
		The quantities \(\{c_i\}_{i\in \mathcal{I}}\) are termed the Koopman modes of \(f\) associated with eigen-values \(\{\lambda_i\}_{i \in \mathcal{I}}\).
		In addition, for any whole number \(n\), the action of the Koopman operator \(U\) is given by a super-position of its' actions on the constituent eigen-functions.
		\begin{displaymath}
			U^n \circ f ~=~ U^n \circ \left( \sum_{i \in \mathcal{I}} c_i~\phi_i \right)
			~=~ \sum_{i \in \mathcal{I}} c_i~ \left( U^n \circ  \phi_i \right)
			~=~ \sum_{i \in \mathcal{I}} c_i~ \lambda_i^n \phi_i .
		\end{displaymath} 
	\end{definition}
	To paraphrase, the Koopman mode decomposition  explains temporal observations of a potentially \emph{nonlinear} process \cref{eq:DS4SDMD} through a  \emph{super-position} of exponentials. 
    Unsurprisingly, many algorithms have been designed to approximate this spectral expansion, with DMD being the most prominent.
	
	\subsection{Dynamic Mode Decomposition}\label{ss:prelims_DMD}
	
	Dynamic Mode Decomposition \cite{rowley2009spectral} is an algorithm that takes time-sequential observations and returns estimates of the constituent Koopman eigen-values and modes.
	The input to DMD is generated by first choosing \(\dictionarylength\) observables of interest (collectively referred to as the dictionary), \(\{\psi_i\}_{i=1}^m \), evaluating them along an orbit, \(\{\Gamma^{j-1}\left(\state_1\right)\}_{j=1}^{\undelayedDMDorder+1}\), and collecting the resulting observations (aka snapshots) as the columns of a matrix \(\mymat{Z}\):
	\begin{equation}\label{eq:Zmat4CompDMD}
		\dictionary := [\psi_i]_{i=1}^\dictionarylength, 
		\quad 
		\mymat{Z} ~:=~ \left[ \dictionary\left( \Gamma^{j-1}\left(\state_1\right) \right) \right]_{j=1}^{\undelayedDMDorder+1}.
	\end{equation}
	DMD formulates a least squares problem using specific sub-matrices of \(\mymat{Z}\), uses the result to construct a Companion matrix whose eigen-values approximate those of the Koopman operator and eigen-vectors lead to estimates of the associated Koopman modes- a procedure that is detailed as \cref{alg:CompanionDMD}.
	\begin{algorithm}
		\caption{Dynamic Mode Decomposition (DMD)}
		\label{alg:CompanionDMD}
		\begin{algorithmic}[1]
			\REQUIRE{Time-series data \(\mymat{Z}\) from \cref{eq:Zmat4CompDMD}.}
			\STATE{Partition \(\mymat{Z}\) as follows:}
 			\begin{equation}\label{eq:Z__equals__X__z_nplus1}
				\mymat{Z} ~=:~ 
				\begin{bmatrix}
					\mymat{X} &|& \myvec{z}_{\undelayedDMDorder+1}
				\end{bmatrix}.
			\end{equation}
			\STATE{Find the best approximation of \(\myvec{z}_{\undelayedDMDorder+1}\) in the span of \(\mymat{X}\).}
			\begin{equation}\label{eq:Optimal_1StepPredictor}
				\opt{\myvec{c}}[\mymat{Z}] ~:=~\mymat{X}^\dagger \myvec{z}_{\undelayedDMDorder+1}.
			\end{equation}
			\STATE{Form the Companion matrix associated to the DMD model \(\opt{\myvec{c}}[\mymat{Z}]\).}
			\begin{equation}\label{eq:Optimal__CompanionMatrix}
				\opt{\mymat{T}}~:=~ \mymat{T}\left(\opt{\myvec{c}}[\mymat{Z}]\right)
			\end{equation}
			\STATE{Compute the eigenvalues, \(\DMDEvals\), and eigen-vectors, \(\{\myvec{v}_i\}_{i=1}^\undelayedDMDorder\) of \(\opt{\mymat{T}}\). }
			\begin{equation}\label{eq:defn__DMD_Eigenvalues}
				\opt{\mymat{T}} \myvec{v}_i~=~\dummy{\lambda}_i \myvec{v}_i
			\end{equation}
			\STATE{Construct the DMD mode, \(\dummy{\myvec{d}_i}\), corresponding to each DMD eigenvalue, \(\dummy{\lambda}_i\).}
			\begin{displaymath}
				\dummy{\myvec{d}_i} ~:=~\mymat{X} \myvec{v}_i.
			\end{displaymath}
			\ENSURE{DMD eigen-values,  \(\DMDEvals\), and their modes, \(\DMDModes\).}
		\end{algorithmic}
	\end{algorithm}

	While this approximation has been studied in \cite{mezic2020numerical}, we work in a much simpler setting that permits the following guarantee purely from linear algebraic arguments:
	\begin{theorem}\label{thm:Suff4CompDMD2CatchAllNKnow}		
		Suppose \PsiLiesInNRSpanOfrDistinctKEFs~ and \ICBeSpectrallyInformative.
		
		If the matrix pair \((\mymat{X},\mymat{Y})\) is \hyperref[defn:Linear_Consistency]{linearly consistent} and there are at-least as many DMD eigenvalues \((\undelayedDMDorder)\) as there are Koopman eigenvalues \((\finitekissdim) \), then, the Koopman eigenvalues are a subset of the DMD eigen-values:
		\begin{displaymath}
			(\mymat{X},\mymat{Y})~\textrm{is linearly consistent}~\&~\undelayedDMDorder \geq \finitekissdim \implies \KEvals \subseteq \DMDEvals.
		\end{displaymath}
		Furthermore, a DMD mode is non-zero if and only if the associated eigenvalue lies in the Koopman spectrum.
		Specifically, if \(\myvec{v}_{\mu}\) is the eigenvector of \( \opt{\mymat{T}}\) corresponding to eigenvalue \(\mu\)\footnote{\(\mu\) is a scalar quantity and is not to be confused with the vectorial quantity \(\myvec{\mu}\) which is the temporal mean.},
		\begin{displaymath}
			\opt{\mymat{T}} \myvec{v}_{\mu}~=~\mu\,\myvec{v}_{\mu},
		\end{displaymath}
		then,
		\begin{displaymath}
			\mymat{X} \myvec{v}_{\mu} \neq \myvec{0} ~\iff~ \mu \in \KEvals.
		\end{displaymath}
	\end{theorem}
	A hyper-parameter of DMD that is fundamental to this and subsequent analyses is  the number of DMD eigen-values, \(\undelayedDMDorder\).
	It determines the complexity of the dynamics that can be captured by the DMD model \(\opt{\myvec{c}}[\mymat{Z}]\) - a property formally recognized as follows:	
	\begin{definition}[Order of a DMD model]\label{defn:Companion_order}
		Given the time-series data \(\mymat{Z} \in \mathbb{C}^{\dictionarylength \times (\undelayedDMDorder+1)}\), the order of the resulting DMD model is defined thus:
		\begin{equation}
			\textrm{Companion-order}[\mymat{Z}] ~:=~ n.
		\end{equation}
	\end{definition}

	\subsection{Mean subtracted Dynamic Mode Decomposition (\(\muDMD\))}\label{ss:prelims_effect_of_msub}
	Mean subtraction is a useful pre-processing step for DMD, when the time series, \(\mymat{Z}\), represents the solution to certain partial differential equations (PDEs) \cite{chen2012variants}. 
	In this amendment, elaborated here as \cref{alg:MeanSubtractedDMD}, one simply removes the temporal mean of \(\mymat{Z}\) from each of its' columns before performing DMD.
	\begin{algorithm}
		\caption{Mean-subtracted Dynamic Mode Decomposition (\(\muDMD\))}
		\label{alg:MeanSubtractedDMD}
		\begin{algorithmic}[1]
			\REQUIRE{Time-series data \(\mymat{Z}\) from \cref{eq:Zmat4CompDMD}.}
			\STATE{Compute the temporal mean, \(\myvec{\mu}\) of the time-series data \(\mymat{Z}\).}
			\begin{equation}\label{eq:Defn_Temporal_Mean}
				\myvec{\mu} := \frac{1}{n+1}\,  \sum_{i = 1}^{n+1} \myvec{z}_i .
			\end{equation}
			\STATE{Remove the temporal mean from each snapshot i.e. column of \(\mymat{Z}\).}
			\begin{equation}\label{eq:Define__Z_ms}
				\mymat{Z}_{\rm ms} := [\myvec{z}_j-\myvec{\mu}]_{j=1}^{n+1}.
			\end{equation}
			\STATE{Use the mean-subtracted data, \(\mymat{Z}_{\rm ms}\), as the input for DMD (\cref{alg:CompanionDMD}).}
			\begin{equation}\label{eq:MeanSubtractedDMD_EvalsModes}
					\msubDMDEvals,~\msubDMDmodes
				~~\longleftarrow~~
				\textrm{DMD}\left(\mymat{Z}_{\rm ms}\right)
			\end{equation}
			\ENSURE{Mean-subtracted DMD eigen-values, \(\msubDMDEvals\), and their modes, \(\msubDMDmodes\).}
		\end{algorithmic}
	\end{algorithm}
	Such a pre-processing enables Koopman eigen-functions with eigenvalue \(1\) to be \emph{exactly} forecast.
	\begin{definition}[Forecasting with \(\muDMD\)]
		Let \(\mymat{Z}_{\rm test}\) denote the time series generated by observing the trajectory beginning at \(\state_{\rm test}\) for \(\undelayedDMDorder\) time steps:
		\begin{equation}
			\mymat{Z}_{\rm test} ~:=~ \left[ \dictionary\left( \Gamma^{j-1}\left( \state_{\rm test} \right) \right)  \right]_{j=1}^{\undelayedDMDorder+1}.
		\end{equation}
		In other words, the \(j\)-th column of \(\mymat{Z}_{\rm test}\), denoted \( (\myvec{z}_{\rm test})_j\), is defined thus:
		\begin{equation}
			(\myvec{z}_{\rm test})_j ~:=~ \dictionary\left( \Gamma^{j-1}\left( \state_{\rm test} \right) \right).
		\end{equation}
		Additionally, let \(\myvec{\mu}_{\rm test}\) denote the temporal mean of \(\mymat{Z}_{\rm test}\) i.e.,
		\begin{equation}\label{eq:Defn__Test_Temporal_Mean}
			\myvec{\mu}_{\rm test} ~:=~ \frac{1}{n+1}\,  \sum_{j = 1}^{n+1} 
			(\myvec{z}_{\rm test})_j.
		\end{equation}
		Then, \(\muDMD\) forecasts the \(j\)-th column of \(\mymat{Z}_{\rm test}\) thus:
		\begin{equation}\label{eq:forecast_muDMD}
			(\reduced{\myvec{z}}_{\rm test})_j ~:=~ \\
			\begin{cases}
				(\myvec{z}_{\rm test})_j, \quad&j~=~1,\dots,\undelayedDMDorder.\\
				\\
				\myvec{\mu}_{\rm test} ~+~ 
				\begin{bmatrix}
					(\reduced{\myvec{z}}_{\rm test})_{j - \undelayedDMDorder + k} - \myvec{\mu}_{\rm test}
				\end{bmatrix}_{k=0}^{\undelayedDMDorder-1}~
				~
				\opt{\myvec{c}}[\mymat{Z}_{\rm ms}], \quad & j~>~\undelayedDMDorder.
			\end{cases}
		\end{equation}
		\end{definition}
	\begin{theorem}[\(\muDMD\) perfectly forecasts constant KEFs]\label{thm:muDMD_preserves_constant_KEFs}
	Suppose we have a vector of observables, \(\mymat{E}^H\dictionary\), that remains constant along every trajectory i.e.,
	\begin{equation}\label{eq:Abstracted_BCs}
			\forall~(\state_{\rm test},j),\quad \mymat{E}^H (\mymat{z}_{\rm test})_j ~=~ \mymat{E}^H (\mymat{z}_{\rm test})_1.
		\end{equation}
	Then, this property is respected by the \(\muDMD\) forecasts as well:
	\begin{equation}\label{eq:msub_forecasts__preserve__BCs}
			\forall~(\state_{\rm test},j), \quad \mymat{E}^H (\reduced{\myvec{z}}_{\rm test})_j ~=~ \mymat{E}^H (\reduced{\myvec{z}}_{\rm test})_1.
		\end{equation}		
	\end{theorem}
	\begin{proof}
		Refer \Cref{s:Proof__muDMD_preserves_KEFs_at_1}.
	\end{proof}
	Consequently, when the underlying dynamical system \cref{eq:DS4SDMD} is derived from a PDE, the associated boundary conditions are always satisfied by the \(\muDMD\) forecasts. 

   	Yet, this innocuous modification \emph{may} compromise DMD-based estimates of the Koopman spectrum.
	\begin{lemma}\label{lem:ChenSufficiency4c_ms}\cite{chen2012variants}
	Consider the analogue of \cref{eq:Z__equals__X__z_nplus1} for the mean-subtracted time series \(\mymat{Z}_{\rm ms}\):
	\begin{equation}\label{eq:Define__X_ms}
		\mymat{Z}_{\rm ms} ~=:~ 
		\begin{bmatrix}
			\mymat{X}_{\rm ms} &|& \myvec{z}_{\undelayedDMDorder+1} - \myvec{\mu}
		\end{bmatrix}.
	\end{equation}
    If \(\mymat{X}_{\rm ms}\) has linearly independent columns, then, the mean-subtracted DMD eigen-values are the \((n+1)\)-th roots of unity, modulo 1.
	\begin{displaymath}
		\mymat{X}_{\rm ms} ~\textrm{has full column rank}~~~\implies~~~
 		\msubDMDEvals ~~=~~ \{ z \neq 1~|~z^{\undelayedDMDorder+1}=1 \}.
	\end{displaymath}
	\end{lemma}
	There are three concerning aspects of this result:
	\begin{enumerate}
		\item \underline{Broad applicability}:
		The requirement of linearly independent columns in \(\mymat{X}_{\rm ms}\) is met when \(\mymat{Z}\) has full column rank, which is often the case when the snapshots are slices of the solution to a PDE.
		\item \underline{Content independence}:
		The mean-removed DMD eigenvalues are entirely determined by the parameter \(\undelayedDMDorder\) and, thus, have no regard for the information contained in \(\mymat{Z}_{\rm ms}\).
		\item \underline{Reduction to a Discrete Fourier Transform (DFT)}:
		The DMD eigen-values are restricted to be the \((n+1)\)th roots of unity, modulo 1.
        As a result, the mean-subtracted DMD modes, \( \msubDMDmodes \), coincide with the Fourier modes of the original time-series, \(\mymat{Z}\).
        More specifically, if \(\omega\) denotes the \((\undelayedDMDorder+1)\)-th root of unity,
        \begin{displaymath}
        	\omega ~:=~ \euler^{\ramuno\left(\frac{2 \pi}{\undelayedDMDorder+1}\right)},
        \end{displaymath}
    	and \(\reduced{\myvec{f}}_i\) represents the \(i\)-th Fourier mode of \(\mymat{Z}\),
    	\begin{displaymath}
    		\reduced{\myvec{f}}_i ~:=~  \frac{1}{\undelayedDMDorder+1}\,  \sum_{j = 1}^{\undelayedDMDorder+1} \myvec{z}_j\, \omega^{-i(j-1)},
    	\end{displaymath}
    	then, we have
    	\begin{displaymath}
    		\reduced{\lambda}_i ~=~\omega^i ~~\&~~\reduced{\myvec{d}_i} ~=~ \reduced{\myvec{f}}_i, \quad \forall\quad i~\in~\{1,\hdots,\undelayedDMDorder\}. 
    	\end{displaymath}
	\end{enumerate}
	The last observation above suggests that the restriction of \(\muDMD\) eigenvalues, which is described in \cref{lem:ChenSufficiency4c_ms}, could be formalized as follows:
	\begin{definition}[DMD-DFT equivalence]\label{def:DMD_DFT_Equivalence}
	We say that mean-subtracted DMD (\cref{alg:MeanSubtractedDMD}) is equivalent to DFT when the eigenvalues output by it coincide with the \((n+1)\)th roots of unity, modulo 1.
	 \begin{displaymath}
	   \muDMD \equiv \textrm{DFT}  \iff \msubDMDEvals ~~=~~ \{ z \neq 1~|~z^{\undelayedDMDorder+1}=1 \}.
	 \end{displaymath}
	\end{definition}
	Since the discovery of DMD-DFT equivalence, alternative variants of DMD have been developed to achieve the goal of mean subtraction (boundary-condition aware model reduction) without inducing DMD-DFT equivalence \cite{chen2012variants,hirsh2019centering}.
	Yet, theoretical investigations of this phenomenon, by itself, have been few and far between.
	We now focus on one such work \cite{hirsh2019centering} that stands out in its efforts to clarify DMD-DFT equivalence.
	
	\subsubsection{Towards untangling the equivalence of \(\muDMD\) to DFT}\label{sss:prelims_msub_state_of_art}
	Hirsch and co-workers \cite{hirsh2019centering} provide a clear exposition of the work by Chen et.al. \cite{chen2012variants} and derive sufficient conditions for \(\mymat{X}_{\rm ms}\) to be rank defective. 
	Understanding their contributions necessitates the Koopman mode decomposition of \(\dictionary\) be quantified (\cref{defn:NonRedundantSpanOfDistinctKEFs}).
	This facilitates comparing the system complexity with the amount of training data (\cref{defn:SamplingRegimes}) and grading the dynamical information contained in the trajectory starting at \(\state_1\) (\cref{defn:SpectrallyInformativeState}).
	Combining these formalisms with a notion of ``ideal'' training data (\cref{defn:Linear_Consistency}) underpins the result of Hirsch et.al. (\cref{thm:ChenSuffMayNotHold}) as well as the contributions in this work.
	
	Our primary assumption is that the Koopman mode expansion of \(\dictionary\) be finite.
	\begin{definition}[Dictionary non-redundantly spanned by distinct KEFs]\label{defn:NonRedundantSpanOfDistinctKEFs}
		The dictionary \(\dictionary\) is said to be in the non-redundant span of \(\finitekissdim\) distinct KEFs if there exists a collection of Koopman eigen-functions, \(\KEFs\), with distinct Koopman eigen-values, \(\KEvals\), such that 
		\begin{enumerate}
			\item Every component of \(\dictionary\) lies in the span of \(\KEFs\).
			\begin{equation}\label{eq:defn__Ctilde}
				\KEFVector ~:=~ [\phi_i]_{i=1}^\finitekissdim \implies \exists~ \tilde{\mymat{C}} \in \mathbb{C}^{\dictionarylength \times \finitekissdim}~\textrm{such that}~\dictionary~=~\tilde{\mymat{C}} \KEFVector.
			\end{equation}
			\item The expansion of \(\dictionary\) in terms of \(\KEFVector\) possesses no redundancies.
			\begin{equation}\label{eq:Every_column_in_CTilde_is_NonZero}
				\forall~i,\quad \tilde{\mymat{C}}\myvec{e}_i \neq \myvec{0}.
			\end{equation}
		\end{enumerate}
	\end{definition}
	When \(\dictionary\) is non-redundantly spanned by \(\finitekissdim\) distinct KEFs, capturing the \(\finitekissdim\) concomitant Koopman modes necessitates a minimum model order of \(\finitekissdim\).
	Now, by \cref{defn:Companion_order}, the model order, \(\undelayedDMDorder\), also quantifies the amount of training data available.
	Hence, the number of distinct KEFs, \(\finitekissdim\), can be used to score the usefulness of the snapshot matrix, \(\mymat{Z}\), in building a generalizable model.
	\begin{definition}[Sampling regimes in DMD]\label{defn:SamplingRegimes}
		Suppose \(\dictionary\) is non-redundantly spanned by \(\finitekissdim\) distinct KEFs. 
		Then, we may grade the (in)sufficiency of the training data, \(\mymat{Z}\), for imbibing the DMD model, \(\opt{\myvec{c}}\), with generalizability as follows:
 		 \begin{enumerate}
			\item Under-sampled \((n < r)\): 
			The order of the DMD model is inadequate for describing the \(\finitekissdim\)-term Koopman mode expansion.
			\item Well-sampled \( (\undelayedDMDorder \geq \finitekissdim) \):
			\begin{enumerate}
				\item Just-sampled \((n=r)\):
				The DMD model has the exact order needed to capture the \(\finitekissdim\) distinct Koopman eigen-values.        
				\item Over-sampled \((n \geq r + 1)\):
				The order of the DMD model is more than sufficient to represent all \(\finitekissdim\) Koopman eigen-values.
			\end{enumerate}
		\end{enumerate}
	\end{definition}
	Despite an abundance of training data, DMD can fail to capture all the underlying Koopman modes  if the initial condition, \(\state_1\), lies on the zero level set of any KEF \(\phi_i\).
	Such dynamically simple states are excluded by the condition of ``spectral informativeness''.
	\begin{definition}[Spectrally informative state]\label{defn:SpectrallyInformativeState}
	Suppose \(\dictionary\) is non-redundantly spanned by \(\finitekissdim\) distinct KEFs \(\KEFs\).
	Then, the state \(\state\) is said to be spectrally informative if it does not lie on the zero level set of any KEF.
	\begin{displaymath}
		\state~\textrm{is spectrally informative}~\iff~\forall~i,~\phi_i(\state)~\neq~0.
	\end{displaymath}
	\end{definition}
	Finally, we introduce a requirement on sub-matrices of \(\mymat{Z}\) that is necessary for the generalizability of DMD.
	\begin{definition}[Linear consistency \cite{tu2013dynamic}]\label{defn:Linear_Consistency}
		Let \(\mymat{Y}\) denote the sub-matrix of \(\mymat{Z}\) corresponding to the last \(\undelayedDMDorder\) columns:
		\begin{equation}\label{eq:Definition__Y}
			\mymat{Z} ~=:~ 
			\begin{bmatrix}
				\myvec{z}_1 & \mymat{Y}
			\end{bmatrix}.
		\end{equation}
		The matrix pair \((\mymat{X}, \mymat{Y})\) is said to be linearly consistent if there exists a linear transform \(\mymat{A}\) that maps \(\mymat{X}\) to \(\mymat{Y}\).
		\begin{displaymath}
			(\mymat{X},\mymat{Y}) \textrm{ is linearly consistent } \iff \exists ~\mymat{A}~\textrm{such that}~ \mymat{A}\mymat{X} ~=~\mymat{Y}.
		\end{displaymath}
	\end{definition}
	With this machinery in place, the contribution of \cite{hirsh2019centering} that forms the starting point of our work reads thus:
	\begin{theorem}\label{thm:ChenSuffMayNotHold}
		Suppose \(\dictionary\) lies in the non-redundant span of \(\finitekissdim\) KEFs with distinct eigen-values \(\KEvals\) and the initial condition \(\state_1\) is spectrally informative.
		Furthermore, assume that one of these eigen-values takes the value of 1, none of them is 0 and  the matrix \(\mymat{X}\) has non-zero mean:
		\begin{displaymath}
			1 \in \KEvals,~0 \notin \KEvals ~\&~\mymat{X} \mathbf{1}_n \neq \mathbf{0}_m.
		\end{displaymath}
		
		If the matrix pair \((\mymat{X},\mymat{Y})\) is linearly consistent and the input to DMD is well-sampled, then, the columns of \(\mymat{X}_{\rm ms}\) are not linearly independent.
		\begin{displaymath}
			(\mymat{X},\mymat{Y})~\textrm{is linearly consistent}~\&~\undelayedDMDorder \geq \finitekissdim \implies \mymat{X}_{\rm ms}~\textrm{doesn't have full column rank}.
		\end{displaymath}
	\end{theorem}
	Although this greatly improves our understanding, the story is still incomplete. 
	One cannot infer non-equivalence of mean-subtracted DMD and DFT, from the rank deficiency of \(\mymat{X}_{\rm ms}\).
	Specifically, \cref{thm:ChenSuffMayNotHold}, in conjunction with \cref{lem:ChenSufficiency4c_ms}, falls short of producing a necessary condition for \(\muDMD \equiv \textrm{DFT}  \).
		
	We address this shortcoming through a necessary and sufficient condition (\cref{thm:musubNDFT}) for DMD-DFT equivalence that is weaker than asking for linear independence of the columns of \(\mymat{X}_{\rm ms}\).
	Consequently, we can generate concrete examples\footnote{Certain well-characterized linear dynamical systems observed through a tailored dictionary using specific choices of the model order} where DMD-DFT equivalence is observed despite \(\mymat{X}_{\rm ms}\) possessing linearly dependent columns (\cref{thm:XmsnotLI_DMD_DFTequiv}).
		
	Nonetheless, a numerical experiment in \cref{thm:ChenSuffMayNotHold} indicates conditions under which DMD-DFT equivalence is precluded.
	Indeed, Figure 4a in \cite{hirsh2019centering} suggests that a spectrally informative initial condition, well-sampling, linear consistency and non-zero mean may prevent \(\muDMD \equiv \textrm{DFT}  \).
		
	We formally prove a refined version of this implicit conjecture in \cref{thm:DMDDFT_Oversampled}.

	
	\section{Exploring the what and where of DMD-DFT equivalence}\label{s:New_MsubDMD_DFT} 

    We begin by developing a necessary and sufficient condition for the equivalence of mean-subtracted DMD and DFT (\cref{thm:musubNDFT}).
    Building on this result, we find that \cref{lem:ChenSufficiency4c_ms} is a stricter sufficient condition for DMD-DFT equivalence and show that its converse does not hold (\cref{thm:XmsnotLI_DMD_DFTequiv}).
    Then, we study the relationship between mean-subtracted DMD and temporal DFT as a function of the DMD model order\footnote{See \cref{defn:Companion_order}}.
    We find that the occurrence of DMD-DFT equivalence is, for all practical purposes, dictated by the model order, \(\undelayedDMDorder\), and the number of Koopman modes, \(\finitekissdim\), comprising the DMD observables \(\dictionary\) (\cref{thm:DMDDFT_Oversampled,cor:DMDFT_Undersampled}).
    Finally, we demonstrate that the requirements for the above results -linearly consistent data, and to a lesser extent, a Koopman invariant \(\dictionary\) - can be met in practice using time delay embedding (\cref{prop:DelaysMakeGoodObservables}).

    \subsection{A geometric characterization of DMD-DFT equivalence}
	\begin{theorem}[A necessary and sufficient condition for DMD-DFT equivalence]\label{thm:musubNDFT}
		Mean subtracted DMD of \(\mymat{Z}\) is equivalent to temporal DFT if and only if the constant vector is orthogonal to \(\mathcal{N}(\mymat{X}_{\rm ms})\), that is, 
		\begin{displaymath}
		\muDMD \equiv \textrm{DFT}  \iff \mathcal{P}_{\mathcal{N}(\mymat{X}_{\rm ms})} \mathbf{1}_n = 0.
		\end{displaymath}
	\end{theorem}
	\begin{proof}
		Refer \Cref{ss:Appendix_Proof_NecessaryNSufficient_msubDMDDFTequivalence}
	\end{proof}
	Thus, DMD-DFT equivalence is ensured when the constant vector has no projection onto \(\mathcal{N}(\mymat{X}_{\rm ms})\) - A less stringent requirement than asking for a full column rank \(\mymat{X}_{\rm ms}\)\footnote{ \cref{lem:ChenSufficiency4c_ms} can be obtained from \cref{thm:musubNDFT} by noting that \(\mathcal{N}(\mymat{X}_{\rm ms}) = {0}\) when \(\mymat{X}_{\rm ms}\) has full column rank.}.
	This relaxation guides us towards situations where \(\muDMD \equiv \textrm{DFT} \), despite \(\mymat{X}_{\rm ms}\) having linearly dependent columns.
	\begin{theorem}[Counter-example that disproves the converse of \Cref{lem:ChenSufficiency4c_ms}]\label{thm:XmsnotLI_DMD_DFTequiv}
		Suppose we have \(\finitekissdim\) KEFs, \(\KEFs\), corresponding to distinct eigen-values \(\KEvals\).
		Let the initial condition \(\state_1\) be spectrally informative about these \(\finitekissdim\) KEFs.
		Furthermore, assume that none of these eigen-values takes the value of \(1\) and the model order \((\undelayedDMDorder)\) lies between \(2\) and \(r\):
		\begin{displaymath}
			1 \notin \KEvals ~~~\&~~~2 \leq \undelayedDMDorder\leq \finitekissdim.
		\end{displaymath}   
        
        Then, there is an explicitly constructible dictionary \(\dictionary\) such that the resulting matrix \(\mymat{X}_{\rm ms}\) possesses linearly dependent columns and, yet, mean-subtracted DMD is rendered equivalent to temporal DFT:
        \begin{displaymath}
        	\exists~\dictionary~\textrm{such that}~\mymat{X}_{\rm ms}~\textrm{doesn't have full column rank}~\&~\muDMD \equiv \textrm{DFT}.
        \end{displaymath}   
	\end{theorem}
	\begin{proof}
	Refer \Cref{ss:Appendix_Proof_Counterexamples_4_Hirsch_numerics_interpretation}
	\end{proof}
	Therefore, rank defectiveness of \(\mymat{X}_{\rm ms}\) does not preclude DMD-DFT equivalence for the time series \(\mymat{Z}\).
    It is worth noting that linear inconsistency of \((\mymat{X},\mymat{Y})\) and limited data are key ingredients in \cref{thm:XmsnotLI_DMD_DFTequiv}.
    Pursuing this line of thought, we find that the parameter regime where both conditions fail is devoid of DMD-DFT equivalence (\cref{thm:DMDDFT_Oversampled}).

    \subsection{Charting the domain of DMD-DFT equivalence}
    We probe the effect of Companion-order[\(\mymat{Z}\)] i.e. \(n\) on DMD-DFT equivalence for the time series \(\mymat{Z}\).
    We begin with the data-rich regime \((n >r)\) where linear consistency prevents DMD-DFT equivalence.
	\begin{theorem}\label{thm:DMDDFT_Oversampled}
		Suppose \(\dictionary\) lies in the non-redundant span of \(\finitekissdim\) distinct KEFs and the initial condition \(\state_1\) is spectrally informative.
    	
		If the matrix pair \((\mymat{X},\mymat{Y})\) is linearly consistent and the input to DMD is over-sampled, then, mean-subtracted DMD is not a temporal DFT.
		\begin{displaymath}
		(\mymat{X},\mymat{Y})~\textrm{is linearly consistent}~\&~ n \geq \finitekissdim + 1 \implies \muDMD \not\equiv \textrm{DFT} .
		\end{displaymath}
	\end{theorem}
	\begin{proof}
		Refer \Cref{ss:Appendix_Proof_None_If_Oversampled_and_LinCon}
	\end{proof}
    Thus, \cref{thm:ChenSuffMayNotHold} can be extended by mildly strengthening the sampling condition, retaining linear consistency and dropping all other requirements.
	The contrapositive is of practical interest: Equivalence of mean-subtracted DMD and DFT may indicate a need to acquire more data for a reliable analysis. 
	
	Moving onto the other data regimes \((\undelayedDMDorder \leq \finitekissdim)\), we see that the training set is limited and hence must be vigilant to avoid over-fitting.
	In particular, the notion of linear consistency, while numerically verifiable, may not lead to good generalization.
	This shortcoming can be addressed if every observable in the dictionary \(\dictionary\), when acted upon by the Koopman operator \koopman, continues to remain in the span of \(\dictionary\).
	\begin{definition}[Koopman invariant dictionary]\label{defn:Koopman_invariant_psi}
		Suppose we extend  \Cref{defn:Koopman_operator} in the obvious manner to operate on vector-valued observables i.e.,
		\begin{equation}\label{eq:Koopman__on__VectorObservables}
			\forall~i, \quad (U \circ \dictionary)_i ~:=~\psi_i \circ \Gamma.
		\end{equation}
		Then, Koopman invariance of \(\dictionary\) is defined thus:	\begin{displaymath}
		\dictionary \textrm{ is Koopman invariant } \iff \exists ~\mymat{A}~\textrm{such that}~ U \circ \dictionary ~=~ \mymat{A} \dictionary.
	\end{displaymath}
	\end{definition}
	In other words, Koopman invariance of \(\dictionary\) means linear consistency along any trajectory.

	Subsequently, we can draw upon \cite{chen2012variants} to complete the picture on DMD-DFT equivalence.
	\begin{corollary}\label{cor:DMDFT_Undersampled}
		Suppose \(\dictionary\) lies in the non-redundant span of \(\finitekissdim\) KEFs with distinct eigen-values \(\KEvals\) and the initial condition \(\state_1\) is spectrally informative.
		
		If \(\dictionary\) is also Koopman invariant, then, the following hold in the just and under-sampled regimes:
		\begin{enumerate}
            \item When \(n=r\), mean-subtracted DMD is equivalent to temporal DFT if and only if \(1\) is not a Koopman eigenvalue.
            \begin{displaymath}
                \textrm{When}~n = r,~\muDMD \equiv \textrm{DFT}   \iff  1 \notin \KEvals.
            \end{displaymath}
			\item When \(n < r\) i.e. the system is under-sampled, mean-subtracted DMD is equivalent to temporal DFT.
            \begin{displaymath}
                n < r \implies \muDMD \equiv \textrm{DFT} . 
            \end{displaymath}
		\end{enumerate}
	\end{corollary}
	\begin{proof}
		Refer \Cref{ss:Appendix_Proof_Under_Just_Sampled}
	\end{proof}

     Thus, we have a complete understanding of DMD-DFT equivalence, as defined in \cref{def:DMD_DFT_Equivalence}, when the observables are Koopman invariant. 
     Barring the pragmatically negligible case of just-sampling (\(n=r\)), under-sampling is necessary and sufficient for the equivalence of DMD and DFT. 

 \subsection{DMD-DFT equivalence and delay embedding}
     Linear consistency of \((\mymat{X},\mymat{Y})\) and, to a smaller extent, Koopman invariance of  \(\dictionary\) have been vital to our study of DMD-DFT equivalence.
     While these properties may not always hold, a simple remedy\footnote{Assuming one has enough snapshots at hand.} is to delay embed the training data \(\mymat{Z}\) before performing \(\muDMD\) (\cref{prop:DelaysMakeGoodObservables}).
     This strategy of prefacing \cref{alg:MeanSubtractedDMD} by taking time-delays, dubbed \(\delayedmuDMD\) (\cref{alg:Delay__MeanRemove__DMD}), permits \cref{thm:DMDDFT_Oversampled,cor:DMDFT_Undersampled} to be rephrased without drawing upon linear consistency or Koopman invariance (\cref{cor:DMD_DFT_Practical}).
\subsubsection{Koopman invariance via time delays}
	\begin{proposition}\label{prop:DelaysMakeGoodObservables}
	Suppose \PsiLiesInNRSpanOfrDistinctKEFs.
    Consider augmenting \(\dictionary\) by taking \(d\) time delays:
 	\begin{displaymath}
    	\dictionary_{\rm d-delayed} ~:=~
    	\begin{bmatrix}
    		\dictionary \\ U \circ \dictionary \\ \vdots \\ U^d \circ \dictionary
    	\end{bmatrix}.
    \end{displaymath}
    If at least \(\finitekissdim-1\) time delays are taken, then, the augmented dictionary is Koopman invariant.
	\begin{displaymath}
	d \geq r-1~ \implies ~ \dictionary_{\rm d-delayed} ~\textrm{ is Koopman invariant}.
	\end{displaymath}
	\end{proposition}

	\begin{proof}
		Refer \Cref{ss:TimeDelays_Gib_KoopmanInvariance}.	
	\end{proof}
	\begin{remark}
  		Although \(r-1\) delays are suggested, \(\dictionary\) may generate a Koopman invariant dictionary with fewer delays.
		Specifically, vector valued observables (i.e., \(m > 1 \) ) may attain Koopman invariance for \( d< r-1\), as described in Lemma C of \cite{le2017higher} and refined in Theorem 2 of \cite{pan2020structure}.
		In contrast, a scalar valued observable needs an absolute minimum of \(r-1\) delays.
		In the forthcoming analysis, for the sake of simplicity, we use the potentially conservative prescription in \cref{prop:DelaysMakeGoodObservables} to ensure Koopman invariance.
	\end{remark}

\subsubsection{DMD-DFT equivalence under delay embedding}
	Consider the variant of \(\muDMD\) that results from delay-embedding the observables \(\dictionary\) before performing \cref{alg:MeanSubtractedDMD}- a  process that is detailed as \cref{alg:Delay__MeanRemove__DMD}.
	\begin{algorithm}
		\caption{Delay-embedded \(\muDMD\) (\(\delayedmuDMD\))}
		\label{alg:Delay__MeanRemove__DMD}
		\begin{algorithmic}[1]
			\REQUIRE{Time-series data \(\mymat{Z}\) from \cref{eq:Zmat4CompDMD} and the number of time-delays \(\fundelayscount\).}
			\STATE{Construct the delay-embedded matrix \(\mymat{Z}_{\rm d-delayed}\) by taking \(d\) time-delays of \(\mymat{Z}\).}
			\begin{equation}\label{eq:Define__Z_d_delayed}
				\mymat{Z}_{\rm d-delayed} := \begin{bmatrix}
					\myvec{z}_1 & \myvec{z}_2 & \dots  & \myvec{z}_{n+1-d}\\
					\myvec{z}_2 & \myvec{z}_3 & \dots  & \myvec{z}_{n+2-d}\\
					\vdots &\vdots  & \ddots  & \vdots \\
					\myvec{z}_{1+d} & \myvec{z}_{2+d} & \dots  & \myvec{z}_{n+1}\\
				\end{bmatrix}.
			\end{equation}
			\STATE{Compute the associated model order,}
   			\begin{equation}\label{eq:theta_order_of_Z_d_delayed}
				\theta~:=~\textrm{Companion-order}[\mymat{Z}_{\rm d-delayed}].
			\end{equation}
			\STATE{Use the delay-embedded data, \(\mymat{Z}_{\rm d-delayed}\), as input for \cref{alg:MeanSubtractedDMD}.}
			\begin{equation}
				\delayedmsubDMDEvals,~\delayedmsubDMDmodes
				~~\longleftarrow~~
				\muDMD\left(\mymat{Z}_{\rm d-delayed}\right)
			\end{equation}
			\ENSURE{\(\delayedmuDMD\) eigen-values, \(\delayedmsubDMDEvals\), and their modes, \(\delayedmsubDMDmodes\).}
		\end{algorithmic}
	\end{algorithm}
    Porting our findings on DMD-DFT equivalence to this strategy requires that we extend the notions of under, just and over-sampling to \(\mymat{Z}_{\rm d-delayed}\).
    The hyper-parameter, \(\undelayedDMDorder\), which determines the sampling regime in \cref{alg:CompanionDMD,alg:MeanSubtractedDMD}, also happens to be the associated model order (\cref{defn:Companion_order}).
    This suggests that generalizing the sampling regimes to \cref{alg:Delay__MeanRemove__DMD} should build upon its model order, \(\theta\)\footnote{Which, by \cref{defn:Companion_order,eq:Define__Z_d_delayed}, has a value of \(n-d\).}.
    In particular, under-sampling would translate to \(\theta < r\), just-sampling would mean \(\theta = r\) and over-sampling would denote \(\theta > r\).
	Hence, we can clarify the notion of DMD-DFT equivalence for \(\delayedmuDMD\):
	\begin{definition}\label{def:dmuDMD__equiv_DFT}
		We say that delay-embedded \(\muDMD\) (\cref{alg:Delay__MeanRemove__DMD}) is equivalent to DFT when the eigenvalues output by it coincide with the \((\theta+1)\)th roots of unity, modulo 1.
		\begin{displaymath}
			\delayedmuDMD \equiv \textrm{DFT}  \iff \delayedmsubDMDEvals ~~=~~ \{ z \neq 1~|~z^{\theta+1}=1 \}.
		\end{displaymath}
	\end{definition}
	Consequently, the more pragmatic version of \cref{thm:DMDDFT_Oversampled} and \cref{cor:DMDFT_Undersampled},  where time delays are used to meet the relatively abstract conditions of linear consistency and Koopman invariance, reads thus:
    \begin{corollary}\label{cor:DMD_DFT_Practical}
        Suppose \(\dictionary\) lies in the non-redundant span of \(\finitekissdim\) KEFs with distinct eigen-values \(\KEvals\) and the initial condition \(\state_1\) is spectrally informative.
        
        If at least \(r-1\) time delays have been taken, i.e. \(d \geq r-1\), then, \(\delayedmuDMD\) reduces to a DFT if and only if we are under-sampled  or in the just-sampled regime without a Koopman eigenvalue at 1.
    \begin{displaymath}
        \delayedmuDMD \equiv \textrm{DFT} \iff (\theta~<~r)~\textrm{OR}~(\theta~=~r~\&~1 \not \in \KEvals).
    \end{displaymath}
    \end{corollary}

	To summarize, when Koopman invariant observables are used, oversampling is (practically) necessary and sufficient to preclude the equivalence of mean-subtracted DMD and DFT. 
	The requisite Koopman invariance can be attained by delay embedding the snapshots.

    \section{Numerical experiments}\label{s:Numerics}
    The guarantees on DMD-DFT equivalence, in \cref{cor:DMD_DFT_Practical}, can be computationally illustrated and tested\footnote{All code used in this analysis can found at \url{https://github.com/gowtham-ss-ragavan/msub_mdselect_dmd.git}}, despite being abstract and fragile.
    The former drawback can be addressed by translating the mathematical statements into concrete computational tasks, and the latter by testing the translated guarantees on an ensemble of trajectories (\Cref{ss:Numerics__Preamble}).
    The resulting framework for numerically probing \cref{cor:DMD_DFT_Practical} is first deployed in scenarios where all the attendant conditions are known to be met (\Cref{ss:Systems_Known_ISS}).
    Subsequently, we repeat the same experiments in situations where the requisites of \cref{cor:DMD_DFT_Practical} may not be met, and draw upon its' contrapositive to generate lower bounds on the number of Koopman modes required to represent the observations (\Cref{ss:Systems_Unknown_ISS}).
    
    \subsection{Preamble: Translation and trial-by-ensemble}\label{ss:Numerics__Preamble}
    \subsubsection{A computation-friendly reformulation of \cref{cor:DMD_DFT_Practical}}
    Validating \cref{cor:DMD_DFT_Practical} requires that we assess its' predictions about the presence or absence of DMD-DFT equivalence, over an appropriate range of model orders.
    Alas, by \Cref{def:dmuDMD__equiv_DFT}, the associated predictions take the form of set (in)equalities- properties that cannot be reliably discerned when the underlying computations use finite precision arithmetic.
    
    So, we re-phrase those predictions using \cref{thm:musubNDFT} in terms of a real-valued and numerically tangible indicator named \(\DistanceToDFT\).
    We begin by rephrasing \cref{thm:musubNDFT} in the context of \(\delayedmuDMD\).
    \begin{lemma}[Recasting  \(\delayedmuDMD \equiv \textrm{DFT} \) as a vector equality]\label{lem:delayedmuDMD_DFT__VectorEquality}
	    Let \(\myvec{\mu}_{\rm d-delayed}\) denote the temporal mean of \(\mymat{Z}_{\rm d-delayed}\),
    	\begin{displaymath}
    		\myvec{\mu}_{\rm d-delayed} := \mymat{Z}_{\rm d-delayed} \frac{\myvec{1}_{\theta+1}}{ \theta+1},
    	\end{displaymath}
    	and \((\mymat{Z}_{\rm d-delayed})_{\rm ms}\) represent the mean-subtracted version of \(\mymat{Z}_{\rm d-delayed}\):
    	\begin{equation}\label{eq:Z_d_delayed_ms_Defn}
    		(\mymat{Z}_{\rm d-delayed})_{\rm ms} := \mymat{Z}_{\rm d-delayed} - \myvec{\mu}_{\rm d-delayed} \myvec{1}_{\theta+1}^T.
    	\end{equation}
 		Additionally, let \((\mymat{X}_{\rm d-delayed})_{\rm ms}\) be the sub-matrix corresponding to the first \(\theta\) columns of \((\mymat{Z}_{\rm d-delayed})_{\rm ms}\):
   		\begin{displaymath}
   			(\mymat{Z}_{\rm d-delayed})_{\rm ms} =: 
   			\begin{bmatrix}
   				(\mymat{X}_{\rm d-delayed})_{\rm ms} &|&
   				(\mymat{Z}_{\rm d-delayed})_{\rm ms}\, \myvec{e}_{\theta+1}
   			\end{bmatrix}.
   		\end{displaymath} 
    	
    	Then, a necessary and sufficient condition for delay-embedded \(\muDMD\) to be equivalent to a temporal DFT is that the constant vector, \(\myvec{1}_{\theta}\), be orthogonal to the null-space of \((\mymat{X}_{\rm d-delayed})_{\rm ms}\):
    	\begin{displaymath}
    		\delayedmuDMD \equiv \textrm{DFT}
    		~~\iff~~
    		\mathcal{P}_{\mathcal{N}\left(\,(\mymat{X}_{\rm d-delayed})_{\rm ms}  \,\right)}\left(\myvec{1}_{\theta}\right) ~=~ \myvec{0}.
    	\end{displaymath}
    \end{lemma}
    As such, the projection of \(\myvec{1}_{\theta}\) onto the null-space of \((\mymat{X}_{\rm d-delayed})_{\rm ms}\) can measure the deviation of \(\delayedmuDMD\) from DFT:
    \begin{definition}[\DistanceToDFT]\label{defn:RelativeDistance_to_DFT}
    	\begin{equation}\label{eq:Numcheck4DMD_DFt_Equivalence}
    		\DistanceToDFT\,[\theta, d] := \frac{\lVert \mathcal{P}_{\mathcal{N}\left(\,(\mymat{X}_{\rm d-delayed})_{\rm ms}  \,\right)}\left(\myvec{1}_{\theta}\right) \rVert}{\sqrt{\theta}}.
    	\end{equation}
    \end{definition}
    Therefore, \(\DistanceToDFT\) takes values on the interval \([0,1]\), with \(0\) indicating DMD-DFT equivalence and any other value in \((0,1]\) representing non-equivalence.
    
    Consequently, we may recast \cref{cor:DMD_DFT_Practical} thus:
    \begin{corollary}\label{cor:DMD_DFT_Numerical}
    	Suppose \(\dictionary\) lies in the non-redundant span of \(\finitekissdim\) KEFs with distinct eigen-values \(\KEvals\) and the initial condition \(\state_1\) is spectrally informative.
    	
    	If at least \(r-1\) time delays have been taken, i.e. \(d \geq r-1\), then, the indicator \(\DistanceToDFT\) is completely determined by the sampling regime and spectral content, as detailed in \cref{tab:DMDInputs_to_RelativeDistance}.
    	\begin{table}[tbhp]\label{tab:DMDInputs_to_RelativeDistance}
    		\caption{\(\DistanceToDFT\) as a function of the sampling regime and the presence or absence of a Koopman eigen-value at \(1\).}
    		\begin{center}
    			\begin{tabular}{|c|c|c|} \hline
    				Sampling regime & Spectral condition & \(\DistanceToDFT\)\\ \hline
    				\(\theta < \finitekissdim\)& None  & \(0\) \\\hline
    				\(\theta = \finitekissdim\)& \(1 \not \in \KEvals\)  & \(0\) \\
    				& \(1 \in \KEvals\)  & \((0,1]\) \\\hline
    				\(\theta > \finitekissdim\)& None  & \((0,1]\) \\\hline
    			\end{tabular}
    		\end{center}
    	\end{table}
    \end{corollary}

   	\begin{remark}\label{rem:Only_computing_an_upper_bound}
		In computing \DistanceToDFT, we do not use the matrix \((\mymat{X}_{\rm d-delayed})_{\rm ms} \) directly as it becomes ill-conditioned for large values of \(d\).
		
		Instead, we work with \((\dummy{\mymat{X}}_{\rm d-delayed})_{\rm ms} \) which is the low rank approximation of \((\mymat{X}_{\rm d-delayed})_{\rm ms} \) obtained by retaining as many leading singular vectors as possible without the condition number of the approximant exceeding \(10^{8}\).
		Specifically, if we denote the best \(s\)-rank approximation of \((\mymat{X}_{\rm d-delayed})_{\rm ms}\) as \(\lowrankapprox{(\mymat{X}_{\rm d-delayed})_{\rm ms}}_s\), then, \((\dummy{\mymat{X}}_{\rm d-delayed})_{\rm ms} \) is defined as follows:
		\begin{equation}
			\begin{aligned}
			\opt{s}~~~~~~~~~~&:=~\max \left\{s~|~\textrm{Condition number}\left( \lowrankapprox{(\mymat{X}_{\rm d-delayed})_{\rm ms}}_s \right) \leq 10^{8} \right\} \\
			(\dummy{\mymat{X}}_{\rm d-delayed})_{\rm ms} ~&:=~ \lowrankapprox{(\mymat{X}_{\rm d-delayed})_{\rm ms}}_{\opt{s}}.
			\end{aligned}
		\end{equation}
		Here, we cap the condition number of the approximation at \(10^{8}\) to ensure that when the inputs of \(\delayedmuDMD\) are defined up to IEEE double precision of \(16\) digits, the \(\delayedmuDMD\) model, \(\opt{\myvec{c}}\left[(\mymat{Z}_{\rm d-delayed})_{\rm ms} \right]\), has at least \(8\) digits of precision \cite{trefethen2022numerical,corless2013graduate}.
		
		Hence, whenever \(\left\lVert (\mymat{X}_{\rm d-delayed})_{\rm ms}  - (\dummy{\mymat{X}}_{\rm d-delayed})_{\rm ms} \right \rVert\) is non-trivial, we end up computing only an upper bound on \DistanceToDFT.
	\end{remark}
    
    \subsubsection{Quantifying a theoretical loophole with ensemble studies}\label{sss:Ensemble_Studies}
    A non-zero value of \(\DistanceToDFT\) is useful only if it can be numerically distinguished from \(0\).
   
    Even if we do see this happen for a given trajectory, such an observation does not preclude the possibility of poor discernibility for other trajectories.
   
    Consequently, we systematically test \cref{cor:DMD_DFT_Numerical} for an ensemble of trajectories.
    For any choice of a dynamical system \cref{eq:DS4SDMD}, we pick an appropriate range of model-orders and delay embedding dimensions where \cref{cor:DMD_DFT_Numerical} is to be examined.
    For each choice of model-order \((\theta)\) and delay embedding dimension \((d)\), we compute \(\DistanceToDFT\) over an ensemble of trajectories and dictionaries.
    
    To see this in detail, suppose \(\theta_{\max}\) is the maximal model-order and \(d_{\max}\) the largest number of time delays to be taken.
    Consider the trajectory of length \(\theta_{\max} + d_{\max} + 1\) beginning at \(\state_1\) and observed through the vector of \(m\) observables \(\dictionary\).
    The length of the trajectory has been chosen so that we can compute \DistanceToDFT~ for every pair \((\theta,d)\) in the targeted parameter range.
    For the choice of \((\theta,d)\), the time series matrix \(\mymat{Z}\) reads thus:
    \begin{equation}\label{eq:Z_for_order_theta_delays_d}
     \mymat{Z} = \begin{bmatrix}
             \dictionary(\state_1) & \dictionary(\Gamma(\state_1)) & \dots  &  \dictionary(\Gamma^{\theta+d}(\state_1))
            \end{bmatrix}.
    \end{equation}
    To paraphrase, \(\mymat{Z}\) corresponds to sampling \(\dictionary\) on the first \(\theta + d+  1\) elements of the trajectory.
    Subsequently, we form \(\mymat{Z}_{\rm d-delayed}\) and compute \(\DistanceToDFT\).
    This calculation is, then, iterated over every pair \((\theta,d)\) in the parameter regime of interest.
    Finally, we also repeat this parametric sweep, at a higher level, over many trajectories\footnote{We also switch the dictionary for each trajectory for an added level of randomness.}.
    Consequently, we can study the statistics of \(\DistanceToDFT\) for every pair of model-order (\(\theta\)) and  delay-embedding dimension \((\fundelayscount) \) being tested.
    This is facilitated by box-whisker plots of this indicator with respect to model-order. 
    
    We test the major results on a variety of dynamical systems of increasing complexity.
    We primarily work with linear-time invariant (LTI) systems
    since their Koopman spectra are easily determined and the conditions required in \cref{cor:DMD_DFT_Numerical} can be met.
    We also validate our theory using the Van der Pol oscillator and the lid-driven cavity from \cite{arbabi2019spectral}. 
    Here, in contrast to the LTI systems, we have no guarantee of meeting the requisites in \cref{cor:DMD_DFT_Numerical} nor do we possess complete knowledge of the associated Koopman spectra.
    Nonetheless, we can still use the contrapositive of \cref{cor:DMD_DFT_Numerical} to get a lower bound on the complexity of the underlying Koopman mode expansion.

    \subsection{Systems satisfying the requisites for \cref{cor:DMD_DFT_Numerical}}\label{ss:Systems_Known_ISS}
    Consider the linear time-invariant dynamical system governed by the following update rule:
    \begin{equation}\label{eq:LTI1ab_LTI3}
    	\Gamma(\state) ~:=~
    	\begin{bmatrix}
    		\lambda_1 & & & \\
    		& \lambda_2 & & \\
    		& & \ddots & \\
    		& & & \lambda_r
    	\end{bmatrix} \state.
    \end{equation}
	We set \(r\) to be \(7\) and study three choices of \(\KEvals\), denoted \(\lti_{1a}\), \(\lti_{1b}\) and \(\lti_3\), that are detailed in \cref{tab:LTI_Eval_List}.
	\begin{table}[tbhp]\label{tab:LTI_Eval_List}
		\caption{LTI systems from \cref{eq:LTI1ab_LTI3} that are used to illustrate \cref{cor:DMD_DFT_Numerical}.}
		\begin{center}
			\begin{tabular}{|c|c|c|c|} \hline
				Name & \(\lti_{1a} \) & \(\lti_{1b} \)& \(\lti_{3} \)\\ \hline
				Eigenvalues&  \(|\lambda_i| = 1,~\forall\,i\) & \(|\lambda_i| = 1\neq \lambda_i,~\forall\,i\)  & \(|\lambda_i| = 1 \neq \lambda_i,~\forall\,i~\in~\{1,2,3,4\}\)\\
				&\(\lambda_1=1 \) &  & \(|\lambda_5|,~|\lambda_6|,~|\lambda_7|~\in~[0.8,~0.9]\) \\\hline
			\end{tabular}
		\end{center}
	\end{table}
    For each of these three systems, the Koopman eigenvalues are within the unit disc, but \emph{only \(\lti_{1a}\) possesses a Koopman eigenvalue at 1}.
    
    The dynamics specified in \cref{eq:LTI1ab_LTI3} implies that each individual coordinate function, \(\upsilon_i\), is a Koopman eigen-function with eigen-value \(\lambda_i\).
    Hence, we have,
    \begin{displaymath}
    	\KEFVector(\state)~=~\state.
    \end{displaymath}
	The observable used to study \cref{eq:LTI1ab_LTI3} is simply a scalar function encoded by an arbitrarily chosen row vector, \(\tilde{\mymat{C}}\), with non-zero entries:
	\begin{displaymath}
		\dictionary~:=~\tilde{\mymat{C}}\KEFVector.
	\end{displaymath}
	Hence, the dictionary \(\dictionary\) is non-redundantly spanned by \(\finitekissdim\) distinct KEFs.
	Furthermore, we ensure that our choice of initial condition, \(\state_1\), is always spectrally informative.
	
    Assuming that at-least \(6\) time delays are taken, \cref{cor:DMD_DFT_Numerical} makes the following predictions:
    \(\lti_{1a}\) has a Koopman eigenvalue at 1. 
    So, \(\DistanceToDFT\) will be insignificant when  \(\theta < 7\) and otherwise for \(\theta \geq 7 \).
    The same trend should hold for \(\lti_{1b}\) and \(\lti_3\), with the boundary shifting from 7 to 8 due to the absence of a Koopman eigenvalue at \(1\).

	For the unitary systems \(\lti_{1a}\) and \(\lti_{1b}\), the expected trend is clearly visible in sub-figures (a) and (b) of \cref{fig:cms_transit_knownss}.
    When \(\theta \) is greater than or equal to \(7 ~(8)\) for \(\lti_{1a}~(\lti_{1b})\), the values of \(\DistanceToDFT\) are many orders of magnitude larger than those seen for smaller \(\theta\).
    Moreover, the shadowing of this feature by the pertinent inter-quartile ranges (IQRs), which are denoted by boxes colored according to the number of delays taken, suggests that the observation is independent of the choice of initial condition.
    
	The icing on the cake is that when \(\theta = r = 7\), the effect of having \(1 \in \KEvals \) is in concordance with \cref{cor:DMD_DFT_Numerical}.
    The system \(\lti_{1b}\), which does not have a Koopman eigenvalue at 1, has low values of \(\DistanceToDFT\) indicating DMD-DFT equivalence.
    To the contrary, \(\lti_{1a}\) exhibits a value that is orders of magnitude larger due to the presence of a Koopman eigenvalue at 1.

	In contrast, the computations in \(\lti_{3}\) exhibit a prominent dependence on the delay embedding dimension.
    Recall that \cref{cor:DMD_DFT_Numerical} predicts the same trend for \(\lti_3\) and \(\lti_{1b}\) when \emph{at-least} \(6\) time delays are taken.
    In sub-figure (c) of \cref{fig:cms_transit_knownss}, we observe that this is indeed true when only \(6\) time delays are taken.
    However, when we increase the number of time delays to \(24\), the spike in \(\DistanceToDFT\) at \(\theta = 8\) is mollified to the extent that it invites skepticism on the theoretical predictions.

    This highlights the computational inadequacies of the forecasts given by \cref{cor:DMD_DFT_Numerical}. 
    For further clarity, we need only formalize the criterion that we implicitly used to ascertain the validity of \cref{cor:DMD_DFT_Numerical} in \(\lti_{1a}\) and \(\lti_{1b}\).
    Checking the predictions of \cref{cor:DMD_DFT_Numerical} necessitates deciding what it means for \DistanceToDFT~ to be ``zero''.
    Specifically, we must select an \(\epsilon \in (0,1]\) such that values of \DistanceToDFT~ smaller than \(\epsilon\) can be considered ``zero''.
    For the unitary systems \(\lti_{1a}\) and \(\lti_{1b}\), we can choose any number between \(10^{-14}\) and \(10^{-10}\) as \(\epsilon\), regardless of the delay embedding dimension \(\fundelayscount\).
    However, in \(\lti_{3}\), we can only pick \(\epsilon\) in a narrow band around \(10^{-10}\).
    Hence, the validity of \cref{cor:DMD_DFT_Numerical} appears ``forced'' in \(\lti_{3}\) when compared to \(\lti_{1a}\) and \(\lti_{1b}\).
	\begin{figure}[]
		\centering
		\subfloat[\(\lti_{1a}\)]{\includegraphics[width = 0.7 \linewidth]{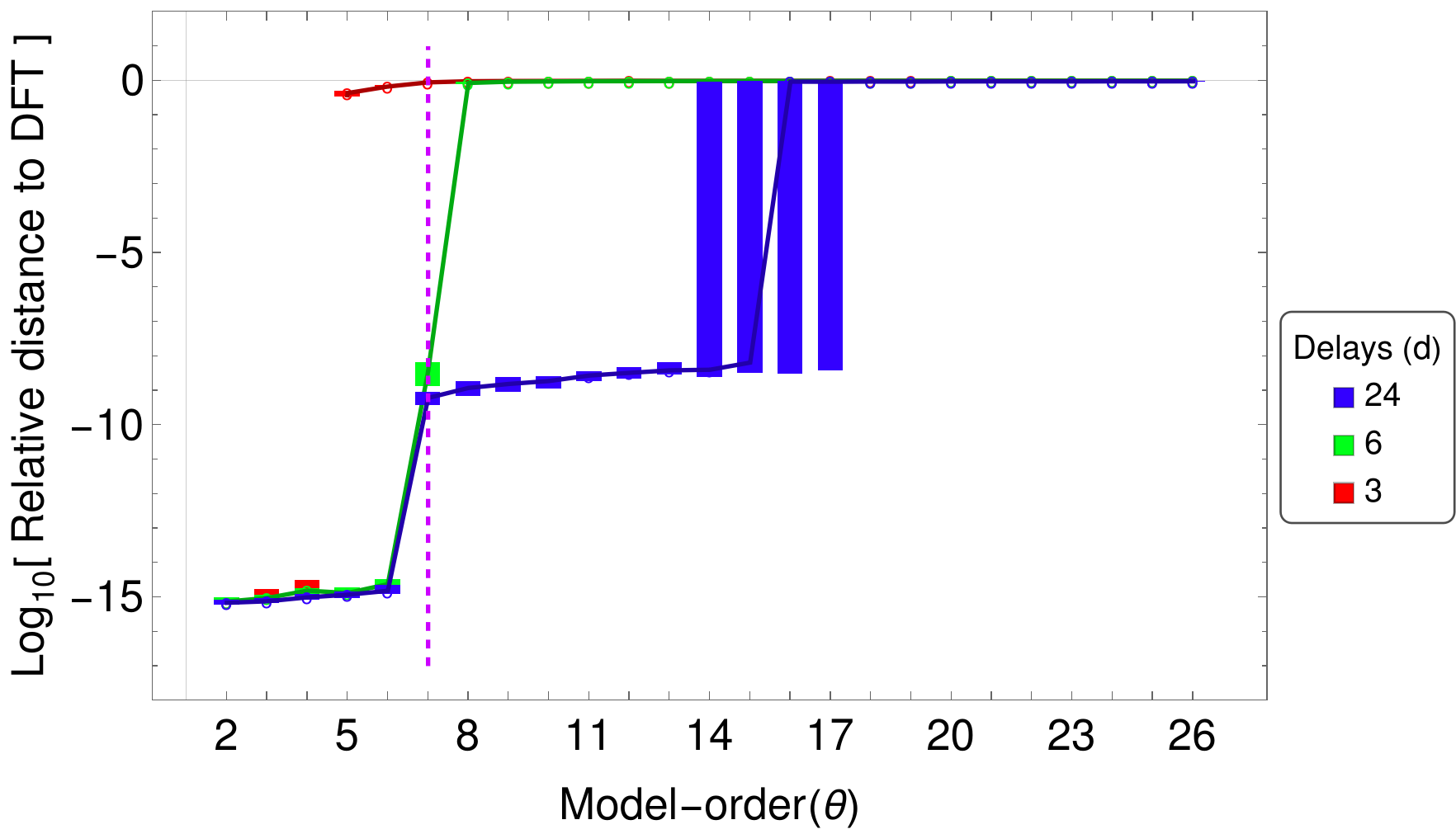}} \\
		\subfloat[\(\lti_{1b}\)]{\includegraphics[width = 0.7 \linewidth]{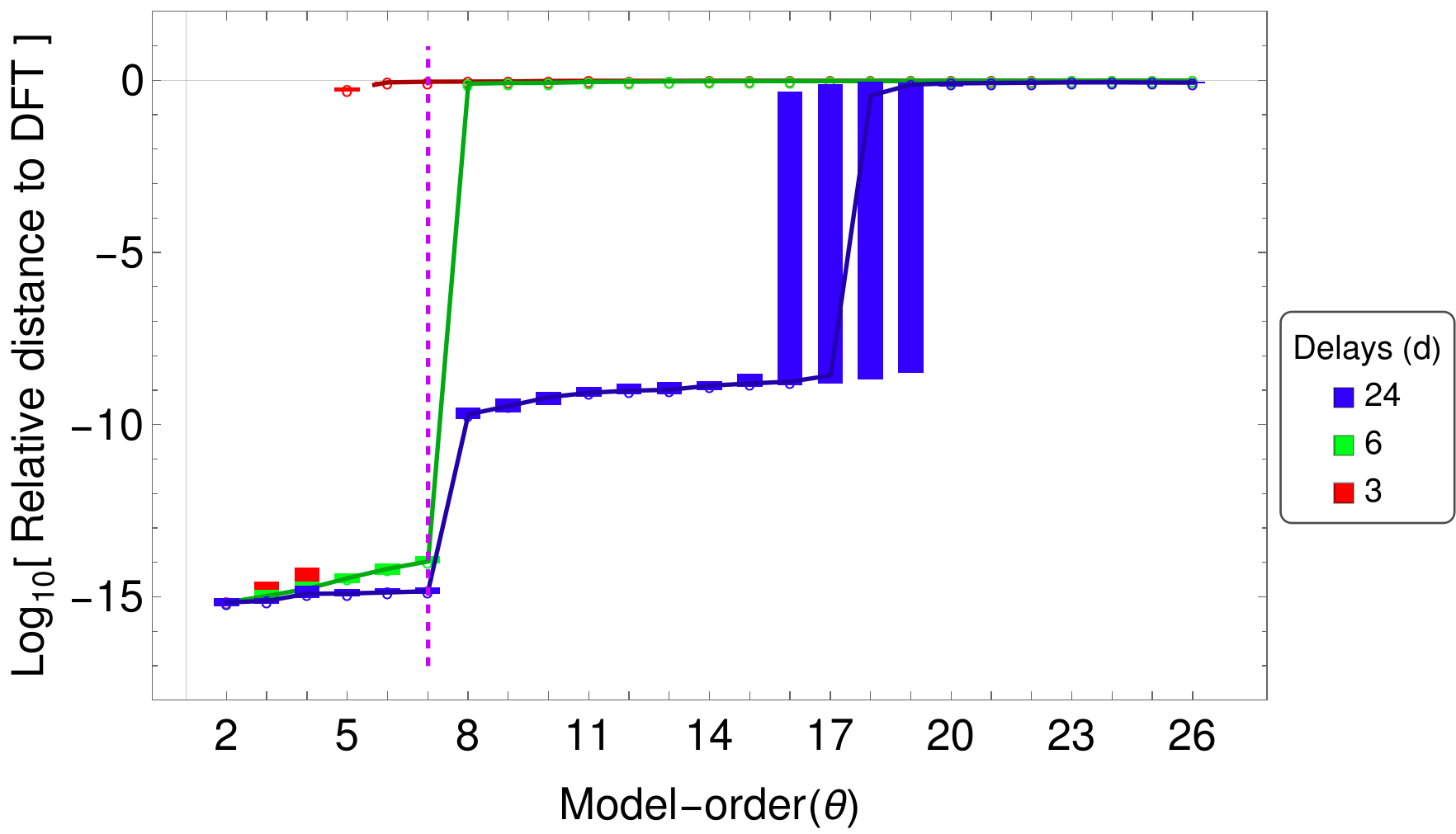}} \\
		\subfloat[\(\lti_{3}\)]{\includegraphics[width = 0.7 \linewidth]{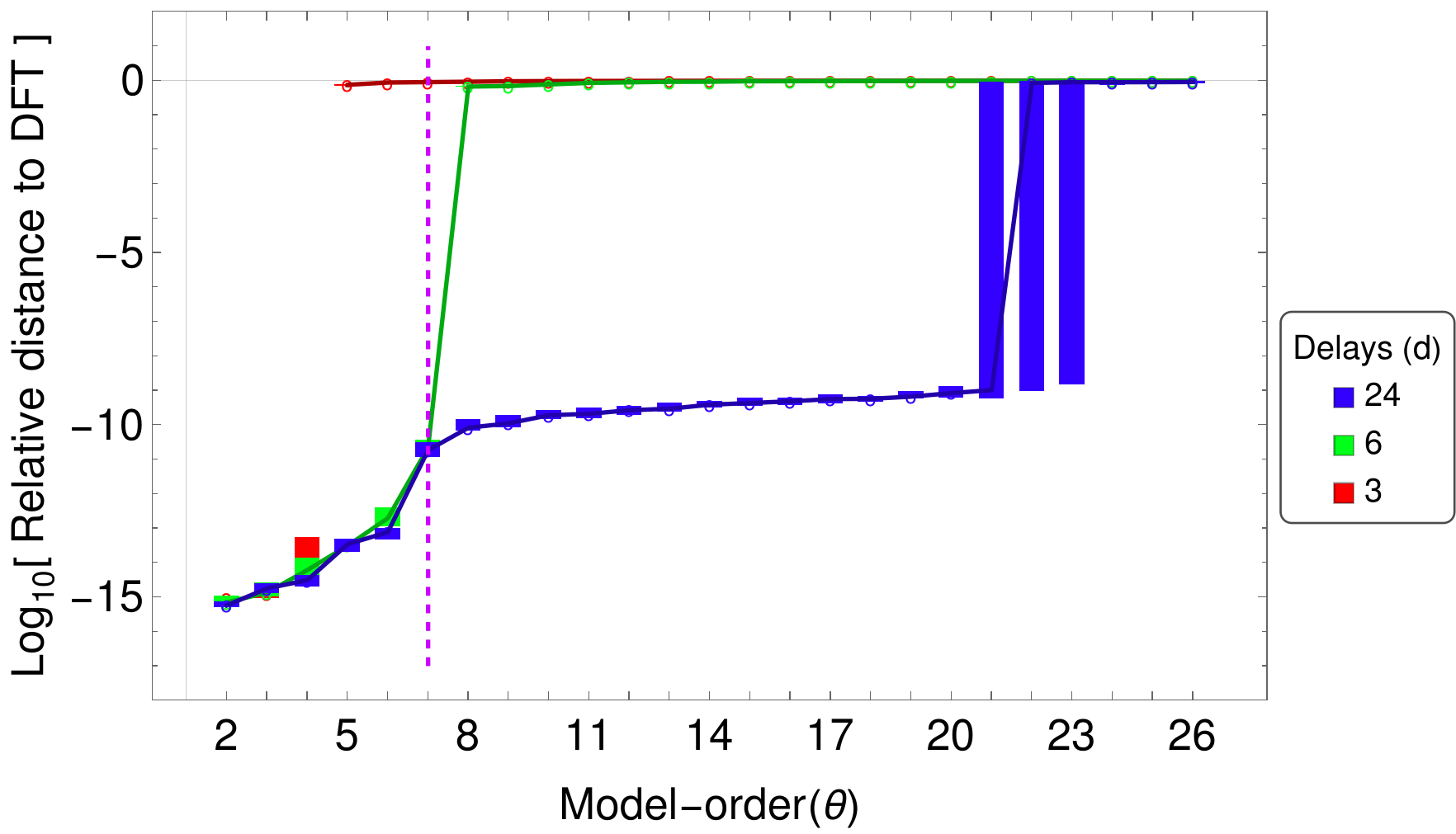}}
		
		\caption{For the linear time-invariant (LTI) systems described by \cref{eq:LTI1ab_LTI3}, box plots of  \(\DistanceToDFT\) reveal its dependence on the model order \((\theta)\) and the number of time delays (Colour-coded). 
        The magenta line indicates the system order \(\finitekissdim\), which is \(7\) for all three examples. 
        When a minimum of 6 time delays are taken and the model order is at least 7 (8) for \(\lti_{1a} (\lti_{1b})\), \DistanceToDFT~ is larger than at lower model orders, as predicted by \cref{cor:DMD_DFT_Numerical}. 
        The pertinent prediction for \(\lti_{3}\) - large values of \DistanceToDFT~for \(\theta \geq 8\) and \(\fundelayscount \geq 6\) - also appears to hold for \(d = 6\) i.e., when only the minimal required time delays are taken.
        However, increasing the delay embedding dimension to \(24\) significantly diminishes the spike in \DistanceToDFT~ at \(\theta = 8\), to the extent that it could be overlooked.
        }
		\label{fig:cms_transit_knownss}
	\end{figure}
	\begin{remark}
		\Cref{fig:cms_transit_knownss} shows secondary spikes in \(\DistanceToDFT\) for all three systems when an excess of time delays are taken.
		In particular, said spikes, which are present at \(\theta = 14, 16~\textrm{and}~21\) for \(\lti_{1a}, \lti_{1b} ~\textrm{and}~\lti_{3}\) respectively, are the result of \(\left\lVert (\mymat{X}_{\rm d-delayed})_{\rm ms}  - (\dummy{\mymat{X}}_{\rm d-delayed})_{\rm ms} \right \rVert\) becoming non-trivial.
		Hence, by \cref{rem:Only_computing_an_upper_bound},  \cref{fig:cms_transit_knownss}  depicts only an upper bound on \(\DistanceToDFT\) after the secondary spikes.
		For further discussion, see \cref{s:DisrepancyChecks}.
	\end{remark}

    \subsection{Systems that \emph{may not be} satisfying the requisites for \cref{cor:DMD_DFT_Numerical}}\label{ss:Systems_Unknown_ISS}
    When we do not know if our dictionary \(\dictionary\) lies in the non-redundant span of \(\finitekissdim\) distinct KEFs or if our initial condition \(\state_1\) is spectrally informative, we can still use the contrapositive of \cref{cor:DMD_DFT_Numerical} to get a lower bound on the number of Koopman modes constituting our dictionary.
    First, we note that if the predictions of \cref{cor:DMD_DFT_Numerical} are inconsistent with a numerical study that has taken sufficient time delays, then, we have either used a dictionary that is not non-redundantly spanned by a finite number of KEFs or our initial condition \(\state_1\) is not spectrally informative.
    In practice, the latter is not an issue because we can restrict our attention to only those KEFs whose zero level set does not contain \(\state_1\). 
    This change of perspective is permissible because the reduced collection of KEFs is equally capable of generating the training set, \(\mymat{Z}\), through time-sequential observations.
    Since the condition of \(\dictionary\) lying in the non-redundant span of \(\finitekissdim\) distinct KEFs can be interpreted as requiring \(\dictionary\) possess only \(\finitekissdim\) Koopman modes, its negation, over the specific guesses of \(\finitekissdim\) described below,  can give a lower bound on the number of Koopman modes comprising \(\dictionary\). 
    Thus, when the delay embedding dimension, \(d\), is large enough, numerical observations of \(\DistanceToDFT\) that are inconsistent with \cref{cor:DMD_DFT_Numerical} can inform the number of Koopman modes \emph{necessary} to represent \(\dictionary\).
    
    The contrapositive of \cref{cor:DMD_DFT_Numerical} can be computationally leveraged by assuming an upper-bound on the number of Koopman modes, performing the ensemble experiment described in \Cref{sss:Ensemble_Studies} and using errors, if any, in the predictions of \cref{cor:DMD_DFT_Numerical} to negate the presumed upper-bound.
    In each forthcoming numerical experiment, we begin with the assumption that our observables are non-redundantly spanned by \(\finitekissdim\) distinct KEFs.
    Although \(\finitekissdim\) is unknown, we assume that it is finite and has a known upper bound of \(\finitekissdim_{\rm max}\).
    This may be a theoretically informed cap or, if little is known about the system, simply the maximum model order that is palatable from a modeling perspective.
    Suppose we take at least \(\finitekissdim_{\max}-1\) time delays i.e., \(\fundelayscount \geq \finitekissdim_{\max}-1\).
    Then, by \cref{cor:DMD_DFT_Numerical}, \(\DistanceToDFT\) must behave like a step function with respect to the model order (\(\theta\)).
    It must be near zero for model orders less than \(\finitekissdim+1\) (\(\finitekissdim\) if 1 is a Koopman eigenvalue) and significantly higher for larger model orders.
    The absence of a step-like behavior for \(\theta\) less than or equal to \(\finitekissdim_{\max}+1\) would, then, mean that  at-least \(\finitekissdim_{\rm max}+1\) distinct KEFs are required to represent the span of our observables i.e., \(\dictionary\) has at-least \(\finitekissdim_{\max} + 1\) Koopman modes.
    In contrast, the presence of a step-like behavior, in agreement with \cref{cor:DMD_DFT_Numerical}, cannot be used to infer that \(\dictionary\) has a finite Koopman mode expansion.

	\subsubsection{Van der Pol oscillator}

    We begin with the Van der Pol oscillator, which is defined by the following\footnote{\(\upsilon_1\) and \(\upsilon_2\) are components of the state \(\state\).} differential equation:
    \begin{equation}\label{eq:ODE__Van_Der_Pol}
        \begin{aligned}
            \dot{\upsilon_1} &= \upsilon_2 \\
            \dot{\upsilon_2} &= (1-\upsilon_1^2)\upsilon_2-\upsilon_1.
        \end{aligned}
    \end{equation}
	We sample the flow map of \cref{eq:ODE__Van_Der_Pol} at equi-spaced points in time to produce the discrete time system \cref{eq:DS4SDMD}.  
    We choose our observable as an arbitrary but known linear combination of the state:
    \begin{equation}\label{eq:Dictionary__Van_Der_Pol}
        \dictionary[\state] := \tilde{\mymat{C}} \state, ~~~~~~ \tilde{\mymat{C}} \in \mathbb{C}^{1 \times 2}.
    \end{equation}

    Checking \cref{cor:DMD_DFT_Numerical} for this system generates \cref{fig:DistanceToDFT__VanDerPol}, where taking at-least \(12\) time delays seems to produce a consistent jump in \(\DistanceToDFT\), when \(\theta\) increases from 10 to 11.
    \begin{figure}
    	\centering
    	\includegraphics[width = 0.95 \linewidth]{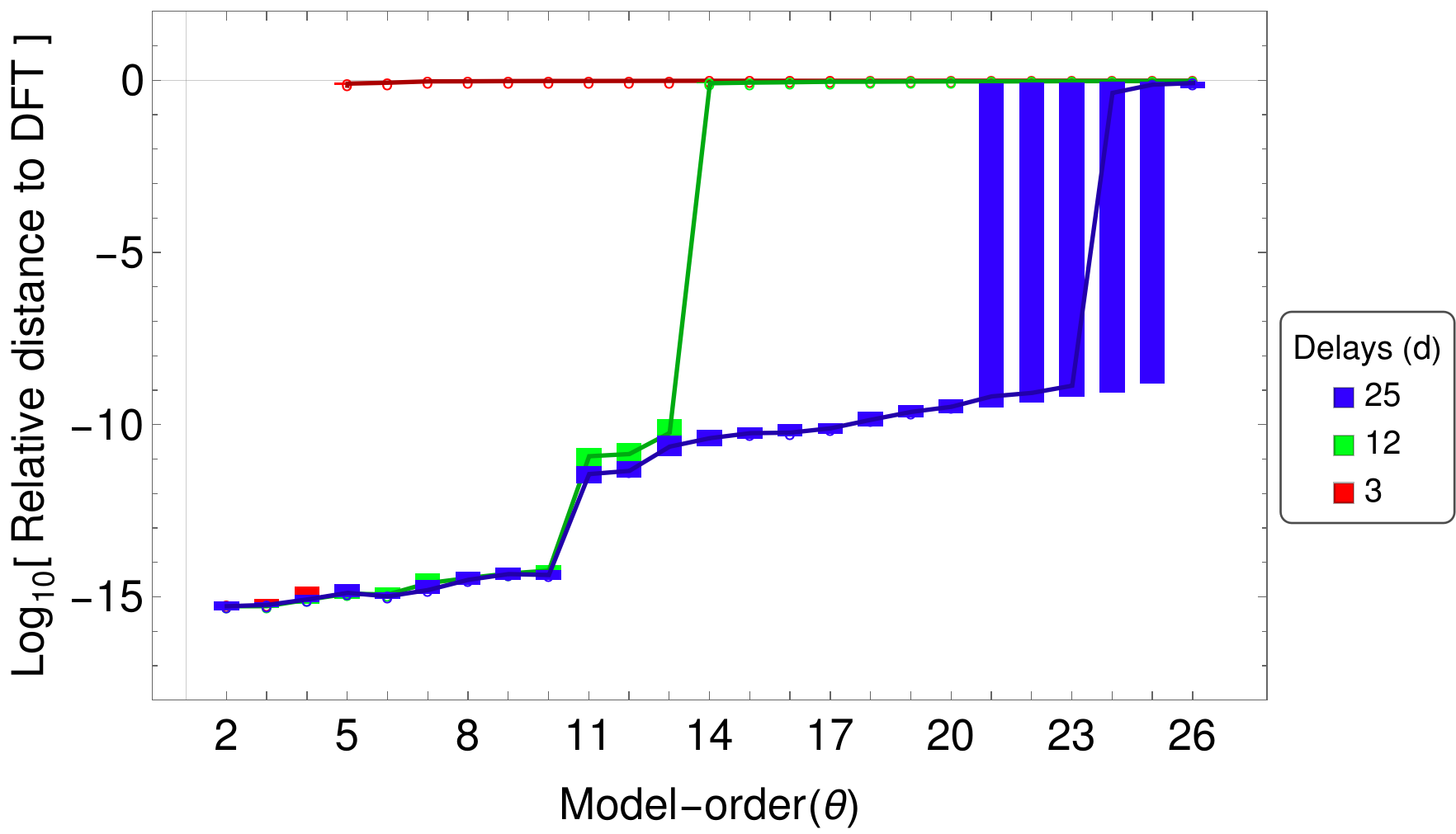}
    	\caption{The system order, \(\finitekissdim\), is unknown for the  \hyperref[eq:Dictionary__Van_Der_Pol]{dictionary that is used} to study the Van der Pol oscillator \cref{eq:ODE__Van_Der_Pol}.
   		Nonetheless, sufficient time delays lead to a step-like trend in \(\DistanceToDFT\).
   		By \cref{cor:DMD_DFT_Numerical}, the location of the jump \emph{might} be indicative of \(\dictionary\) possessing no more than 11 Koopman modes.}
    	\label{fig:DistanceToDFT__VanDerPol}
    \end{figure}
    Hence, \(\dictionary\) \emph{might not} possess more than 11 Koopman modes.
    Unfortunately, this speculation cannot be turned into a guarantee using \cref{cor:DMD_DFT_Numerical}.
    
    \subsubsection{Lid-driven cavity}

	Consider the lid-driven cavity for Reynolds numbers (\(Re\)) between \(13\times 10^3\) and \(30\times 10^3\). 
	As \(Re\) increases, the fluid flow transitions from periodic  to chaotic, passing through quasi-periodic and mixed behavior \cite{arbabi2019spectral}.
	We sample the continuous-time evolution\footnote{Numerical simulations from \cite{arbabi2019spectral}.} of the lid-driven cavity at equi-spaced points in time to produce the discrete-time dynamics \cref{eq:DS4SDMD}.
	We also choose an arbitrary linear functional of an associated stream-function as our observable.
	This sets the stage to study DMD-DFT equivalence, over a range of model orders \((\theta)\) and delay embedding dimensions \((d)\). 

	Over the range of \(Re\) considered, the qualitative change in dynamics  is reflected in the relationship between (the estimates\footnote{See \cref{s:DisrepancyChecks}} of) \(\DistanceToDFT\)  and \(\theta\) (\cref{fig:cms_transit_cavity}).
    As the Reynolds number increases, we see that the step-like trend associated with \(\fundelayscount = 25\) has its plateau beginning at larger values of \(\theta\) before disappearing altogether in the last scenario \((Re=30\times 10^3)\). 
	Hence, our observable \emph{may} have a finite number of Koopman modes in the first three cases \((Re=13\times 10^3,~ 16\times 10^3~\textrm{and}~ 20\times 10^3)\).
    
    However, for \(Re=30\times 10^3\), we can go further and certify that our observable, \(\dictionary\), possesses at least 26 Koopman modes.
    To get this guarantee, we first assume that \(\dictionary\) has utmost 25 Koopman modes or, equivalently, that it lies in the non-redundant span of utmost \(25\) distinct KEFs.
    In other words, we assume \(r_{\rm max} = 25\).
    The concomitant requirement of at least 24 time delays is met by the study corresponding to the dark blue boxes in panel (d) of \cref{fig:cms_transit_cavity}.
    Hence, by \cref{cor:DMD_DFT_Numerical}, we would expect it to spike, at the latest, by \(\theta = 26\).
    In contrast to the studies with a lower number of time delays (\(\fundelayscount~=~3~\textrm{and}~12\)), the trend shown by the dark blue boxes lacks a jump.
    Hence,  \(\dictionary\) cannot have fewer than 26 Koopman modes.

    Therefore, \cref{cor:DMD_DFT_Numerical} can provide a data-driven lower bound on the number of Koopman modes of \(\dictionary\).
    This is accomplished by taking a minimum number of time delays and looking for \emph{discrepancies} with the predictions made in \cref{cor:DMD_DFT_Numerical}.
    Alas, one cannot use the confirmation of the same predictions to certify that \(\dictionary\) possesses a finite Koopman mode expansion.

	\begin{figure}[]
		\centering
		\subfloat[\(Re ~=~ 13\times 10^3 \) (Periodic)]{\includegraphics[width = 0.45 \linewidth]{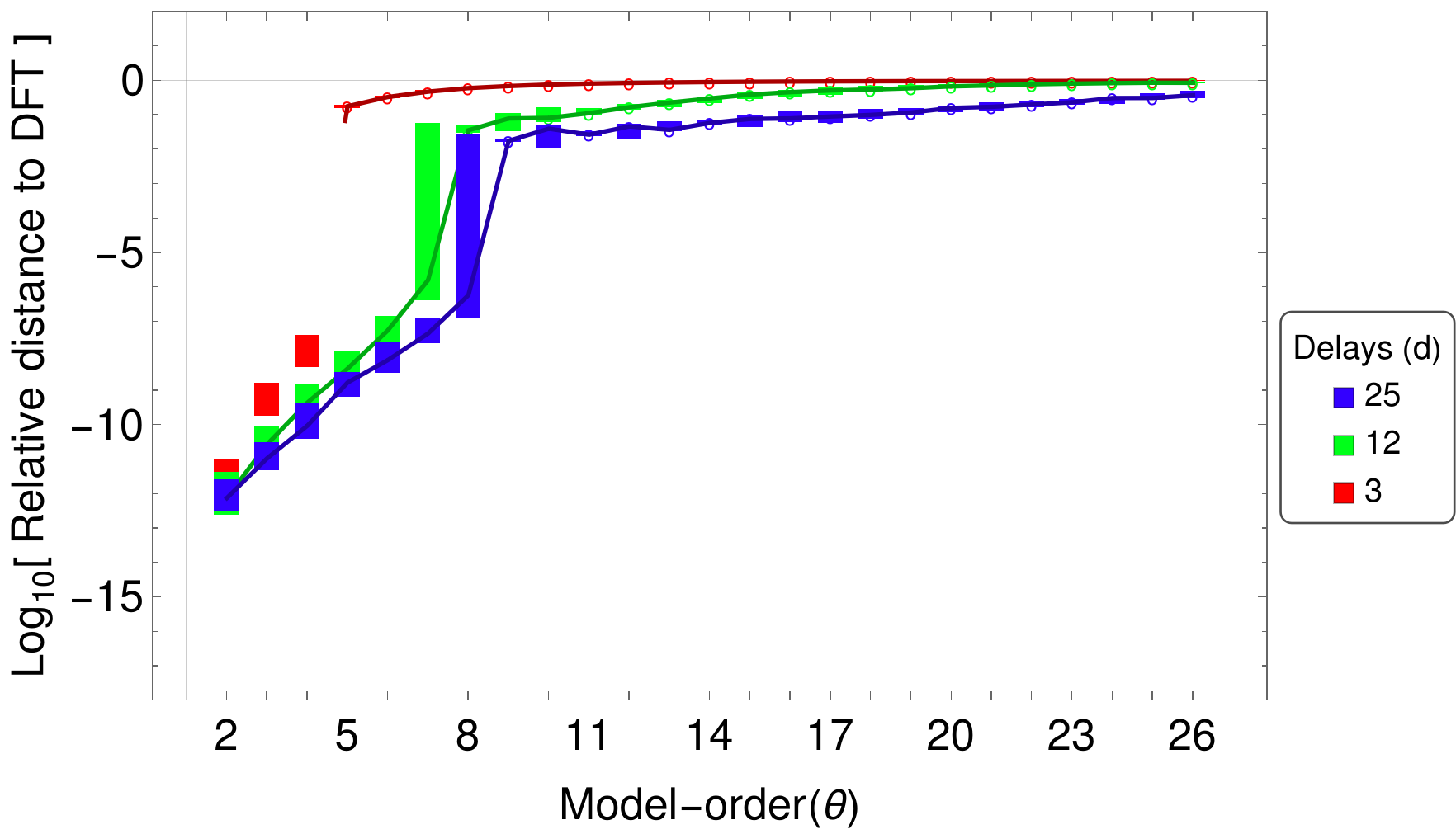}} \quad
		\subfloat[\(Re ~=~ 16\times 10^3\) (Quasi-periodic)]{\includegraphics[width = 0.45 \linewidth]{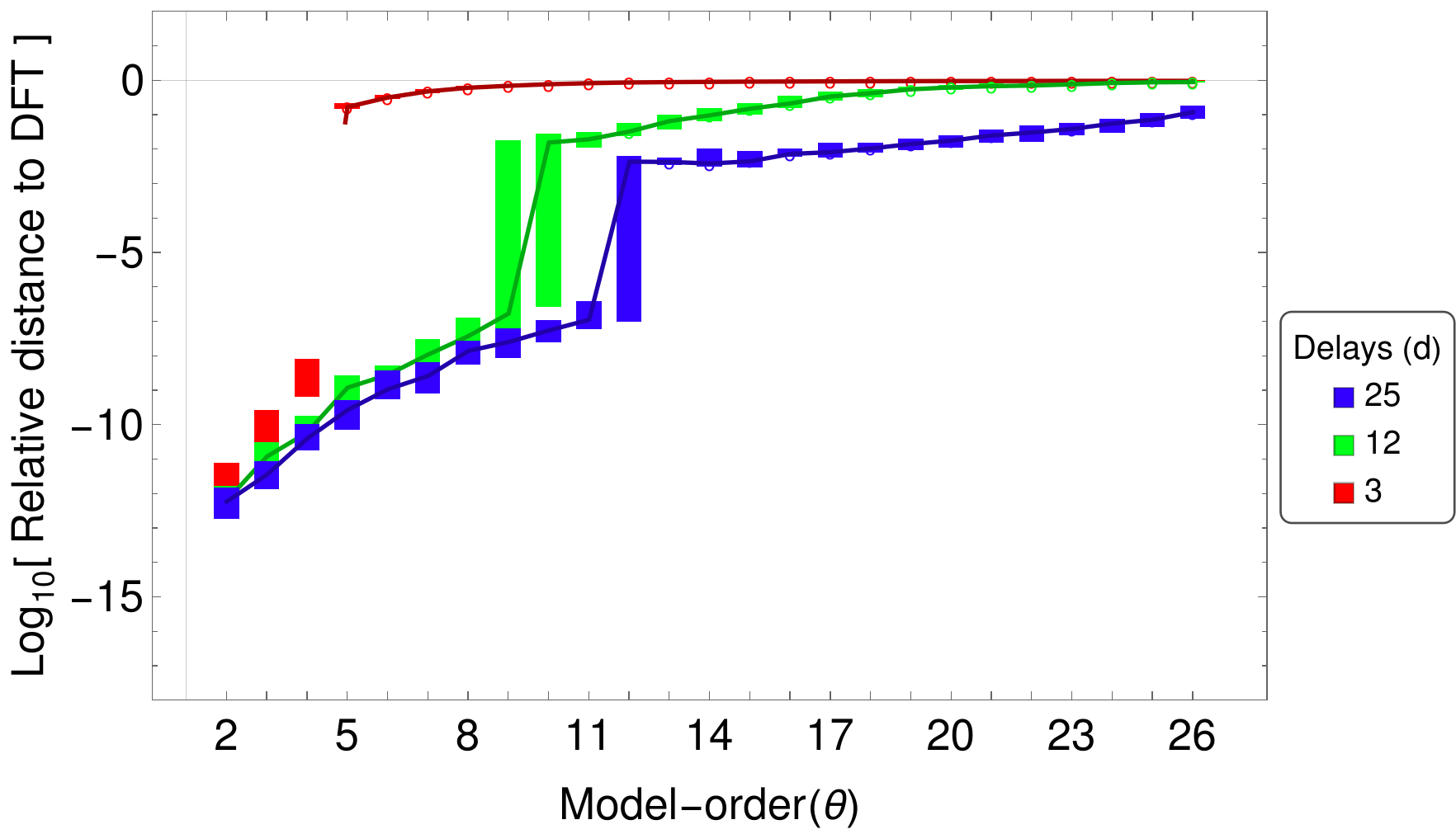}}  \\
		\subfloat[\(Re ~=~ 20\times 10^3\) (Mixed)]{\includegraphics[width = 0.45 \linewidth]{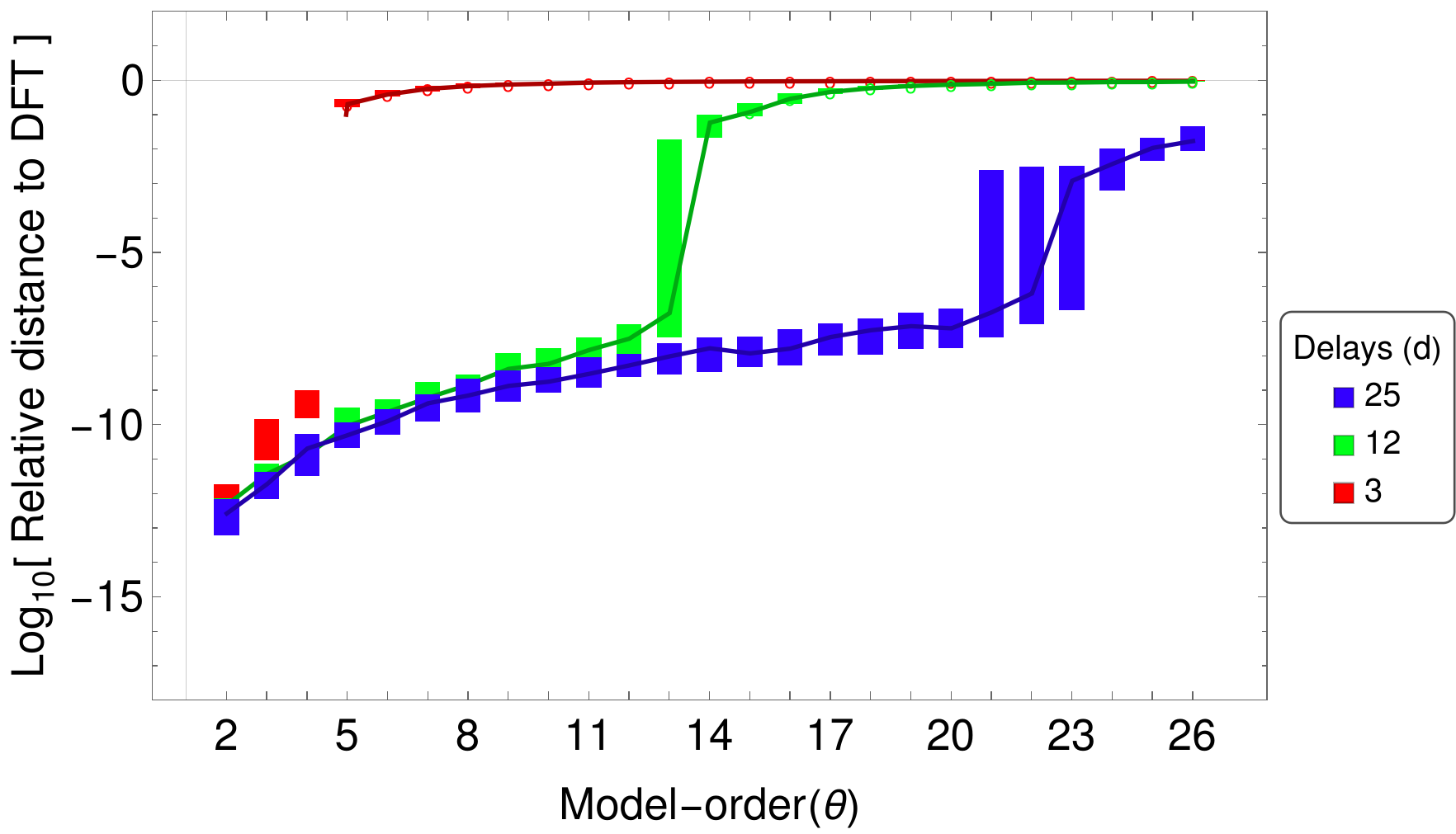}} \quad
		\subfloat[\(Re~=~ 30\times 10^3\) (Chaotic)]{\includegraphics[width = 0.45 \linewidth]{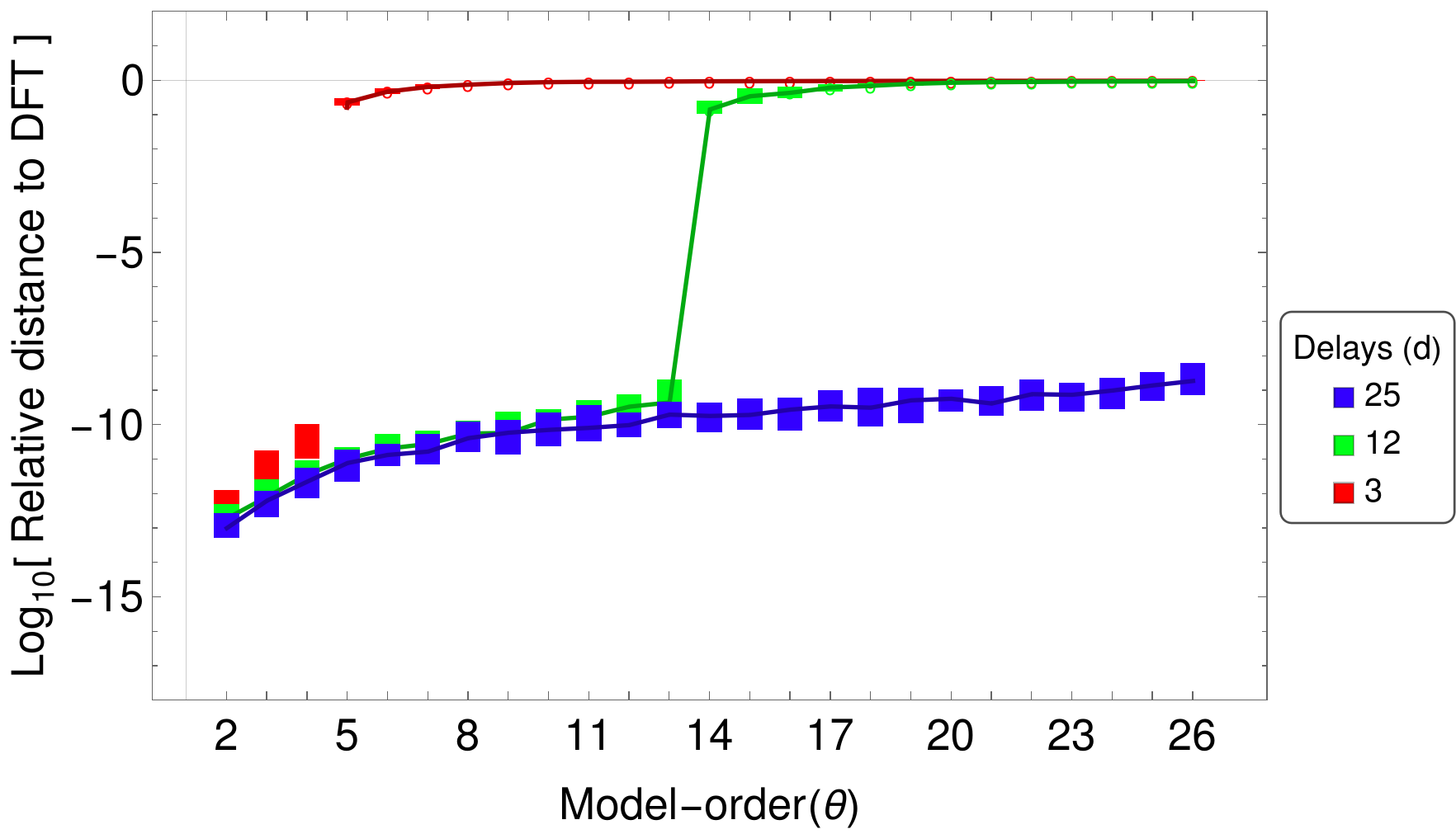}} 
		\caption{ 			
		As \(Re\) increases, the lid-driven cavity flow grows in complexity.
		This correlation is reflected in the above box plots of (\hyperref[s:DisrepancyChecks]{estimated}) \(\DistanceToDFT\).		
        Panels (a)-(c) display jumps that occur at larger values of \(\theta\).
		In contrast, the final plot (d) is conspicuous in its lack of a discontinuity. 
        By \cref{cor:DMD_DFT_Numerical}, we can infer that \(\dictionary\) possesses at-least 26 Koopman modes.
        This observation agrees with the fact that the Koopman operator does not possess any eigenfunctions when the underlying dynamics is chaotic.
        }
		\label{fig:cms_transit_cavity}
	\end{figure}
    
	\section{Conclusions and Future Work}
	\label{s:Conclusions}
	The relation between mean-subtracted DMD and temporal DFT has been clarified, for observables that possess only a finite number of Koopman modes.
	When a collection of such observables spans a subspace invariant under the Koopman operator and is chosen as the dictionary in DMD, non-equivalence of mean-subtracted DMD and DFT is tantamount to sufficiency of training data.
	The requisite invariance can be attained by taking as many time delays as the number of distinct Koopman modes.
	Therefore, DMD-DFT equivalence vanishes when data is plentiful, and delay embedded.
	The contrapositive suggests that DMD-DFT equivalence can be used as an indicator of inadequate training data- a property that contrasts with its' original perception as a potential rot.
	
	For future work, we note that increasing the number of time delays, far beyond the prescription of \cref{prop:DelaysMakeGoodObservables}, can acutely blur the numerical distinction between DMD-DFT equivalence and non-equivalence.
	The theory developed in this work is oblivious to the choice of delay embedding dimension, beyond requiring a minimum value.
	Understanding this distortion produced by large delay embeddings will improve the reliability of DMD-DFT equivalence in inferring the order of the underlying dynamics.
	
	\section*{Acknowledgments}
	
	This work was supported by the Army Research Office (ARO-MURI W911NF-17-1-030) and the National Science Foundation
	(Grant no. 1935327).

	\bibliographystyle{siamplain}
	\bibliography{refs_tiny}
	
	\appendix

	\section{Proof of \cref{thm:muDMD_preserves_constant_KEFs}}\label{s:Proof__muDMD_preserves_KEFs_at_1}
	
	For any choice of \(\state_{\rm test}\), we can establish \cref{eq:msub_forecasts__preserve__BCs} by induction on \(j\). 
	
	\mypara{Step 1}
	When \(j=1\), the identity \cref{eq:msub_forecasts__preserve__BCs} trivially holds. 
	
	\mypara{Step 2}
	Suppose \cref{eq:msub_forecasts__preserve__BCs} holds for \(j \leq o\).
	
	\mypara{Step 3}
	We need to show that \cref{eq:msub_forecasts__preserve__BCs} also holds for \(j = o+1\) i.e.,
	\begin{displaymath}
		\mymat{E}^H (\reduced{\myvec{z}}_{\rm test})_{o+1} ~=~ \mymat{E}^H (\reduced{\myvec{z}}_{\rm test})_1.
	\end{displaymath}
	To this end, we will split the analysis into two cases: 
	
	\mysubpara{Step 3.1: \(o + 1 \leq \undelayedDMDorder\)}
	
	Sequential application of \cref{eq:forecast_muDMD}, \cref{eq:Abstracted_BCs} and \cref{eq:forecast_muDMD}  yields the desired conclusion.
	\begin{displaymath}
		\mymat{E}^H \EmphasiseReduction{(\reduced{\myvec{z}}_{\rm test})_{o+1}}
		~=~ \EmphasiseReduction{
			\mymat{E}^H (\myvec{z}_{\rm test})_{o+1}
		} 
		~=~ \mymat{E}^H 
		\EmphasiseReduction{
			(\myvec{z}_{\rm test})_1 
		}
		~=~ \mymat{E}^H (\reduced{\myvec{z}}_{\rm test})_1.
	\end{displaymath} 
	
	\mysubpara{Step 3.2: \(o+1 > \undelayedDMDorder\)}
	
	Unpacking \((\reduced{\myvec{z}}_{\rm test})_{o+1}\) with \cref{eq:forecast_muDMD}, we get:
	\begin{displaymath}
		\begin{aligned}
			\mymat{E}^H 
			\EmphasiseReduction{
				(\reduced{\myvec{z}}_{\rm test})_{o+1}
			} ~&=~ 
			\mymat{E}^H 
			\left(  \myvec{\mu}_{\rm test} ~+~ 
			\begin{bmatrix}
				(\reduced{\myvec{z}}_{\rm test})_{o+1 - \undelayedDMDorder + k} - \myvec{\mu}_{\rm test}
			\end{bmatrix}_{k=0}^{\undelayedDMDorder-1}~
			~
			\opt{\myvec{c}}[\mymat{Z}_{\rm ms}] 
			\right) \\
			~&=~ 
			\EmphasiseReduction{
				\mymat{E}^H \myvec{\mu}_{\rm test}
			} ~+~
			\begin{bmatrix}
				\EmphasiseReduction{
					\mymat{E}^H(\reduced{\myvec{z}}_{\rm test})_{o+1 - \undelayedDMDorder + k}
				} 
				- 
				\EmphasiseReduction{
					\mymat{E}^H\myvec{\mu}_{\rm test}
				}
			\end{bmatrix}_{k=0}^{\undelayedDMDorder-1}~
			~
			\opt{\myvec{c}}[\mymat{Z}_{\rm ms}]. 
		\end{aligned}
	\end{displaymath}
	
	The expression \(\mymat{E}^H \myvec{\mu}_{\rm test} \) can be simplified by applying \cref{eq:Defn__Test_Temporal_Mean} followed by \cref{eq:Abstracted_BCs}:
	\begin{displaymath}
		\mymat{E}^H 
		\EmphasiseReduction{
			\myvec{\mu}_{\rm test} 
		}
		~=~ \frac{1}{\undelayedDMDorder+1} \sum_{j=1}^{\undelayedDMDorder+1} 
		\EmphasiseReduction{
			\mymat{E}^H (\myvec{z}_{\rm test})_j 
		}
		~=~ \frac{1}{\undelayedDMDorder+1} \sum_{j=1}^{\undelayedDMDorder+1} \mymat{E}^H
		(\myvec{z}_{\rm test})_1
		~=~  \mymat{E}^H
		(\myvec{z}_{\rm test})_1.
	\end{displaymath} 
	Similarly, the term \(\mymat{E}^H(\reduced{\myvec{z}}_{\rm test})_{o+1 - \undelayedDMDorder + k}\) can be reduced by noting that \(k \in \{0, \dots, \undelayedDMDorder-1\}\).
	Since this means \(o + 1 - \undelayedDMDorder + k \leq o\),  we can apply Step 2 of the induction followed by \cref{eq:forecast_muDMD} to get:
	\begin{displaymath}
		\EmphasiseReduction{
			\mymat{E}^H (\reduced{\myvec{z}}_{\rm test})_{o + 1 - \undelayedDMDorder + k}
		}
		~=~ \mymat{E}^H 
		\EmphasiseReduction{
			(\reduced{\myvec{z}}_{\rm test})_{1}
		}
		~=~ \mymat{E}^H (\myvec{z}_{\rm test})_1.
	\end{displaymath}
	Combining the above reductions and, then, using \cref{eq:forecast_muDMD} yields the desired identity:
	\begin{displaymath}
		\begin{aligned}
			&\EmphasiseReduction{
				\mymat{E}^H \myvec{\mu}_{\rm test}
			} ~+~
			\begin{bmatrix}
				\EmphasiseReduction{
					\mymat{E}^H(\reduced{\myvec{z}}_{\rm test})_{o+1 - \undelayedDMDorder + k}
				} 
				- 
				\EmphasiseReduction{
					\mymat{E}^H\myvec{\mu}_{\rm test}
				}
			\end{bmatrix}_{k=0}^{\undelayedDMDorder-1}~
			~
			\opt{\myvec{c}}[\mymat{Z}_{\rm ms}] \\
			~=~	
			&\mymat{E}^H
			\EmphasiseReduction{
				(\myvec{z}_{\rm test})_1
			}	
			~+~
			\begin{bmatrix}
				\EmphasiseReduction{
					\mymat{E}^H (\myvec{z}_{\rm test})_1 
					- 
					\mymat{E}^H	(\myvec{z}_{\rm test})_1
				}
			\end{bmatrix}_{k=0}^{\undelayedDMDorder-1}~
			~
			\opt{\myvec{c}}[\mymat{Z}_{\rm ms}] \\	
			~=~	&\mymat{E}^H
			(\reduced{\myvec{z}}_{\rm test})_1.
		\end{aligned}
	\end{displaymath}		
	\manualQED

	\section{Unpacking the structure induced by recurrent assumptions on the dictionary and the training trajectory}
	
	All the major results in this work make the following two assumptions:
	\begin{enumerate}
		\item The dictionary \PsiLiesInNRSpanOfrDistinctKEFs, \(\KEFs\), with distinct eigenvalues, \(\KEvals\).
		\item The initial condition \(\state_1\) is \SpectrallyInformative.
	\end{enumerate}
	In this setting, the \hyperref[defn:Koopman_Mode_Decomposition]{Koopman mode expansion}  of \(\dictionary\) manifests as a specific factorization of the time series \(\mymat{Z}\) (\Cref{ss:KMD_Compact}).
	Furthermore, one of these factors, the matrix of Koopman modes, provides an alternative characterization of  \hyperref[defn:Koopman_invariant_psi]{Koopman invariant} dictionaries (\Cref{ss:RephraseKoopmanInvariance}).
	
	\subsection{The Koopman mode factorization of \(\mymat{Z}\)}\label{ss:KMD_Compact}
	\begin{lemma}\label{lem:KMF_of_Z}
		Suppose \PsiLiesInNRSpanOfrDistinctKEFs, \(\KEFs\), with distinct eigenvalues, \(\KEvals\).
		
		Then, the row space of \(\mymat{Z}\) is contained in the row space of the Vandermonde matrix, 
		\begin{equation}\label{eq:defn_Theta}
			\mymat{\Theta} := 
			\begin{bmatrix}
				1 &\lambda_1 & \lambda_1^2 & \hdots & \lambda_1^n \\
				1 &\lambda_2 & \lambda_2^2 & \hdots & \lambda_2^n \\
				\vdots & \vdots & \vdots & \ddots & \vdots \\
				1 &\lambda_r & \lambda_r^2 & \hdots & \lambda_r^n \\
			\end{bmatrix}.
		\end{equation}
		Specifically, if we define \(\mymat{C}\) to be the column scaling of \hyperref[defn:NonRedundantSpanOfDistinctKEFs]{\(\tilde{\mymat{C}}\)} by the eigen-coordinates of the initial condition \(\state_1\), 	
		\begin{equation}\label{eq:C_definition}
			\diag[\KEFVector(\state_1)] ~:=~ 
			\begin{bmatrix}
				\phi_1(\state_1) & & & \\
				& \phi_2(\state_1) & & \\
				& & \ddots & \\
				& & & \phi_r(\state_1)
			\end{bmatrix}, \quad
			\mymat{C} ~:=~\tilde{\mymat{C}}~ \diag[\KEFVector(\state_1)],
		\end{equation}
		we get a ``Koopman mode'' factorization of \(\mymat{Z}\):
		\begin{equation}\label{eq:KMF__of__Z}
			\mymat{Z}
			= \mymat{C}  \mymat{\Theta}. 
		\end{equation} 
		
		Furthermore, a \SpectrallyInformative~ \(\state_1\) is necessary and sufficient to ensure there are no redundancies in said decomposition:
		\begin{equation}\label{eq:C_has_no_zero_cols__IFF__IC_is_spectrally_informative}
			\state_1~\textrm{is spectrally informative}~~\iff~~\forall~i,\quad \mymat{C}\myvec{e}_i \neq \myvec{0}.
		\end{equation}
	\end{lemma}
	
	\mypara{Proof}
	
	\mysubpara{Deducing the Koopman mode factorization of \(\mymat{Z}\)}
	Recall the construction of \(\mymat{Z}\) from \cref{eq:Zmat4CompDMD}:
	\begin{displaymath}
		\mymat{Z} 
		~=~ \left[ 
		\myvec{z}_{j} \right]_{j=0}^{\undelayedDMDorder}
		~=~ \left[ \dictionary\left( \Gamma^{j}\left(\state_1\right) \right) \right]_{j=0}^{\undelayedDMDorder}.
	\end{displaymath}
	Here, for the sake of convenience, we have numbered the columns of \(\mymat{Z}\) from \(0\) to \(\undelayedDMDorder\), instead of the original choice of \(1\) to \(\undelayedDMDorder+1\).
	We can begin by using \cref{eq:defn__Ctilde} to write \(\dictionary\) in terms of the eigenfunctions \(\KEFVector\) and, then, leverage \cref{eq:Koopman__on__VectorObservables} to replace the potentially nonlinear map \(\Gamma^{j} \) with the Koopman operator.
	\begin{displaymath}
		\EmphasiseReduction{\dictionary}(\Gamma^j(\state_1)) 
		~=~ \tilde{\mymat{C}}\left( \EmphasiseReduction{\KEFVector\circ \Gamma^j}\right)(\state_1)
		~=~ \tilde{\mymat{C}}\left(U^j \circ \KEFVector\right) (\state_1).
	\end{displaymath}
	Now, if we unpack each component of \(\left(U^j \circ \KEFVector\right) (\state_1)\) and invoke the defining property of a Koopman eigen-function (\ref{defn:Koopman_Eigenfunctions}),
	\begin{displaymath}
		\left(U^j \circ \KEFVector\right) (\state_1)
		~=~ \left[ \left( \EmphasiseReduction{U^j \circ \phi_i}\right) (\state_1) \right]_{i=1}^\finitekissdim
		~=~ \left[ \lambda_i^j \phi_i (\state_1) \right]_{i=1}^\finitekissdim,
	\end{displaymath}
	then, we can use \cref{eq:C_definition} to replace \(\tilde{\mymat{C}}\) with \(\mymat{C}\).
	\begin{displaymath}
		\tilde{\mymat{C}} \left[ \lambda_i^j \EmphasiseReduction{\phi_i (\state_1)} \right]_{i=1}^\finitekissdim
		~=~		\EmphasiseReduction{\tilde{\mymat{C}}~ \diag[\KEFVector(\state_1)]}  \left[\lambda_i^j \right]_{i=1}^\finitekissdim
		~=~\mymat{C} \left[\lambda_i^j \right]_{i=1}^\finitekissdim.
	\end{displaymath}
	Compiling this expression for each column of \( \mymat{Z} \) gives \cref{eq:KMF__of__Z}.
	
	\mysubpara{Connecting spectral informativeness to the absence of redundancies}
	According to \cref{eq:Every_column_in_CTilde_is_NonZero}, every column of \(\tilde{\mymat{C}}\) is non-zero.
	
	When \(\state_1\) is spectrally informative, each of the \(\finitekissdim\) KEFs, \(\KEFs\), evaluates to a non-zero value at \(\state_1\).
	Since these same non-zero values scale the columns of \(\tilde{\mymat{C}}\), in \cref{eq:C_definition}, to produce \(\mymat{C}\), every column of \(\mymat{C}\) is also non-zero.
	
	On the other hand, if \(\state_1\) is not spectrally informative, then, there is a KEF that evaluates to zero at \(\state_1\).
	Hence, we can apply \cref{eq:C_definition} to deduce that one of the columns in \(\mymat{C}\) is zero.
	
	\manualQED

	\begin{remark}
		If we define \(\mymat{\Theta}_j\) as \(\mymat{\Theta}\) without the last \(j\) columns,
		\begin{equation}\label{eq:Vandermonde_and_submatrices}
			\mymat{\Theta}_j ~:=~          \begin{bmatrix}
				1 &\lambda_1 & \lambda_1^2 & \hdots & \lambda_1^{n-j} \\
				1 &\lambda_2 & \lambda_2^2 & \hdots & \lambda_2^{n-j} \\
				\vdots & \vdots & \vdots & \ddots & \vdots \\
				1 &\lambda_r & \lambda_r^2 & \hdots & \lambda_r^{n-j} \\
			\end{bmatrix},
		\end{equation}	
		and denote by \(\mymat{\Lambda}\) an appropriate diagonal matrix of Koopman eigenvalues,
		\begin{equation}\label{eq:KEvalsDiagonally}
			\mymat{\Lambda}~:=~
			\begin{bmatrix}
				\lambda_1 & & & \\
				& \lambda_2 & & \\
				& & \ddots & \\
				& & & \lambda_r
			\end{bmatrix},
		\end{equation}
		then, an analogue of \cref{eq:KMF__of__Z} exists for the sub-matrices \(\mymat{X}\) and \(\mymat{Y}\) too:
		\begin{equation}\label{eq:KMF__of__X_Y}
			\begin{aligned}
				\mymat{X} &= [~\myvec{z}_j~]_{j=0}^{\undelayedDMDorder-1} =& \mymat{C} \mymat{\Theta}_1.\\
				\mymat{Y} &= [~\myvec{z}_j~]_{j=1}^{\undelayedDMDorder} =& \mymat{C}  \mymat{\Lambda} \mymat{\Theta}_1.
			\end{aligned}
		\end{equation} 	
	\end{remark}

	\subsection{An analytically useful re-phrasal of Koopman invariance}\label{ss:RephraseKoopmanInvariance}
	
	\begin{proposition}\label{prop:ManyFacesofLinCon}
		Suppose \PsiLiesInNRSpanOfrDistinctKEFs.
		Then, \hyperref[defn:Koopman_invariant_psi]{Koopman invariance of \(\dictionary\)} is equivalent to the columns of \(\tilde{\mymat{C}}\) being linearly independent.
		\begin{equation}\label{eq:U_Invariance__via__Ctilde}
			\dictionary~\textrm{is Koopman invariant}~\iff~\tilde{\mymat{C}}~\textrm{has full column rank.}
		\end{equation}
		
		Under the additional condition of a \SpectrallyInformative\, \(\state_1\), Koopman invariance can be inferred, alternatively, from the matrix \(\mymat{C}\):
		\begin{equation}\label{eq:U_Invariance__via__C}
			\dictionary~\textrm{is Koopman invariant}~\iff~\mymat{C}~\textrm{has full column rank.}
		\end{equation}
	\end{proposition}
	
	\mypara{Proof}
	We will focus on establishing \cref{eq:U_Invariance__via__Ctilde} because it is readily transmuted by a spectrally informative \(\state_1\) into \cref{eq:U_Invariance__via__C}  .
	
	\mysubpara{Koopman invariance via \( \tilde{\mymat{C}} \)}
	
	The backward implication can be proven using the \(m \times m\) matrix,
	\begin{displaymath}
		\mymat{A}~:=~\tilde{\mymat{C}} \mymat{\Lambda} \tilde{\mymat{C}}^\dagger,
	\end{displaymath}
	as the guess for the action of the Koopman operator on \(\dictionary\). 
	In particular, we will show that \(U \circ \dictionary~=~ \mymat{A} \dictionary\).
	To begin with, we can expand \(\dictionary\) on the right-hand side using \cref{eq:defn__Ctilde} and, then, unpack \(\mymat{A}\):
	\begin{displaymath}
		\mymat{A} \EmphasiseReduction{\dictionary}
		~=~ \EmphasiseReduction{\mymat{A}} \tilde{\mymat{C}} \KEFVector 
		~=~ (\tilde{\mymat{C}} \mymat{\Lambda} \tilde{\mymat{C}}^\dagger ) \tilde{\mymat{C}} \KEFVector.
	\end{displaymath}
	Since \(\tilde{\mymat{C}}\) has full column rank, \(\tilde{\mymat{C}}^\dagger  \tilde{\mymat{C}} = \mymat{I}\). 
	Using this identity lets us apply \cref{eq:KEvalsDiagonally,eq:Koopman__on__VectorObservables} simultaneously to bring in the Koopman operator \(U\):              
	\begin{displaymath}
		\tilde{\mymat{C}} \mymat{\Lambda}
		\EmphasiseReduction{\tilde{\mymat{C}}^\dagger  \tilde{\mymat{C}}} 
		\KEFVector
		~=~ \tilde{\mymat{C}} 
		\EmphasiseReduction{\mymat{\Lambda} \KEFVector}
		~=~ \tilde{\mymat{C}}~ \left(U \circ \KEFVector\right).
	\end{displaymath}
	Finally, we can use the linearity of \(U\) to link \(\tilde{\mymat{C}}\) and \(\KEFVector\), thereby permitting another application of \cref{eq:defn__Ctilde}:
	\begin{displaymath}
		\EmphasiseReduction{\tilde{\mymat{C}}~ U} \circ \KEFVector
		~=~ U \circ \EmphasiseReduction{\tilde{\mymat{C}}  \KEFVector}
		~=~ U \circ \dictionary.
	\end{displaymath}
	
	The forward implication can be tackled by starting with the definition of a Koopman invariant dictionary and reversing the arguments used above.
	Recall that a Koopman invariant \(\dictionary\) implies, by \cref{defn:Koopman_invariant_psi}, the existence of a matrix \(\mymat{A}\) such that
	\begin{displaymath}
		U \circ \dictionary~=~ \mymat{A} \dictionary.
	\end{displaymath}
	Now, for the left-hand side, we can simply reverse the steps taken towards the conclusion of the forward implication- Invoke \cref{eq:defn__Ctilde}, follow it up with the linearity of \(U\) and conclude by applying \cref{eq:KEvalsDiagonally,eq:Koopman__on__VectorObservables} simultaneously: 
	\begin{displaymath}
		U \circ \EmphasiseReduction{\dictionary} ~=~
		\EmphasiseReduction{	U \circ \tilde{\mymat{C}}} \KEFVector
		~=~ \tilde{\mymat{C}} \EmphasiseReduction{	U \circ \KEFVector}
		~=~ \tilde{\mymat{C}} \mymat{\Lambda} \KEFVector.
	\end{displaymath}
	For the right-hand side, simply apply \cref{eq:defn__Ctilde}.
	\begin{displaymath}
		\mymat{A} \EmphasiseReduction{\dictionary} ~=~ \mymat{A} \tilde{\mymat{C}} \KEFVector.
	\end{displaymath}
	Combining the two expansions, we get:
	\begin{displaymath}
		\tilde{\mymat{C}} \mymat{\Lambda} \KEFVector ~=~ \mymat{A} \tilde{\mymat{C}} \KEFVector.
	\end{displaymath}
	Since this relationship holds for all values of \(\KEFVector\), we have:
	\begin{displaymath}
		\tilde{\mymat{C}} \mymat{\Lambda} ~=~ \mymat{A} \tilde{\mymat{C}}.
	\end{displaymath}
	Observe that, by \cref{eq:KEvalsDiagonally}, the columns of \(\tilde{\mymat{C}}\) are the eigen vectors of \(\mymat{A}\) corresponding to distinct eigen-values.
	Hence, they must be linearly independent.

	\mysubpara{Koopman invariance via \( \mymat{C} \)}
	When the initial condition \(\state_1\) is spectrally informative, then, \(\mymat{C}\) has full column rank if and only if \(\tilde{\mymat{C}}\) also possesses full column rank.
	According to \cref{defn:SpectrallyInformativeState}, \(\state_1\) is spectrally informative when each of the \(\finitekissdim\) Koopman eigenfunctions, \(\KEFs\), evaluates to a non-zero value at \(\state_1\).
	Now, \cref{eq:C_definition} tells us that these same values are diagonally stacked to form \(\diag[\KEFVector(\state_1)]\).
	In other words,
	\begin{displaymath}
		\state_1~\textrm{is spectrally informative}~\iff~\diag[\KEFVector(\state_1)]~\textrm{is invertible.}
	\end{displaymath}
	Furthermore, this matrix is used to construct \(\mymat{C}\) from \(\tilde{\mymat{C}}\):
	\begin{displaymath}
		\mymat{C} ~=~\tilde{\mymat{C}}~ \diag[\KEFVector(\state_1)].
	\end{displaymath}
	Therefore, when \(\state_1\) is spectrally informative, we have,
	\begin{displaymath}
		\tilde{\mymat{C}}~\textrm{has full column rank}~\iff~\mymat{C}~\textrm{has full column rank.}
	\end{displaymath}
	
	Taken together with \cref{eq:U_Invariance__via__Ctilde}, we see that \cref{eq:U_Invariance__via__C} holds as well.
	
	\manualQED
	
	
	\section{The size of a Vandermonde matrix can determine its rank}
	Vandermonde matrices will be ubiquitous in the forthcoming analysis.
	In preparation, we highlight two rank conditions, using the Vandermonde matrix,
	\begin{equation}\label{eq:Vandermonde__Example}
		\mymat{V} ~:=~ 
		\begin{bmatrix}
			1 &\nu_1 & \nu_1^2 & \hdots & \nu_1^{c-1} \\
			1 &\nu_2 & \nu_2^2 & \hdots & \nu_2^{c-1} \\
			\vdots & \vdots & \vdots & \ddots & \vdots \\
			1 &\nu_r & \nu_r^2 & \hdots & \nu_r^{c-1} \\
		\end{bmatrix}.
	\end{equation}
	
	\begin{definition}[Nodes of a Vandermonde matrix]\label{def:Vandermonde_Nodes}
		The \(r\) complex numbers \(\{\nu_i\}_{i=1}^r\), whose powers constitute \(\mymat{V}\), are called its ``nodes''.
	\end{definition}

	\begin{lemma}[Node distinctness lets size dictate rank]\cite{horn2012matrix,hirsh2019centering}
		Suppose the nodes of \(\mymat{V}\) are distinct:
		\begin{displaymath}
			\forall~(i,j), \quad i \neq j \implies \nu_i \neq \nu_j.
		\end{displaymath}
		Then, either the rows or the columns of \(\mymat{V}\) are linearly independent:
		\begin{align}
			\mymat{V}~\textrm{has full row rank}~&\iff~ r \leq c, \label{eq:vandermonde_row_rank_lemma} \\
			\mymat{V}~\textrm{has full column rank}~&\iff~ c \leq r.\label{eq:vandermonde_column_rank_lemma}
		\end{align}
	\end{lemma}

	\section{On approximating the Koopman eigen-values via DMD}\label{s:Appendix_ProveModeNormScreens}
	
	Assuming that \PsiLiesInNRSpanOfrDistinctKEFs\, and \ICBeSpectrallyInformative, we prove that linearly consistent and well-sampled training data produces DMD eigen-values, \(\DMDEvals\), that contain the true Koopman eigen-values, \(\KEvals\). 
	Furthermore, the latter are characterized by the corresponding DMD modes being non-zero.
	
	We begin by generalizing an existence result from \cite{pan2020structure}, where it is unduly restricted to roots of unity, to allow for arbitrary and distinct eigen-values:
	\begin{proposition}\label{prop:PanDurai_Generalizability}
		Suppose \PsiLiesInNRSpanOfrDistinctKEFs.	
		If the DMD problem is well-sampled, then, the last snapshot, \(\myvec{z}_{\undelayedDMDorder+1}\), is a linear combination of the preceding snapshots.
		\begin{displaymath}
			\undelayedDMDorder \geq \finitekissdim \implies \myvec{z}_{\undelayedDMDorder+1} \in \mathcal{R}(\mymat{X}).
		\end{displaymath}
	\end{proposition}
	\begin{proof}
		We just need to recast the arguments used in Theorem 1 of \cite{pan2020structure}, using the machinery developed so far. 
		According to \cref{eq:vandermonde_column_rank_lemma,eq:vandermonde_row_rank_lemma}, the condition \(\undelayedDMDorder \geq \finitekissdim\) ensures that the last column of \(\mymat{\Theta}\) lies in the span of \(\mymat{\Theta}_1\) i.e.,
		\begin{displaymath}
			\mymat{\Theta}\,\myvec{e}_{\undelayedDMDorder+1}		\in \mathcal{R}(\mymat{\Theta}_1).
		\end{displaymath}
		If we pre-multiply both sides with \(\mymat{C}\), then, we can apply \cref{eq:KMF__of__Z,eq:KMF__of__X_Y} to reach the desired conclusion.
	\end{proof}
	
	Next, we leverage linear consistency and well-sampling to comment on  \(\tilde{\mymat{C}}\).
	\begin{proposition}\label{prop:WellPosed_Means_LinCon_KoopmanInvariance_R_Equivalent}
		Suppose \PsiLiesInNRSpanOfrDistinctKEFs\, and \ICBeSpectrallyInformative.
		If \( (\mymat{X},\mymat{Y})\) is linearly consistent and the DMD problem is well-sampled, then, the columns of \(\tilde{\mymat{C}}\) are linearly independent.
		\begin{displaymath}
			(\mymat{X},\mymat{Y})~\textrm{is linearly consistent} ~\&~	\undelayedDMDorder \geq \finitekissdim 
			~\implies~
			\tilde{\mymat{C}}~\textrm{has full column rank}.
		\end{displaymath}
	\end{proposition}
	\begin{proof}
		Firstly, if \(\mymat{A}\) is the matrix, specified  by linear consistency of \((\mymat{X},\mymat{Y})\), that maps \(\mymat{X}\) to \(\mymat{Y}\), i.e.,
		\begin{displaymath}
			\mymat{A} \mymat{X} = \mymat{Y},
		\end{displaymath}
		then, we have the following identity:
		\begin{equation}\label{eq:LinCon_KI_proof_equation1}
			(\mymat{A} \tilde{\mymat{C}} -  \tilde{\mymat{C}} \mymat{\Lambda})~
			\diag[\KEFVector(\state_1)] \mymat{\Theta}_1 = \mymat{0}.
		\end{equation}
		We can derive this identity by using \cref{eq:KMF__of__X_Y} to expand the data matrices \(\mymat{X}\) and \(\mymat{Y}\)  as follows:
		\begin{displaymath}
			\mymat{A} \EmphasiseReduction{\mymat{X}} = \EmphasiseReduction{\mymat{Y}} \iff 	(\mymat{A} \mymat{C}) \mymat{\Theta}_1 = (\mymat{C} \mymat{\Lambda})\mymat{\Theta}_1.
		\end{displaymath}
		The dependence on the initial condition \(\state_1\) can be factored out using the definition of \(\mymat{C}\) \cref{eq:C_definition} together with the fact that \(\mymat{\Lambda}\) and \( \diag[\KEFVector(\state_1)]\) commute:
		\begin{displaymath}
			\begin{aligned}
				\mymat{A} \EmphasiseReduction{\mymat{C}} \mymat{\Theta}_1 ~&=~
				\mymat{A} \tilde{\mymat{C}} \diag[\KEFVector(\state_1)] \mymat{\Theta}_1,
				\\
				\EmphasiseReduction{\mymat{C}} \mymat{\Lambda} \mymat{\Theta}_1
				~=~ \tilde{\mymat{C}}~ \EmphasiseReduction{\diag[\KEFVector(\state_1)] \mymat{\Lambda}} \mymat{\Theta}_1
				~&=~ (\tilde{\mymat{C}} \mymat{\Lambda})~ \diag[\KEFVector(\state_1)]  \mymat{\Theta}_1.
			\end{aligned}
		\end{displaymath}
		Combining these results, we have \cref{eq:LinCon_KI_proof_equation1}.

		However, the matrix \(\diag[\KEFVector(\state_1)] \mymat{\Theta}_1\) has full row rank.
		This can be deduced by first using \(\state_1\) being spectrally informative alongside \cref{eq:C_definition} to infer that \(\diag[\KEFVector(\state_1)]\) is invertible.
		Then, we can consider the \(\finitekissdim \times \undelayedDMDorder\) Vandermonde matrix \(\mymat{\Theta}_1\).
		Since it has distinct nodes, well-sampling \((n \geq r )\) lets us use \cref{eq:vandermonde_row_rank_lemma} to conclude that \(\mymat{\Theta}_1 \) has linearly independent rows.
		Since the product of two matrices with full row rank will also have full row rank,
		we find that the rows of \(\diag[\KEFVector(\state_1)] \mymat{\Theta}_1\) are linearly independent.
		
		So, \cref{eq:LinCon_KI_proof_equation1} becomes equivalent to the following identity:
		\begin{equation}\label{eq:LinCon_KI_proof_equation2}
			\mymat{A} \tilde{\mymat{C}} = \tilde{\mymat{C}} \mymat{\Lambda}.
		\end{equation}
		
		A specific interpretation of this identity, based on the construction of \(\mymat{\Lambda}\), yields the desired conclusion.
		To see this, observe that \(\mymat{\Lambda}\) is constructed in \cref{eq:KEvalsDiagonally} as a diagonal matrix, by collecting the Koopman eigen-values, \(\KEvals\).
		Now, these eigen-values are \(\finitekissdim\) distinct complex numbers.
		Hence, the resulting interpretation of \cref{eq:LinCon_KI_proof_equation2}- the columns of \(\tilde{\mymat{C}}\) are eigen-vectors of \(\mymat{A}\) corresponding to distinct eigen-values- tells us that \(\tilde{\mymat{C}}\) has linearly independent columns.
	\end{proof}
	
	With these pieces in place, we can establish the recovery of Koopman eigenvalues via Companion DMD.
	
	\subsection{Proof of \cref{thm:Suff4CompDMD2CatchAllNKnow}}
		
	We preface the proof by establishing that \(\mymat{C}\) has full column rank.
	Firstly, Proposition \ref{prop:WellPosed_Means_LinCon_KoopmanInvariance_R_Equivalent} tells us that the columns of \(\tilde{\mymat{C}}\) are linearly independent.		
	Building on this, we need only invoke \(\state_1\) being spectrally informative alongside (\ref{eq:C_definition}) to show that \(\mymat{C}\) also has full column rank.
	
	\mypara{DMD eigenvalues contain the Koopman eigenvalues}
	The key idea here is that well-sampling leads to a system of linear equations that can be peeled away, using the \(\finitekissdim\) distinct KEFs that span our dictionary \(\dictionary\), to produce the desired conclusion.
	
	Firstly, well-sampling means the Companion matrix \(\opt{\mymat{T}}\), defined in \cref{eq:Optimal__CompanionMatrix}, `maps' \(\mymat{X}\) to \(\mymat{Y}\) as follows:
	\begin{displaymath}
		\mymat{X} \opt{\mymat{T}} ~=~ \mymat{Y}.
	\end{displaymath}
	This can be derived by considering the least-squares interpretation of \cref{eq:Optimal_1StepPredictor}:
	\begin{equation}\label{eq:CompanionDMDObjFun}
		\myvec{c}^*[\mymat{Z}] ~=~ \mymat{X}^\dagger \myvec{z}_{\undelayedDMDorder+1} ~=~
		\begin{cases}
			\arg\min_{\myvec{c}}~~~\lVert \myvec{c} \rVert_2~&\myvec{z}_{n+1}\in \mathcal{R}(\mymat{X}) \\
			\textrm{subject to}~ \mymat{X} \myvec{c} ~=~\myvec{z}_{n+1}\\ \\
			\arg\min_{\myvec{c}}~~~ \lVert \mymat{X} \myvec{c} - \myvec{z}_{n+1} \rVert_2~&\textrm{Otherwise}
		\end{cases} .
	\end{equation}
	Since  \(\undelayedDMDorder \geq \finitekissdim\), \cref{prop:PanDurai_Generalizability} tells us that we are in the first case described above. 
	Hence, 
	\begin{displaymath}
		\mymat{X} \opt{\myvec{c}}[\mymat{Z}] ~=~ \myvec{z}_{\undelayedDMDorder+1},
	\end{displaymath}
	or equivalently, by \cref{eq:Optimal__CompanionMatrix,eq:Definition__Y},
	\begin{displaymath}
		\mymat{X} \opt{\mymat{T}} ~=~ \mymat{Y}.
	\end{displaymath}
	
	Now, this identity can be unpacked using the Koopman mode factorization of \(\mymat{X}\) and \(\mymat{Y}\) to give the following eigenvector-eigenvalue equation:
	\begin{equation}\label{eq:Vandermonde_Rows_Are_Left_Evecs}
		\mymat{\Theta}_1 \opt{\mymat{T}} ~=~ \mymat{\Lambda} \mymat{\Theta}_1,
	\end{equation}
	Such a reduction can be achieved by expanding both \(\mymat{X}\) and \(\mymat{Y}\) according to \cref{eq:KMF__of__X_Y},  and then discarding \(\mymat{C}\), owing to its full column rank:
	\begin{displaymath}
		\EmphasiseReduction{\mymat{X}} \opt{\mymat{T}} - \EmphasiseReduction{\mymat{Y}} ~=~\mymat{0}
		\iff
		\EmphasiseReduction{\mymat{C}} \left(\mymat{\Theta}_1 \opt{\mymat{T}} - \mymat{\Lambda} \mymat{\Theta}_1\right) ~=~ \mymat{0}
		\iff
		\mymat{\Theta}_1 \opt{\mymat{T}} - \mymat{\Lambda} \mymat{\Theta}_1 ~=~ \mymat{0}.
	\end{displaymath}
	Since \(\mymat{\Lambda}\) is a diagonal matrix, the resulting identity
	can be viewed as an eigenvector-eigenvalue equation.
	
	Consequently, the construction of \(\mymat{\Lambda}\) from the Koopman eigenvalues, \(\KEvals\), gives the desired conclusion.
	Firstly, observe that the rows of \(\mymat{\Theta}_1\) are the left eigenvectors of \(\opt{\mymat{T}}\) with the corresponding eigenvalues on the diagonal of \(\mymat{\Lambda}\). 
	Now, \cref{eq:KEvalsDiagonally} tells us that the Koopman eigenvalues, \( \KEvals\), constitute the diagonal of \(\mymat{\Lambda}\).
	Additionally, \cref{eq:defn__DMD_Eigenvalues} defines the DMD eigenvalues, \(\DMDEvals\), to be the eigenvalues of \(\opt{\mymat{T}}\).
	Combining these two observations, we have:
	\begin{equation}\label{eq:Koopman_Spectrum_Within_DMD_Spectrum}
		\KEvals ~\subseteq~\DMDEvals.
	\end{equation}

	\mypara{A DMD eigenvalue with a non-zero DMD mode lies in the Koopman spectrum}
	Moving onto the second part of \cref{thm:Suff4CompDMD2CatchAllNKnow}, we can prove the forward implication (paraphrased above) by observing that \(\mymat{\Theta}_1 \myvec{v}_{\mu} ~\neq~\myvec{0} \). 
	To see this, start with \( \mymat{X} \myvec{v}_{\mu} ~\neq~\myvec{0} \) and sequentially apply \cref{eq:KMF__of__X_Y} followed by the full column rank of \(\mymat{C}\):
	\begin{displaymath}
		\EmphasiseReduction{\mymat{X}} \myvec{v}_{\mu} ~\neq~\myvec{0} 
		\iff \EmphasiseReduction{\mymat{C}} \mymat{\Theta}_1 \myvec{v}_{\mu}~\neq~\myvec{0}
		\iff \mymat{\Theta}_1 \myvec{v}_{\mu}~\neq~\myvec{0}.
	\end{displaymath}

	Moreover, \(\mymat{\Theta}_1 \myvec{v}_{\mu}\) solves the eigenvector-eigenvalue equation associated with \(\mymat{\Lambda}\), for an eigenvalue of \(\mu\).
	This can be obtained by	starting with the definition of \(\myvec{v}_{\mu} \),
	pre-multiplying both sides with \(\mymat{\Theta}_1\) and, then, applying \cref{eq:Vandermonde_Rows_Are_Left_Evecs}:
	\begin{displaymath}
		\opt{\mymat{T}} \myvec{v}_{\mu}~=~\mu~\myvec{v}_{\mu}
		\implies
		\EmphasiseReduction{
			\mymat{\Theta}_1~\opt{\mymat{T}} 
		}
		\myvec{v}_{\mu}~=~\mu~\mymat{\Theta}_1 ~\myvec{v}_{\mu}~\implies~
		\mymat{\Lambda} (\mymat{\Theta}_1 \myvec{v}_{\mu}) ~=~\mu~( \mymat{\Theta}_1 \myvec{v}_{\mu} ).
	\end{displaymath}
	
	Consequently, \(\mu\) is an eigenvalue of \(\mymat{\Lambda}\) and hence, by \cref{eq:KEvalsDiagonally}, also lies in the Koopman spectrum i.e., \(\mu \in \KEvals\).

	\mypara{A DMD mode that corresponds to a Koopman eigenvalue must be non-zero}
	The backward implication can be established using the duality of the left and right eigenvectors of \(\opt{\mymat{T}}\).
	Firstly, if \(\mu\) is a Koopman eigenvalue, then, \cref{eq:Koopman_Spectrum_Within_DMD_Spectrum} tells us that \(\mu\) is also a DMD eigenvalue i.e., \(\mu ~\in~\DMDEvals\).
	So, we can let \(\myvec{w}_{\mu} \) denote the\footnote{We say ``the'' and not ``a'' because companion matrices can only have a single eigenvector for an eigenvalue, regardless of its multiplicity \cite{horn2012matrix}. } normalized left eigenvector of \(\opt{\mymat{T}}\) corresponding to the eigenvalue \(\mu\):
	\begin{displaymath}
		\myvec{w}_{\mu}^T \opt{\mymat{T}}~=~ \mu\,\myvec{w}_{\mu}^T, \quad \lVert \myvec{w}_{\mu} \rVert_2 ~=~ 1.
	\end{displaymath}
	Then, by duality, we have
	\begin{equation}\label{eq:Duality_eigen-vectors_CompanionMatrix}
		\myvec{w}_{\mu}^T \myvec{v}_{\mu} ~\neq~ 0.
	\end{equation}
	Since  \(\mu \in  \KEvals \), by \cref{eq:Vandermonde_Rows_Are_Left_Evecs}, there is a row of \(\mymat{\Theta}_1 \) that is a scaled version of \(\myvec{w}_{\mu}^T \). 
	Hence, by \cref{eq:Duality_eigen-vectors_CompanionMatrix}, we have,
	\begin{displaymath}
		\mymat{\Theta}_1 \myvec{v}_{\mu} ~\neq~ \myvec{0}.
	\end{displaymath}
	In light of the full column rank of \(\mymat{C}\) and \cref{eq:KMF__of__X_Y}, we see that \( \mymat{X} \myvec{v}_{\mu} \), the DMD mode corresponding to \(\mu \), has non-zero norm.
	\begin{displaymath}
		\mymat{\Theta}_1 \myvec{v}_{\mu} ~\neq~ \myvec{0} 
		\implies
		\EmphasiseReduction{\mymat{C} \mymat{\Theta}_1} \myvec{v}_{\mu} ~\neq~ \myvec{0} 
		\implies 
		\mymat{X} \myvec{v}_{\mu} ~\neq~ \myvec{0}.
	\end{displaymath}
	
	\manualQED
	
	\section{On DMD-DFT equivalence}\label{s:Appendix_DMD-DFT}
	\begin{lemma}[A vector equality that encodes \(\muDMD \equiv \textrm{DFT}\)]\label{lem:DMDDFT_VectorEquality}
		A necessary and sufficient condition for \(\muDMD\) (\cref{alg:MeanSubtractedDMD}) to be equivalent\footnote{As described in \cref{def:DMD_DFT_Equivalence}.} to a temporal DFT is that the \(\muDMD\) model, \(\opt{\myvec{c}}[\mymat{Z}_{\rm ms}]\), coincide with \(- \myvec{1}_{\undelayedDMDorder}\).
		\begin{displaymath}
			\muDMD \equiv \textrm{DFT}
			~~\iff~~
			\opt{\myvec{c}}[\mymat{Z}_{\rm ms}] ~=~ - \myvec{1}_{\undelayedDMDorder}.
		\end{displaymath}
	\end{lemma}
	\begin{proof}
		By \cref{def:DMD_DFT_Equivalence}, we have,
		\begin{displaymath}
			\muDMD \equiv \textrm{DFT}  \iff \msubDMDEvals ~~=~~ \{ z \neq 1~|~z^{\undelayedDMDorder+1}=1 \}.
		\end{displaymath}
		Since \(\msubDMDEvals\) are defined in \cref{alg:MeanSubtractedDMD} to be the eigenvalues of the Companion matrix \(\mymat{T}\left( \opt{\myvec{c}}[\mymat{Z}_{\rm ms}]  \right)\),
		we have the desired conclusion.
	\end{proof}
	
	\subsection{Proof of \cref{thm:musubNDFT}}\label{ss:Appendix_Proof_NecessaryNSufficient_msubDMDDFTequivalence}
	
	According to \cref{lem:DMDDFT_VectorEquality}, it is sufficient to establish the following identity:
	\begin{equation}\label{eq:muDMD__is__Projected1_n}
		\myvec{c}^*[\mymat{Z}_{\rm ms}] 
		~=~ -\mathbf{1}_\undelayedDMDorder  
		~~+~~ 
		\mathcal{P}_{\mathcal{N}(\mymat{X}_{\rm ms})}\myvec{1}_\undelayedDMDorder.
	\end{equation}
	This is simply because once \cref{eq:muDMD__is__Projected1_n} has been proven, we only need to leverage \cref{lem:DMDDFT_VectorEquality} to reach the desired conclusion:
	\begin{displaymath}
		\mathcal{P}_{\mathcal{N}(\mymat{X}_{\rm ms})}\myvec{1}_\undelayedDMDorder = \myvec{0}
		~\iff~
		\opt{\myvec{c}}[\mymat{Z}_{\rm ms}] ~=~ - \myvec{1}_{\undelayedDMDorder}
		~\iff~
		\muDMD \equiv \textrm{DFT}.		
	\end{displaymath}
	
	So, we begin by observing that \(\opt{\myvec{c}}[\mymat{Z}_{\rm ms}]\) is the closest point to the origin from an affine subspace described by the mean-subtracted data:
	\begin{equation}\label{eq:MeanSubDMD}
		\begin{aligned}
			\myvec{c}^*[\mymat{Z}_{\rm ms}] 
			~=~& \arg\min_{\myvec{c}}~~~\left\lVert \myvec{c} \right\rVert_2^2 \\
			~&\textrm{subject to}~ \mymat{X}_{\rm ms} \myvec{c} ~=~\myvec{z}_{\undelayedDMDorder+1}-\myvec{\mu}.
		\end{aligned}
	\end{equation}
	To see this, recall that for any matrix \(\mymat{A}\) and vector \(\myvec{b}\) of appropriate dimensions, the expression \(\mymat{A}^\dagger \myvec{b}\) can be interpreted as the minimum norm least squares solution, when \(\myvec{b} \in \mathcal{R}(\mymat{A})\).
	If we are to apply this in the context of \(\myvec{c}^*[\mymat{Z}_{\rm ms}] \), which is constructed according to \cref{eq:Define__X_ms,eq:Z__equals__X__z_nplus1,eq:Optimal_1StepPredictor} as follows,
	\begin{equation}\label{eq:c_star_Zms__Unpacked}
		\myvec{c}^*[\mymat{Z}_{\rm ms}] ~=~ \mymat{X}_{\rm ms}^\dagger (\myvec{z}_{\undelayedDMDorder+1} - \myvec{\mu}),
	\end{equation}
	then, we need only check if \(\myvec{z}_{\undelayedDMDorder+1} - \myvec{\mu}\) lies in the column-space of \(\mymat{X}_{\rm ms}\).
	This can be verified from \cref{eq:Defn_Temporal_Mean,eq:Define__Z_ms} which say that the columns of  \(\mymat{Z}_{\rm ms}\) sum to zero:
	\begin{displaymath}
		\EmphasiseReduction{
			\mymat{Z}_{\rm ms}
		} \myvec{1}_{\undelayedDMDorder+1} 
		~=~
		\EmphasiseReduction{
			\sum_{j=1}^{\undelayedDMDorder+1} \myvec{z}_j
		} - 
		\sum_{j=1}^{\undelayedDMDorder+1} \myvec{\mu}
		~=~
		(\undelayedDMDorder + 1) \myvec{\mu} - 
		(\undelayedDMDorder + 1) \myvec{\mu}
		~=~ 
		\myvec{0}.
	\end{displaymath}
	In the context of \cref{eq:Define__X_ms}, this means,
	\begin{equation}\label{eq:negative1n__feasible_for_muDMD}
		\mymat{X}_{\rm ms} \myvec{1}_{\undelayedDMDorder} + (\myvec{z}_{\undelayedDMDorder+1} - \myvec{\mu})
		~=~ \myvec{0}.
	\end{equation}
	In other words, \(\myvec{z}_{\undelayedDMDorder+1}-\myvec{\mu}\) lies in the range-space of \(\mymat{X}_{\rm ms}\).
	Hence, \cref{eq:c_star_Zms__Unpacked} can be re-written as the minimum norm least squares problem \cref{eq:MeanSubDMD}.
	
	Re-parametrizing the affine subspace gives us the following re-phrasal of \cref{eq:MeanSubDMD}:
	\begin{equation}\label{eq:MeanSubDMD__Re_Parametrized}
		\myvec{c}^*[\mymat{Z}_{\rm ms}] 
		~=~ -\mathbf{1}_\undelayedDMDorder  ~~+~~ \left( 
		\arg\min_{\myvec{d} \in \mathcal{N}(\mymat{X}_{\rm ms})}~~~\left\lVert \mathbf{1}_\undelayedDMDorder - \myvec{d} \right\rVert_2^2
		\right) .
	\end{equation}
	This stems from the feasible set of \cref{eq:MeanSubDMD},
	\begin{equation}\label{eq:AffineSS_4_msub}
		\{\myvec{c}~|~\mymat{X}_{\rm ms} \myvec{c} ~=~\myvec{z}_{\undelayedDMDorder+1}-\myvec{\mu} \},
	\end{equation}
	being an affine subspace, which is an object that results from the set addition of a subspace and a constant vector called the offset.
	In representing an affine subspace, we can use any of its elements as the offset.
	Since \cref{eq:negative1n__feasible_for_muDMD} tells us that \(-\myvec{1}_\undelayedDMDorder\) is an element of \cref{eq:AffineSS_4_msub}, the latter may be rewritten as follows:
	\begin{displaymath}
		\{\myvec{c}~|~\exists~\myvec{d}~\in~\mathcal{N}(\mymat{X}_{\rm ms})~\textrm{such that}~\myvec{c} = -\mathbf{1}_\undelayedDMDorder + \myvec{d} \}.
	\end{displaymath}
	Using this re-parametrization in conjunction with the identity,
	\begin{displaymath}
		\left\lVert - \mathbf{1}_\undelayedDMDorder + \myvec{d} \right\rVert_2^2
		~=~
		\left\lVert \mathbf{1}_\undelayedDMDorder - \myvec{d} \right\rVert_2^2,
	\end{displaymath}
	we can recast \cref{eq:MeanSubDMD} as \cref{eq:MeanSubDMD__Re_Parametrized}.
	
	Now, the best representation of \(\myvec{1}_{\undelayedDMDorder}\) in \(\mathcal{N}(\mymat{X}_{\rm ms})\) is \(\mathcal{P}_{\mathcal{N}(\mymat{X}_{\rm ms})}\myvec{1}_\undelayedDMDorder\), the orthogonal projection of \(\myvec{1}_{\undelayedDMDorder}\) onto \(\mathcal{N}(\mymat{X}_{\rm ms})\):
	\begin{equation}\label{eq:1_n__Best_Rep_In_N_Xms}
		\mathcal{P}_{\mathcal{N}(\mymat{X}_{\rm ms})}\myvec{1}_\undelayedDMDorder
		~=~	
		\arg\min_{\myvec{d} \in \mathcal{N}(\mymat{X}_{\rm ms})}~~~\left\lVert \mathbf{1}_\undelayedDMDorder - \myvec{d} \right\rVert_2^2
	\end{equation}
	
	Feeding \cref{eq:1_n__Best_Rep_In_N_Xms} back to \cref{eq:MeanSubDMD__Re_Parametrized} gives
	\cref{eq:muDMD__is__Projected1_n}.

	\manualQED
	\begin{remark}
		\Cref{thm:musubNDFT} does not assume that the dictionary \PsiLiesInNRSpanOfrDistinctKEFs.
	\end{remark}
	\subsection{Equivalence when \(\mymat{X}_{\rm ms}\) has linearly dependent columns}\label{ss:Appendix_Proof_Counterexamples_4_Hirsch_numerics_interpretation}
	In order to prove  \cref{thm:XmsnotLI_DMD_DFTequiv}, we establish the following algebraic result:
	\begin{lemma}\label{lem:Simplify_1_Sans_EasyMean}
		\begin{displaymath}
			\mathbf{1}_r~-~\left(\frac{1}{n+1}\right)\mymat{\Theta} \mathbf{1}_{n+1} ~=~ 
			-\left(\frac{1}{n+1} \right)(\mymat{\Lambda} - \mymat{I}) \mymat{\Theta}_1 
			\begin{bmatrix}
				n \\ n-1\\ \vdots \\ 1
			\end{bmatrix}.
		\end{displaymath}
	\end{lemma}
	\begin{proof}
		Consider the \(i\)-th component of the vector on the left:
		\begin{displaymath}
			1 - \left(\frac{1}{n+1}\right) ~
			\begin{bmatrix}
				1 & \lambda_i & \lambda_i^2 & \hdots & \lambda_i^\undelayedDMDorder
			\end{bmatrix}
			~ \mathbf{1}_{n+1}.
		\end{displaymath}
		Expanding the inner product and factoring out the denominator, we get:
		\begin{displaymath}
			\frac{1}{n+1} \left( (n+1) - (1+\lambda_i + \lambda_i^2 +\hdots+ \lambda_i^n) \right).
		\end{displaymath}
		The two sums in the numerator have \( (n+1)\) terms each.
		By summing them in pairs, we obtain:
		\begin{displaymath}
			-\frac{1}{n+1} \left( (\lambda_i-1) + (\lambda_i^2-1) + \hdots + (\lambda_i^n - 1 ) \right).
		\end{displaymath}
		By factoring out \((\lambda_i-1)\) from each of the remaining \(n\) terms, we get the following:
		\begin{displaymath}
			- \frac{\lambda_i-1}{n+1}
			\begin{bmatrix}
				1&\lambda_i&...&\lambda_i^{n-1}
			\end{bmatrix}
			\begin{bmatrix}
				n \\ n-1\\ \vdots \\ 1
			\end{bmatrix}.
		\end{displaymath}
		Collect the components into a column vector, and group terms to reach the desired conclusion.
	\end{proof}
	
	\subsubsection{Proof of \cref{thm:XmsnotLI_DMD_DFTequiv}}

	We will begin by explicitly constructing the dictionary alluded to and, then, show that it indeed leads to DMD-DFT equivalence despite \(\mymat{X}_{\rm ms}\) having linearly dependent columns.
	
	\mypara{Constructing the counter-example}
	
	Let \(\myvec{l}\) denote the last column of \(\mymat{\Theta}_1\) and \(\myvec{l}^\perp\) be the projection of \(\myvec{l}\) orthogonal to the span of the first \(\undelayedDMDorder-1\) columns of \(\mymat{\Theta}_1\). 
	\begin{equation}\label{eq:Define__l__l_perp}
		\myvec{l}~:=~
		\begin{bmatrix}
			\lambda_1^{\undelayedDMDorder-1} \\ \lambda_2^{\undelayedDMDorder-1} \\ \vdots \\ \lambda_r^{\undelayedDMDorder-1}
		\end{bmatrix},\quad
		\myvec{l}^\perp~:=~ \mathcal{P}_{\mathcal{N}(\mymat{\Theta}_2^{H})} \myvec{l}.
	\end{equation}
	Here, we have used \cref{eq:Vandermonde_and_submatrices} to denote the sub-matrix formed by the first \(n-1\) columns of \(\mymat{\Theta}\) as \(\mymat{\Theta}_2\). 
	Now,  we can define  the dictionary, \(\dictionary\), as follows:
	\begin{equation}\label{eq:Define__TheThornyDictionary}
		\tilde{\mymat{C}}_0~:=~(\myvec{l}^\perp)^{H}(\mymat{\Lambda} - \mymat{I})^{-1}, \quad \tilde{\mymat{C}} ~=~ \tilde{\mymat{C}}_0 \diag[\KEFVector(\state_1)]^{-1}, \quad \dictionary ~=~ \tilde{\mymat{C}} \myvec{\phi}.  
	\end{equation}
	
	\mypara{Validation}
	
	We begin with a quick sanity check that our observables are indeed non-trivial.
	By \cref{eq:vandermonde_column_rank_lemma}, \(n \leq r \) ensures that \(\mymat{\Theta}_1 \in \mathbb{C}^{r \times n}\) has full column rank.
	So, \(\myvec{l}\), the last column of \(\mymat{\Theta}_1\), will not be in the span of the remaining columns.
	Hence, \(\myvec{l}^\perp \neq \myvec{0}\) and our observables, represented by \(\tilde{\mymat{C}}\), are non-trivial.

	Our objective is to show that \(\mymat{X}_{\rm ms}\) has linearly dependent columns and, simultaneously, satisfies \(\mathcal{P}_{\mathcal{N}(\mymat{X}_{\rm ms})} \myvec{1}_n = \myvec{0} \).
	
	\mysubpara{Columns of \(\mymat{X}_{\rm ms}\) are linearly dependent}
	Rank defectiveness of \(\mymat{X}_{\rm ms}\) follows directly from its shape.
	The identities \cref{eq:Zmat4CompDMD,eq:Define__Z_ms,eq:Define__X_ms}, when taken together, tell us that \(\mymat{X}_{\rm ms}\) is of size \(\dictionarylength \times \undelayedDMDorder\).
	Since \(\tilde{\mymat{C}}\) is a row vector, \(\dictionarylength = 1\) and, hence,  \( \mymat{X}_{\rm ms}\) is also a row vector.
	The condition \(\undelayedDMDorder \geq 2 \) means \( \mymat{X}_{\rm ms}\) has at least 2 columns and these must be linearly dependent as it only has a single row.
	
	\mysubpara{\(\muDMD\) is equivalent to a DFT}
	We can prove DMD-DFT equivalence by showing that \(\mathcal{N}(\mymat{X}_{\rm ms})\) coincides with \(\mathcal{R}(\mymat{B})\), where \( \mymat{B} \in \mathbb{C}^{\undelayedDMDorder \times (\undelayedDMDorder-1)} \) is the banded matrix defined thus:
	\begin{displaymath}
		\mymat{B}~:=~\begin{bmatrix}
			-1 &   &      \\
			1 & \ddots  &      \\
			& \ddots  &-1     \\
			&   &  1   
		\end{bmatrix}.
	\end{displaymath}
	Evidently, \( \mymat{B}^H \myvec{1}_n = \myvec{0}\). 
	Indeed, the columns of \(\mymat{B} \) form a basis for the orthogonal complement of the span of \(\myvec{1}_n\).
	\begin{displaymath}
		\mathcal{R}(\mymat{B})~=~\mathcal{N}(\myvec{1}_n^H).
	\end{displaymath}
	If we can show that \(\mathcal{R}(\mymat{B}) = \mathcal{N}(\mymat{X}_{\rm ms})\), then, \cref{thm:musubNDFT} tells us that we will have DMD-DFT equivalence:
	\begin{displaymath}
		\mathcal{R}(\mymat{B}) = \mathcal{N}(\mymat{X}_{\rm ms})
		\implies \mathcal{P}_{\mathcal{N}(\mymat{X}_{\rm ms})} \myvec{1}_n 
		~=~ \mathcal{P}_{\mathcal{R}(\mymat{B})} \myvec{1}_n 
		~=~  \mathcal{P}_{\mathcal{N}(\myvec{1}_n^H)} \myvec{1}_n
		~=~ \myvec{0}.
	\end{displaymath}
	So, we have an alternative approach to proving DMD-DFT equivalence.
	
	Before proceeding further, we perform some algebraic simplifications.
	The specific choice of \(\tilde{\mymat{C}}\) in \cref{eq:Define__TheThornyDictionary}, when applied alongside \cref{eq:C_definition,eq:KMF__of__Z}, gives,
	\begin{displaymath}
		\mymat{Z}= \tilde{\mymat{C}}_0 \mymat{\Theta},
	\end{displaymath}
	which in turn simplifies \(\myvec{\mu}\) thus:
	\begin{equation}\label{eq:Counterexample_Compact_mu}
		\myvec{\mu} = \tilde{\mymat{C}}_0 (\mymat{\Theta} \mathbf{1}_{n+1}) \left( \frac{1}{n+1} \right).
	\end{equation}
	Consequently, \(\mymat{Z}_{\rm ms} \) and \(\mymat{X}_{\rm ms}\) may be factored as follows:
	\begin{displaymath}
		\mymat{Z}_{\rm ms} = 
		\begin{bmatrix}
			- \myvec{\mu} & \tilde{\mymat{C}}_0
		\end{bmatrix}
		\begin{bmatrix}
			\mathbf{1}_{n+1}^{T}\\ \mymat{\Theta} 
		\end{bmatrix}
		, \quad 
		\mymat{X}_{\rm ms} = 
		\begin{bmatrix}
			- \myvec{\mu} & \tilde{\mymat{C}}_0
		\end{bmatrix}
		\begin{bmatrix}
			\mathbf{1}_{n}^{T}\\ \mymat{\Theta}_1
		\end{bmatrix}.
	\end{displaymath}
	
	Now, we can show that \(\mathcal{R}(\mymat{B}) \subseteq \mathcal{N}(\mymat{X}_{\rm ms}) \).
	Consider the product \(\mymat{X}_{\rm ms} \mymat{B} \) in light of the columns of \(\mymat{B}\) being orthogonal to \(\myvec{1}_n\).
	\begin{displaymath}
		\mymat{X}_{\rm ms} \mymat{B}
		~=~ 
		\begin{bmatrix}
			- \myvec{\mu} & \tilde{\mymat{C}}_0
		\end{bmatrix}
		\begin{bmatrix}
			\mathbf{1}_{n}^{T}\\ \mymat{\Theta}_1
		\end{bmatrix}
		\mymat{B}
		~=~\tilde{\mymat{C}}_0 \mymat{\Theta}_1 \mymat{B}.
	\end{displaymath}
	The matrix product \( \mymat{\Theta}_1 \mymat{B} \) may be re-written as follows with a Vandermonde matrix to the right.
	\begin{displaymath}
		\begin{aligned}
			\mymat{\Theta}_1 \mymat{B} &=
			\begin{bmatrix}
				\lambda_1 -1 & \lambda_1^2 - \lambda_1 & \hdots& \lambda_1^{n-1} - \lambda_1^{n-2} \\
				\lambda_2 -1 & \lambda_2^2 - \lambda_2 & \hdots& \lambda_2^{n-1} - \lambda_2^{n-2} \\
				\vdots & \vdots & \ddots & \vdots \\
				\lambda_r -1 & \lambda_r^2 - \lambda_r & \hdots& \lambda_r^{n-1} - \lambda_r^{n-2} \\
			\end{bmatrix}\\
			&=
			\begin{bmatrix}
				\lambda_1 - 1 & & & \\
				& \lambda_2-1 & &  \\
				& & \ddots &  \\
				& & & \lambda_r-1
			\end{bmatrix}
			\begin{bmatrix}
				1 & \lambda_1 & \hdots & \lambda_1^{n-2} \\
				1 & \lambda_2 & \hdots & \lambda_2^{n-2} \\
				\vdots & \vdots & \ddots & \vdots \\
				1 & \lambda_r & \hdots & \lambda_r^{n-2} \\
			\end{bmatrix}\\
			&= (\mymat{\Lambda} - \mymat{I}) \mymat{\Theta}_2.
		\end{aligned}
	\end{displaymath}
	This lets us use \cref{eq:Define__TheThornyDictionary} to bring in \(\myvec{l}^\perp\), whose definition in \cref{eq:Define__l__l_perp} completes this line of thought:
	\begin{displaymath}
		\tilde{\mymat{C}}_0 \EmphasiseReduction{\mymat{\Theta}_1 \mymat{B}}
		~=~ \EmphasiseReduction{\tilde{\mymat{C}}_0 (\mymat{\Lambda} - \mymat{I})} \mymat{\Theta}_2 
		~=~ \EmphasiseReduction{(\myvec{l}^\perp)^{H} \mymat{\Theta}_2} 
		~=~ \myvec{0} \implies
		\mathcal{R}(\mymat{B}) \subseteq \mathcal{N}(\mymat{X}_{\rm ms}) .
	\end{displaymath}    
	
	Now, to show that  \(\mathcal{R}(\mymat{B}) = \mathcal{N}(\mymat{X}_{\rm ms}) \), we need only prove that \(\mymat{X}_{\rm ms} \neq \mymat{0}\).
	To see this, observe that \(\mathcal{R}(\mymat{B})\) is a \(n-1\) dimensional subspace of \(\mathbb{C}^n\).
	So, if \(\mathcal{R}(\mymat{B})\) is a strict subset of \(\mathcal{N}(\mymat{X}_{\rm ms})\), then, by the rank-nullity theorem, \(\mymat{X}_{\rm ms} = \mymat{0}\).
	Therefore, by contraposition, when \(\mymat{X}_{\rm ms}\) is non-zero, \(\mathcal{R}(\mymat{B})\) coincides with \(\mathcal{N}(\mymat{X}_{\rm ms})\).
	
	Presently, the first column of \(\mymat{X}_{\rm ms}\) is non-zero.
	To deduce this, we can start by unpacking said column of \(\mymat{X}_{\rm ms}\):
	\begin{displaymath}
		\mymat{X}_{\rm ms} \myvec{e}_1 
		~=~ 
		\begin{bmatrix}
			- \myvec{\mu} & \tilde{\mymat{C}}_0
		\end{bmatrix}
		\begin{bmatrix}
			\mathbf{1}_{n}^{T}\\ \mymat{\Theta}_1
		\end{bmatrix} \myvec{e}_1 
		~=~ -\myvec{\mu} + \tilde{\mymat{C}}_0 \mathbf{1}_r .
	\end{displaymath}
	If we expand \(\myvec{\mu} \) as in \cref{eq:Counterexample_Compact_mu}, it becomes amenable to apply \cref{lem:Simplify_1_Sans_EasyMean}:
	\begin{displaymath}
		\begin{aligned}
			-\EmphasiseReduction{\myvec{\mu}} + \tilde{\mymat{C}}_0 \mathbf{1}_r ~&=~ 
			\tilde{\mymat{C}}_0 
			\EmphasiseReduction{
				\left[\mathbf{1}_r - (\mymat{\Theta} \mathbf{1}_{n+1})\left(\frac{1}{n+1}\right)\right]
			}\\
			~&=~ \tilde{\mymat{C}}_0 (\mymat{\Lambda} - \mymat{I}) \mymat{\Theta}_1 
			\begin{bmatrix}
				n \\ n-1\\ \vdots \\ 1
			\end{bmatrix}
			\left(\frac{-1}{n+1}\right).
		\end{aligned}        
	\end{displaymath}
	Once again, we can use \cref{eq:Define__TheThornyDictionary,eq:Define__l__l_perp} to introduce \(\myvec{l}^\perp\) and leverage its orthogonality to the first \(n-1\) columns of  \(\mymat{\Theta}_1\):
	\begin{displaymath}
		\begin{aligned}
			\EmphasiseReduction{	\tilde{\mymat{C}}_0 (\mymat{\Lambda} - \mymat{I})} \mymat{\Theta}_1 
			\begin{bmatrix}
				n \\ n-1\\ \vdots \\ 1
			\end{bmatrix}
			\left(\frac{-1}{n+1}\right) 
			~=~& 
			\EmphasiseReduction{
				(\myvec{l}^\perp)^{H} \mymat{\Theta}_1 
				\begin{bmatrix}
					n \\ n-1\\ \vdots \\ 1
				\end{bmatrix}
			}
			\left(\frac{-1}{n+1}\right) \\
			~=~ & \EmphasiseReduction{(\myvec{l}^\perp)^{H} \myvec{l}} \left(\frac{-1}{n+1}\right)  \\
			~=~ & \lVert \myvec{l}^\perp \rVert^2 \left(\frac{-1}{n+1}\right) .
		\end{aligned}
	\end{displaymath}
	Since we've already seen that \(\myvec{l}^\perp \neq \myvec{0} \), the matrix-vector product \(\mymat{X}_{\rm ms} \myvec{e}_1 \neq \myvec{0} \).
	
	Therefore, we have DMD-DFT equivalence.
	
	\manualQED
	
	\subsection{A modal decomposition of mean-subtracted data}\label{ss:Appendix_Zms_Cms_Thetams}
	When the Koopman mode expansion possesses only a finite number of terms, \cref{lem:KMF_of_Z} tells us that the time series \(\mymat{Z}\) can be written as the product of \(\mymat{C} \) and \(\mymat{\Theta}\).
	Here, we derive an analogous factorization for the mean-removed data-set \(\mymat{Z}_{\rm ms}\).
	
	\begin{definition}[A conformal partitioning of \(\mymat{C}\) and \(\mymat{\Theta}\)]\label{defn:msub__illustrating__partition}
		Suppose \PsiLiesInNRSpanOfrDistinctKEFs\, and \ICBeSpectrallyInformative.	
		Additionally, let \(1\) be a Koopman eigenvalue i.e., \(1 \in \KEvals.\)
		Then, \cref{eq:KMF__of__Z} motivates defining \(\myvec{c}_1\), \(\mymat{C}_{\not \owns 1}\) and \(\mymat{\Theta}_{\not \owns 1} \) via the following conformal partitioning:
		\begin{equation}
			\mymat{C} ~=:~ \begin{bmatrix}
				\myvec{c}_1	& \mymat{C}_{\not \owns 1}
			\end{bmatrix}, \quad
			\mymat{\Theta} ~=:~ 
			\begin{bmatrix}
				\myvec{1}_{\undelayedDMDorder+1}^T \\
				\mymat{\Theta}_{\not \owns 1}
			\end{bmatrix}.
		\end{equation}
		To paraphrase, \(\myvec{c}_1\) is the column of \(\mymat{C}\) that multiplies \(\myvec{1}_{n+1}^T\) in the product \(\mymat{C} \mymat{\Theta}\), \(\mymat{C}_{\not \owns 1}\) is the sub-matrix of \(\mymat{C}\) obtained by excluding \(\myvec{c}_1\) and
		\(\mymat{\Theta}_{\not \owns 1} \) is the sub-matrix of \(\mymat{\Theta} \) produced by deleting the row corresponding to \(\myvec{1}_{n+1}^T\). 
	\end{definition}

	\begin{table}[tbhp]\label{tab:Constructing__C_ms_Theta_ms}
		\caption{\underline{Constructing \(\mymat{C}_{\rm ms}\) and \(\mymat{\Theta}_{\rm ms}\):} There are four possible scenarios in mean-subtraction, depending on the presence of a Koopman eigenvalue at 1 and the value of \(\myvec{\mu}\).
			The peculiarities of each case are incorporated in the definitions of \(\mymat{C}_{\rm ms}\) and \(\mymat{\Theta}_{\rm ms}\) \emph{(last two columns)}, with an eye towards inducing an analogue of \cref{eq:KMF__of__Z} for \(\mymat{Z}_{\rm ms}\).   }
		\begin{center}
			\begin{tabular}{|c|c|c||c|c|}
				\hline
				Case & \multicolumn{2}{c||}{Attributes defining the case} & \multicolumn{2}{c|}{Case properties}\\\cline{2-5}
				& Spectral property & Mean value  & \(\mymat{C}_{\rm ms}\) & \(\mymat{\Theta}_{\rm ms}\) \\ \hline
				A  & \(1 \notin \KEvals\) & \(\myvec{\mu} = \myvec{0} \) & 	\(\mymat{C} \)& \(\mymat{\Theta}\) \\[10 pt]
				B & \(''\) & \(\myvec{\mu} \neq \myvec{0}\) 
				& 	\(
				\begin{bmatrix}
					\mymat{C} & -\myvec{\mu}
				\end{bmatrix}
				\)
				&  \(
				\begin{bmatrix}
					\mymat{\Theta} \\ \mathbf{1}_{n+1}^{T}
				\end{bmatrix} 
				\)\\[10 pt]
				\hline
				C  & \(1 \in \KEvals\)  & \(\myvec{\mu} = \myvec{c}_1 \) & 	\(\mymat{C}_{\not \owns 1}\) & \( \mymat{\Theta}_{\not \owns 1}\)\\[10 pt]
				D  & \(''\) & \( \myvec{\mu} \neq \myvec{c}_1 \) & 
				\(
				\begin{bmatrix}
					\myvec{c}_1 - \myvec{\mu} & \mymat{C}_{\not \owns 1}
				\end{bmatrix}
				\) & \(
				\begin{bmatrix}
					\myvec{1}_{n+1}^T \\
					\mymat{\Theta}_{\not \owns 1}
				\end{bmatrix}
				\) \\[10 pt]
				\hline
			\end{tabular}
		\end{center}
	\end{table}
	
	\begin{lemma}\label{lem:KMF__of__Z_ms}
		Suppose \PsiLiesInNRSpanOfrDistinctKEFs\, and \ICBeSpectrallyInformative.
		
		If we construct the two matrices \(\mymat{C}_{\rm ms}\) and  \(\mymat{\Theta}_{\rm ms}\) as described by \cref{defn:msub__illustrating__partition} and  \cref{tab:Constructing__C_ms_Theta_ms}, then, every column of \(\mymat{C}_{\rm ms}\) is non-zero, \(\mymat{\Theta}_{\rm ms}\) has distinct nodes and, more importantly,
		\begin{equation}\label{eq:KMF__of__Z_ms}
			\mymat{Z}_{\rm \rm ms} = \mymat{C}_{\rm ms} \mymat{\Theta}_{\rm ms}.
		\end{equation}
		In other words, every row of \(\mymat{Z}_{\rm \rm ms}\) lies in the row space of the Vandermonde matrix \(\mymat{\Theta}_{\rm ms}\), which is related to but potentially different from \(\mymat{\Theta}\).
	\end{lemma}
	\begin{proof}
		A direct application of \cref{lem:KMF_of_Z} to \cref{eq:Define__Z_ms}, in the context of \cref{defn:msub__illustrating__partition} and \cref{tab:Constructing__C_ms_Theta_ms}.
	\end{proof}
	\begin{remark}\label{rem:Submatrices__of__Theta_ms}
		We can extend the parallel between \cref{eq:KMF__of__Z} and \cref{eq:KMF__of__Z_ms} by defining \((\mymat{\Theta}_{\rm ms})_j\) to be the sub-matrix of \(\mymat{\Theta}_{\rm ms}\) produced by deleting the last \(j\) columns.
		Consequently, \(\mymat{X}_{\rm ms}\) (defined in \cref{eq:Define__X_ms}) has the following Koopman mode factorization:
		\begin{equation}\label{eq:KMF__of__X_ms}
			\mymat{X}_{\rm ms} ~=~ \mymat{C}_{\rm ms}\, (\mymat{\Theta}_{\rm ms})_1.
		\end{equation}
	\end{remark}
	\begin{remark}
		The analogy between \cref{eq:KMF__of__Z} and \cref{eq:KMF__of__Z_ms} suggests that the nodes of \(\mymat{\Theta}_{\rm ms}\) will be as central to the upcoming analysis, as the nodes of \(\mymat{\Theta}\) (namely the Koopman eigenvalues \(\KEvals\)) were in the justification of \cref{thm:Suff4CompDMD2CatchAllNKnow}.
		Hence, we formally recognize them thus:
		\begin{displaymath}
			\sigma(\mymat{\Theta}_{\rm ms}) ~:=~\textrm{Nodes of the Vandermonde matrix}~ \mymat{\Theta}_{\rm ms}.
		\end{displaymath}
	\end{remark}

	\subsection{Over-sampling and linear consistency prevent equivalence}\label{ss:Appendix_Proof_None_If_Oversampled_and_LinCon}
	We preface the justification of \cref{thm:DMDDFT_Oversampled} by partitioning the outcome of mean subtraction into four scenarios, depending on the presence of 1 in the ``mean-subtracted Koopman spectrum'' \(\sigma(\mymat{\Theta}_{\rm ms})\) and the Koopman spectrum \(\KEvals\) (\Cref{tab:4_possibilities_of_msub}).
	This classification helps develop auxiliary guarantees on the temporal mean  and \(\mymat{C}_{\rm ms}\) which simplify the proof of  \cref{thm:DMDDFT_Oversampled}.

	\begin{table}[tbhp]\label{tab:4_possibilities_of_msub}
		\caption{\underline{Delineating the effect of mean subtraction:} There are four possible outcomes of mean-removal, depending on the presence (or absence) of \(1\) in the nodes of \(\mymat{\Theta}_{\rm ms} \)  and in the Koopman eigen-values. 
			For each case, we detail \(\sigma(\mymat{\Theta}_{\rm ms})\) alongside \(\mymat{C}_{\rm ms}\) and \(\mymat{\Theta}_{\rm ms}\) (whenever they have a simple expression). }
		\begin{center}
			\begin{tabular}{|c|c|c||c|c|c|}
				\hline 
				Case & \multicolumn{2}{c||}{Attributes defining the case} & \multicolumn{3}{c|}{Case properties}\\\cline{2-6}
				& Effect of removing \(\myvec{\mu}\)  & Spectral property &\( \sigma(\mymat{\Theta}_{\rm ms}) \) &  \(\mymat{C}_{\rm ms}\) & \(\mymat{\Theta}_{\rm ms}\) \\ \hline
				
				I & \(1 \in\sigma(\mymat{\Theta}_{\rm ms})\) & \(1 \notin \KEvals\) &  \(\KEvals \cup \{ 1 \}\) & &  \\
				
				II & '' & \(1 \in \KEvals\) & \(\KEvals\)  & & \(\mymat{\Theta}\) \\
				
				\hline
				III & \(1 \notin \sigma(\mymat{\Theta}_{\rm ms})\) &\(1 \notin \KEvals\) & \(\KEvals\)   & \(\mymat{C}\)& \(\mymat{\Theta}\)\\
				
				IV & '' & \(1 \in \KEvals\) & \(\KEvals \cap \{ 1 \}^c \)   & \( \mymat{C}_{\not \owns 1} \)& \\
				\hline
			\end{tabular}
		\end{center}
	\end{table}

	\subsubsection{Auxiliary guarantees}
	
	Firstly, we show that a Koopman invariant dictionary ensures the temporal mean, \(\myvec{\mu}\), is zero if and only if we are in Case III.
	
	\begin{lemma}\label{lem:Zero_mu_iff_Case_III}
		Suppose \PsiLiesInNRSpanOfrDistinctKEFs\, and \ICBeSpectrallyInformative.
		
		If \(\tilde{\mymat{C}}\) has full column rank, then, the mean \(\myvec{\mu}\) is zero only in Case III of \cref{tab:4_possibilities_of_msub}.
		\begin{displaymath}
			\myvec{\mu}=\myvec{0} ~\iff~\textrm{Case III}.
		\end{displaymath}
	\end{lemma}
	\paragraph{Proof}
	The four cases described in \cref{tab:4_possibilities_of_msub} are mutually exclusive and enumerate all the possible outcomes of mean subtraction.
	Hence, our objective boils down to establishing that the temporal mean, \(\myvec{\mu}\), is
	\begin{enumerate}
		\item non-zero in Cases I, II and IV.
		\item zero in Case III.
	\end{enumerate}
	We will proceed by contradiction for Cases I and III.
	The other two cases, II and IV, can be established directly.

	\mysubpara{Case I}
	According to \cref{tab:4_possibilities_of_msub}, we have \(1 \in \sigma(\mymat{\Theta}_{\rm ms})\) and \(1 \notin \KEvals\).
	We will show that the latter condition is contradicted if we have \(\myvec{\mu} = \myvec{0}\).
	
	Suppose the temporal mean is zero.
	Then, the additional condition of \(1 \notin \KEvals\) puts us in Case A of \cref{tab:Constructing__C_ms_Theta_ms}.
	Perusing the corresponding value of \(\mymat{\Theta}_{\rm ms}\), we see that it coincides with \(\mymat{\Theta}\) i.e., 
	\begin{displaymath}
		\myvec{\mu} ~=~\myvec{0} \implies \mymat{\Theta}_{\rm ms} = \mymat{\Theta}.
	\end{displaymath}
	Now, if we draw on the effect of mean-subtraction in Case I of \cref{tab:4_possibilities_of_msub}, we can see that \(1\) is a node of \(\mymat{\Theta}_{\rm ms}\) and, hence, also of \(\mymat{\Theta}\):
	\begin{displaymath}
		\mymat{\Theta}_{\rm ms} = \mymat{\Theta} \implies 1 \in \sigma(\mymat{\Theta}).
	\end{displaymath}
	However, the simultaneous application of \cref{eq:defn_Theta,def:Vandermonde_Nodes} tells us that the nodes of \(\mymat{\Theta}\) are the Koopman eigenvalues \(\KEvals\).
	Therefore, we have,
	\begin{displaymath}
		1 \in \sigma(\mymat{\Theta}) \implies 1 \in \KEvals.
	\end{displaymath}
	This contradicts the spectral property of Case I, as listed in \cref{tab:4_possibilities_of_msub}.
	
	\mysubpara{Cases II and IV}
	According to \cref{tab:4_possibilities_of_msub}, the property shared by Cases II and IV is that \(1 \) is a Koopman eigenvalue i.e., \(1 \in \KEvals \). 
	We will see that this property, in conjunction with the full column rank of \(\tilde{\mymat{C}}\), ensures a non-zero mean i.e., \(\myvec{\mu} \neq \myvec{0}. \)

	To begin with, note that our dictionary lies in the non-redundant span of \(\finitekissdim\) distinct Koopman eigenfunctions.
	So, we can use \cref{eq:KMF__of__Z} to expand the time-average \(\myvec{\mu}\), which is defined in \cref{eq:Defn_Temporal_Mean}, as follows:
	\begin{displaymath}
		\myvec{\mu} ~=~ \frac{1}{\undelayedDMDorder+1} \mymat{C} \mymat{\Theta} \myvec{1}_{\undelayedDMDorder+1}.
	\end{displaymath}
	Now, \(1 \in \KEvals\) ensures that \(\mymat{\Theta} \mathbf{1}_{n+1} \) has at least one non-zero coefficient and, hence, is non-zero:
	\begin{displaymath}
		1 \in \KEvals \implies \mymat{\Theta} \myvec{1}_{\undelayedDMDorder+1} \neq \myvec{0}.
	\end{displaymath}
	Furthermore, the spectrally informative initial condition \(\state_1\) lets the full column rank of \(\tilde{\mymat{C}}\) be inherited by \(\mymat{C}\).
	\begin{displaymath}
		\state_1~\textrm{is spectrally informative}~\implies \mymat{C}~\textrm{has full column rank.}
	\end{displaymath}
	Combining these two inferences, we get \(\myvec{\mu} \neq \myvec{0}\):
	\begin{displaymath}
		\mymat{C}~\textrm{has full column rank}~~\&~~\mymat{\Theta} \myvec{1}_{\undelayedDMDorder+1} \neq \myvec{0}~\implies \mymat{C} \mymat{\Theta} \myvec{1}_{\undelayedDMDorder+1} \neq \myvec{0} \implies \myvec{\mu} \neq \myvec{0}.
	\end{displaymath}
	
	\mysubpara{Case III}
	Here, \cref{tab:4_possibilities_of_msub} tells us that \(1 \notin \sigma(\mymat{\Theta}_{\rm ms})\) and \(1 \notin \KEvals\).
	We will show that the former is contradicted if the mean is non-zero.

	Suppose \(\myvec{\mu} \neq \myvec{0}\).
	Taken together with the condition \(1 \notin \KEvals\), we find ourselves in Case B of \cref{tab:Constructing__C_ms_Theta_ms}.
	Perusing the corresponding value of \(\mymat{\Theta}_{\rm ms} \), we see that \(1 \in \sigma(\mymat{\Theta}_{\rm ms})\), which is the desired contradiction.

	\manualQED
	
	Now, we proceed to connect the range of \(\mymat{C}_{\rm ms}\)  to those of \(\mymat{C}\) and \(\mymat{C}_{\not \owns 1}\).
	The definition of  \(\mymat{C}_{\rm ms}\) in \cref{eq:KMF__of__Z_ms} belies the intricacies of mean subtraction that are detailed in \cref{tab:4_possibilities_of_msub}. 
	Starting with the matrix \(\mymat{C}\), it is possible to gain (Case I) or lose (Case IV) a column.
	An existing column may be modified (Case II) or there may no effect whatsoever (Case III).
	Fortunately, we need only understand the relationship between the column spaces of \(\mymat{C}_{\rm ms},~\mymat{C}\) and their sub-matrices.
	\begin{proposition}\label{prop:C_without_static_mode}
		Suppose \PsiLiesInNRSpanOfrDistinctKEFs\, and \ICBeSpectrallyInformative.
		Then,	
		\begin{displaymath}
			\begin{aligned}
				\myvec{\mu} = \myvec{0} &\implies \mymat{C}_{\rm ms} = \mymat{C}. \\
				\myvec{\mu} \neq \myvec{0} &\implies \mathcal{R}(\mymat{C}_{\rm ms}) = \mathcal{R}(\mymat{C}_{\not \owns 1}) .
			\end{aligned}
		\end{displaymath}
		Additionally, if we define \((\mymat{C}_{\rm ms})_{\not \owns 1} \) to be the analogue of \(\mymat{C}_{\not \owns 1}\) for \(\mymat{C}_{\rm ms}\)\footnote{Specifically, \((\mymat{C}_{\rm ms})_{\not \owns 1} \) is taken to be the sub-matrix of \(\mymat{C}_{\rm ms}\) formed by excluding the column (if it exists) corresponding to the eigenvalue at 1}, we also get:
		\begin{displaymath}
			\mathcal{R}\left( (\mymat{C}_{\rm ms})_{\not \owns 1} \right) 
			= \mathcal{R}
			\left( 
			\mymat{C}_{\not \owns 1}  
			\right).
		\end{displaymath}
	\end{proposition}
	\begin{proof}
		Follows from the definitions of \(\mymat{C}\) \cref{eq:C_definition} and \(\mymat{C}_{\rm ms}\) (\cref{tab:Constructing__C_ms_Theta_ms}), with routine algebraic manipulations.
	\end{proof}

	\subsubsection{Proof of \cref{thm:DMDDFT_Oversampled}}
	We are given a dictionary \((\dictionary)\) that lies in the non-redundant span of \(\finitekissdim\) distinct KEFs and a spectrally informative initial condition \((\state_1)\).
	The resulting training data \((\mymat{X}, \mymat{Y})\) is linearly consistent and over-sampled.
	Under these conditions, we need to show that mean-subtracted DMD is not a temporal DFT.
	
	We can preface the justification by observing that over-sampling implies well-sampling.
	So, we can apply \cref{prop:WellPosed_Means_LinCon_KoopmanInvariance_R_Equivalent} to deduce that \(\tilde{\mymat{C}}\) has linearly independent columns.
	Now, we can proceed to a case-by-case analysis.

	\mypara{Case I}
	Here, \cref{tab:4_possibilities_of_msub} tells us that \(1\) is a node of \(\mymat{\Theta}_{\rm ms} \) i.e., \(1 \in \sigma(\mymat{\Theta}_{\rm ms})\).
	We will establish non-equivalence of \(\muDMD\) and DFT by contradiction.
	
	Suppose mean-subtracted DMD is equivalent to DFT.
	Then, the vector \(\myvec{1}_{\undelayedDMDorder}^H\) lies in the row space of \((\mymat{\Theta}_{\rm ms})_1\).
	This can be seen by first using \cref{thm:musubNDFT} to re-cast the equivalence of \(\muDMD\) and DFT as the following identity:
	\begin{displaymath}
		\mathcal{P}_{\mathcal{N}(\mymat{X}_{\rm ms})} \myvec{1}_n~=~\myvec{0}.
	\end{displaymath}
	Since \( \mathcal{N}(\mymat{X}_{\rm ms})^\perp = \mathcal{R}(\mymat{X}_{\rm ms}^{H})\), an equivalent statement would be:
	\begin{displaymath}
		\exists~ \myvec{d} ~\textrm{such that}~ \myvec{d}^{H}\mymat{X}_{\rm ms} = \mathbf{1}_n^{H}.
	\end{displaymath}
	Using \cref{eq:KMF__of__X_ms}, we then get:
	\begin{equation}\label{eq:DMD_DFT_Equivalence_via_rows_of_Xms}
		\exists~ \myvec{d} ~\textrm{such that}~ (\myvec{d}^{H}\mymat{C}_{\rm ms}) (\mymat{\Theta}_{\rm ms})_1 = \mathbf{1}_n^{H}.
	\end{equation}
	
	However, over-sampling endows \((\mymat{\Theta}_{\rm ms})_1 \) with linearly independent rows.
	According to \cref{tab:4_possibilities_of_msub,rem:Submatrices__of__Theta_ms}, the matrix \((\mymat{\Theta}_{\rm ms})_1\) has \(\finitekissdim+1\) rows and \(\undelayedDMDorder\) columns.
	Since over-sampling means \(\undelayedDMDorder \geq \finitekissdim+1\), we can apply \cref{eq:vandermonde_row_rank_lemma} to conclude that \((\mymat{\Theta}_{\rm ms})_1 \) has full row rank.
	
	Consequently, \(\myvec{d}^H \mymat{C}_{\rm ms} \) is the unique representation of \(\myvec{1}_n^H\) in the row space of \((\mymat{\Theta}_{\rm ms})_1 \).
	Taken together with the property that \(1\) is a node of \(\mymat{\Theta}_{\rm ms} \) and, hence, of \((\mymat{\Theta}_{\rm ms})_1 \), we see that \(\myvec{d}^H \mymat{C}_{\rm ms} \) must be zero at all indices except that corresponding to the node at 1 i.e.,
	\begin{displaymath}
		\myvec{d}^{H} (\mymat{C}_{\rm ms})_{\not \owns 1} = \myvec{0}.
	\end{displaymath}
	
	Now, the sequential application of \cref{lem:Zero_mu_iff_Case_III,prop:C_without_static_mode} yields a contradiction.
	Since we are not dealing with Case III, \cref{lem:Zero_mu_iff_Case_III} gives \(\myvec{\mu} \neq \myvec{0}\).
	Therefore, \cref{prop:C_without_static_mode} says,
	\begin{displaymath}
		\mathcal{R}(\mymat{C}_{\rm ms}) ~=~ \mathcal{R}(\mymat{C}_{\not \owns 1}) ~=~ \mathcal{R}\left( (\mymat{C}_{\rm ms})_{\not \owns 1} \right) .
	\end{displaymath}
	Consequently, we get,
	\begin{displaymath}
		\myvec{d}^H \mymat{C}_{\rm ms}~=~\myvec{0},
	\end{displaymath}
	which contradicts \cref{eq:DMD_DFT_Equivalence_via_rows_of_Xms}.

	\mypara{Case II}
	By \cref{tab:4_possibilities_of_msub}, the number \(1\) continues to be a node of \(\mymat{\Theta}_{\rm ms}\) i.e, \(1 \in \sigma(\mymat{\Theta}_{\rm ms})\).
	
	Non-equivalence of \(\muDMD\) and DFT can be established by re-using the same line of reasoning from Case I, for the most part.
	
	We need only update the arguments used to show \((\mymat{\Theta}_{\rm ms})_1 \) has full row rank.
	\Cref{tab:4_possibilities_of_msub,rem:Submatrices__of__Theta_ms} tell us that \((\mymat{\Theta}_{\rm ms})_1\) has \(\finitekissdim\) rows and \(\undelayedDMDorder\) columns.
	Even though \((\mymat{\Theta}_{\rm ms})_1 \) has one fewer row than in Case I, over-sampling means \(\undelayedDMDorder \geq \finitekissdim+1\), so there are still fewer rows than columns. 
	Consequently, \cref{eq:vandermonde_row_rank_lemma} tells us that \((\mymat{\Theta}_{\rm ms})_1\) has full row rank.
	
	\mypara{Case III}
	In this scenario, by \cref{tab:4_possibilities_of_msub}, \(1\) is not a node of \(\mymat{\Theta}_{\rm ms}\) i.e., \(1 \notin \sigma(\mymat{\Theta}_{\rm ms})\).
	We can establish the non-equivalence of \(\muDMD\) and DFT by contradiction.
	
	Suppose \(\muDMD\) is equivalent to DFT.
	Then, the Vandermonde matrix,
	\begin{displaymath}
		\mymat{\Theta}_{\rm aug} := 	
		\begin{bmatrix}
			(\mymat{\Theta}_{\rm ms})_1 \\ \mathbf{1}_n^{H}
		\end{bmatrix},
	\end{displaymath}
	has linearly dependent rows.
	To see this, we can retrace the first two steps from Case I - sequentially apply \cref{thm:musubNDFT} and the identity \( \mathcal{N}(\mymat{X}_{\rm ms})^\perp = \mathcal{R}(\mymat{X}_{\rm ms}^{H})\) - to get the following assertion:
	\begin{displaymath}
		\exists~ \myvec{d} ~\textrm{such that}~ \myvec{d}^{H}\mymat{X}_{\rm ms} = \mathbf{1}_n^{H}.
	\end{displaymath}
	Shifting \(\myvec{1}_n^H\) to the left and using \cref{eq:KMF__of__X_ms}, we obtain an equivalent statement:
	\begin{displaymath}
		\exists ~\myvec{d} ~\textrm{such that}~ 
		\begin{bmatrix}
			\myvec{d}^{H} \mymat{C}_{\rm ms} & -1
		\end{bmatrix}
		\begin{bmatrix}
			(\mymat{\Theta}_{\rm ms})_1 \\ \mathbf{1}_n^{H}
		\end{bmatrix} = \myvec{0}.
	\end{displaymath}
	Hence, the rows of \( \mymat{\Theta}_{\rm aug} \) are linearly dependent.

	However, this deduction can be contradicted by drawing on the condition of over-sampling.
	Firstly, \( 1 \notin \sigma(\mymat{\Theta}_{\rm ms})\) ensures that  \(\mymat{\Theta}_{\rm aug}\) has distinct nodes.
	Now, \cref{tab:4_possibilities_of_msub} tells us \(\mymat{\Theta}_{\rm aug}\) has \(r+1\) rows and \(\undelayedDMDorder\) columns.
	Additionally, the condition of over-sampling (\(\undelayedDMDorder \geq \finitekissdim+1\)) ensures that there are at-least as many columns as rows.
	Hence, we can draw on \cref{eq:vandermonde_row_rank_lemma} to infer that \(\mymat{\Theta}_{\rm aug}\) has full row rank which is a contradiction.
	
	\mypara{Case IV}
	According to \cref{tab:4_possibilities_of_msub}, the number \(1\) continues to be excluded from the nodes of \(\mymat{\Theta}_{\rm ms}\) i.e., \(1 \notin \sigma(\mymat{\Theta}_{\rm ms})\).
	
	Non-equivalence of \(\muDMD\) and DFT can be shown using the same arguments from Case III, with minor updates regarding the size of \(\mymat{\Theta}_{\rm aug}\).
	Specifically, by \cref{tab:4_possibilities_of_msub}, the matrix \(\mymat{\Theta}_{\rm aug}\) will now have \(r\) rows and \(\undelayedDMDorder\) columns.
	Despite having one fewer row than in Case III, over-sampling (\(\undelayedDMDorder \geq \finitekissdim+1\)) ensures that the columns of \(\mymat{\Theta}_{\rm aug}\) continue to outnumber the rows.
	Hence, \cref{eq:vandermonde_row_rank_lemma} tells us that \(\mymat{\Theta}_{\rm aug}\) has full row rank.
	
	\manualQED
	
	\begin{remark}\label{rem:Generality__OversamplingProof}
		In the above proof of \cref{thm:DMDDFT_Oversampled}, Cases II and IV can also be established under the weaker condition of well-sampling (\(\undelayedDMDorder \geq \finitekissdim\)).
		We will leverage this subsequently, to prove \cref{cor:DMDFT_Undersampled}, in Case II of the just-sampled regime (\(\undelayedDMDorder = \finitekissdim\)).
	\end{remark}

	\subsection{Exploring the under and just-sampled regimes}\label{ss:Appendix_Proof_Under_Just_Sampled}
	\Cref{lem:ChenSufficiency4c_ms} explains the under-sampled regime. 
	Alas, the just-sampled scenario does not afford such a simple explanation.
	To substantiate this, we begin by formalizing an intermediary of import.
	\begin{proposition}\label{prop:LinearConsistency_Restricts_KernelOfCms}
		Suppose \PsiLiesInNRSpanOfrDistinctKEFs\, and \ICBeSpectrallyInformative.
		
		If the columns of \(\tilde{\mymat{C}}\) are linearly independent and \(1\) is a node of \(\mymat{\Theta}_{\rm ms}\), then, \(\mathcal{N}(\mymat{C}_{\rm ms})\) is the one-dimensional subspace spanned by \(\mymat{\Theta}_{\rm ms} \mathbf{1}_{n+1}\).
		\begin{displaymath}
			\tilde{\mymat{C}}~\textrm{has full column rank}~\&~1 \in \sigma(\mymat{\Theta}_{\rm ms}) \implies \mathcal{N}(\mymat{C}_{\rm ms})~=~\mathcal{R}(\mymat{\Theta}_{\rm ms} \mathbf{1}_{n+1}).
		\end{displaymath}
	\end{proposition}
	
	\paragraph{Proof}
	We begin by showing that \(\mathcal{N}(\mymat{C}_{\rm ms})\) contains \(\mathcal{R}(\mymat{\Theta}_{\rm ms} \mathbf{1}_{n+1})\). 
	Consider \(\mymat{Z}_{\rm ms}\) as defined in \cref{eq:Define__Z_ms}.
	Computing the average of its columns and applying \cref{eq:Defn_Temporal_Mean} gives:
	\begin{displaymath}
		\mymat{Z}_{\rm ms} \frac{\mathbf{1}_{n+1}}{n+1} = \myvec{0}.
	\end{displaymath}
	Expanding \(\mymat{Z}_{\rm ms}\) using \cref{eq:KMF__of__Z_ms} and invoking the definition of \(\mathcal{N}(\mymat{C}_{\rm ms})\), we find
	\begin{displaymath}
		\EmphasiseReduction{\mymat{Z}_{\rm ms}} \frac{\mathbf{1}_{n+1}}{n+1} = \myvec{0} 
		~\iff~ 
		\mymat{C}_{\rm ms} \EmphasiseReduction{\mymat{\Theta}_{\rm ms} \mathbf{1}_{n+1}} = \EmphasiseReduction{\myvec{0}} 
		~\implies~
		\mathcal{N}(\mymat{C}_{\rm ms}) \supseteq \mathcal{R}(\mymat{\Theta}_{\rm ms} \mathbf{1}_{n+1}).
	\end{displaymath}
	
	Now, we need to show that \(\mathcal{N}(\mymat{C}_{\rm ms})\) is in-turn subsumed by \(\mathcal{R}(\mymat{\Theta}_{\rm ms}  \myvec{1}_{\undelayedDMDorder+1})\).
	
	This is equivalent to showing that \(\mathcal{N}(\mymat{C}_{\rm ms})\) is a one dimensional subspace.
	Such an equivalence arises from \(1\) being a node of \(\mymat{\Theta}_{\rm ms}\), a property that ensures	 \( \mymat{\Theta}_{\rm ms} \myvec{1}_{n+1} ~\neq~ \myvec{0} \). 
	Hence, \( \mathcal{R}(\mymat{\Theta}_{\rm ms} \myvec{1}_{n+1}) \) is  a non-trivial subspace of dimension 1. 
	Since \(\mathcal{N}(\mymat{C}_{\rm ms}) \supseteq \mathcal{R}(\mymat{\Theta}_{\rm ms} \mathbf{1}_{n+1})\), we need only show that \(\mathcal{N}(\mymat{C}_{\rm ms})\) is one-dimensional to turn the set containment into an equality.
	
	Hence, we quantify the dimension of \(\mathcal{N}(\mymat{C}_{\rm ms})\), through the more tractable \(\mathcal{R}(\mymat{C}_{\not \owns 1}) \).
	Once again, \(1\) being a node of \(\mymat{\Theta}_{\rm ms}\) means we are either in Case I or II of \cref{tab:4_possibilities_of_msub}.
	This lets us apply \cref{lem:Zero_mu_iff_Case_III} to infer that \(\myvec{\mu} \neq \myvec{0} \).
	Consequently, we can apply \cref{prop:C_without_static_mode} to get \(\mathcal{R}(\mymat{C}_{\rm ms}) = \mathcal{R}(\mymat{C}_{\not \owns 1} )\). 
	By the rank-nullity theorem, we then obtain:
	\begin{displaymath}
		\dim(\mathcal{N}(\mymat{C}_{\rm ms} ))~=~\textrm{Number of columns in}~\mymat{C}_{\rm ms}~-~ \dim(\mathcal{R}(\mymat{C}_{\not \owns 1})  ).
	\end{displaymath}

	Since \(\dim(\mathcal{R}(\mymat{C}_{\not \owns 1})  )\) is always one less than the number of columns in \(\mymat{C}_{\rm ms}\), we have the desired conclusion.
	We can compute \(\dim(\mathcal{R}(\mymat{C}_{\not \owns 1})  )\) using \cref{tab:4_possibilities_of_msub}.
	Since \(1 \in \sigma(\mymat{\Theta}_{\rm ms})\), we need only consider Cases I and II.
	We preface these considerations by noting that the spectrally informative initial condition ensures \(\mymat{C}\) inherits, from \(\tilde{\mymat{C}}\), a full column rank of \(\finitekissdim\).
	\mysubpara{Case I}
	Here, \(1\) is not a Koopman eigenvalue i.e., \(1 \notin \KEvals \) and the Vandermonde matrix \(\mymat{\Theta}_{\rm ms}\) has \(r+1\) rows. 
	
	According to \cref{eq:KMF__of__Z_ms}, \(\mymat{\Theta}_{\rm ms}\) can be pre-multiplied by \(\mymat{C}_{\rm ms}\).
	So, \(\mymat{C}_{\rm ms}\) must have \(\finitekissdim+1\) columns. 
	Since \(1 \notin \KEvals\),  \cref{defn:msub__illustrating__partition} says \(\mymat{C}_{\not \owns 1}\) coincides with \(\mymat{C}\) which has a full column rank of \(\finitekissdim\). 
	Hence,
	\begin{displaymath}
		\dim(\mathcal{N}(\mymat{C}_{\rm ms} ))~=~(r+1) - (r)~=~1.
	\end{displaymath}
	\mysubpara{Case II} 
	In this scenario, \(1\) is a Koopman eigenvalue i.e., \(1 \in \KEvals \) and \(\mymat{\Theta}_{\rm ms}\) has \(r\) rows.
	
	As before, we can draw upon \cref{eq:KMF__of__Z_ms} to deduce that \(\mymat{C}_{\rm ms}\) has \(r\) columns.
	Since \(1 \in \KEvals\), \cref{defn:msub__illustrating__partition} says that \(\mymat{C}_{\not \owns 1}\) is obtained by excluding one of the \(\finitekissdim\) columns in \(\mymat{C}\).
	Hence, the full column rank of \(\mymat{C}\)  translates to \(\mymat{C}_{\not \owns 1}\) possessing \(r-1\) linearly independent columns. 
	Therefore, we have:
	\begin{displaymath}
		\dim(\mathcal{N}(\mymat{C}_{\rm ms} ))~=~(r) - (r-1)~=~1.
	\end{displaymath}

	\manualQED
	
	
	\Cref{cor:DMDFT_Undersampled} is now amenable to reasoning:
	
	\subsubsection{Proof of \cref{cor:DMDFT_Undersampled}}

	There are two distinct segments to this proof, corresponding to the under-sampled (\(n < r\)) and the just-sampled (\(n=r\)) regimes. 
	We preface the discussion by highlighting a consequence of \cref{prop:ManyFacesofLinCon} and the initial condition \(\state_1\) being spectrally informative - the matrices \(\tilde{\mymat{C}}\) and \(\mymat{C}\) both have full column rank. 
	
	\mypara{Under-sampling \((\undelayedDMDorder < \finitekissdim)\)}
	
	In this regime, we will establish the equivalence of \(\muDMD\) and DFT using \cref{lem:ChenSufficiency4c_ms}.
	So, we will focus on showing that \(\mymat{X}_{\rm ms}\) has full column rank. 
	
	The pertinent analysis splits according to the efficacy of mean-subtraction, as outlined in the second column of \cref{tab:4_possibilities_of_msub}.
	
	\mysubpara{Cases I and II}
	According to \cref{tab:4_possibilities_of_msub}, \(1\) is a node of the Vandermonde matrix \(\mymat{\Theta}_{\rm ms}\) i.e., \(1 \in \sigma(\mymat{\Theta}_{\rm ms})\).

	Proof by contradiction.
	Suppose \(\mymat{X}_{\rm ms}\) does not have full column rank i.e.,
	\begin{displaymath}
		\exists~ \myvec{z} \neq \myvec{0} ~\textrm{such that}~\mymat{X}_{\rm ms} \myvec{z} = \myvec{0}.
	\end{displaymath}
	
	Then, the matrix \(\mymat{\Theta}_{\rm ms}\) has a non-trivial nullspace.
	To see this, we first use \cref{eq:KMF__of__X_ms} to expand \(\mymat{X}_{\rm ms}\) thus:
	\begin{displaymath}
		\exists~ \myvec{z} \neq \myvec{0} ~\textrm{such that}~
		\mymat{C}_{\rm ms} (\mymat{\Theta}_{\rm ms})_1 \myvec{z} = \myvec{0}.
	\end{displaymath}
	By \cref{prop:LinearConsistency_Restricts_KernelOfCms}, we know \(\mathcal{N}(\mymat{C}_{\rm ms}) = \mathcal{R}(\mymat{\Theta}_{\rm ms} \myvec{1}_{n+1})\). 
	Hence, there exists a non-zero scalar \(\beta\) such that:
	\begin{displaymath}
		(\mymat{\Theta}_{\rm ms})_1 \myvec{z} = \beta \mymat{\Theta}_{\rm ms} \mathbf{1}_{n+1}~\iff~\mymat{\Theta}_{\rm ms} 
		\begin{bmatrix}
			\myvec{z} - \beta \mathbf{1}_n \\
			-\beta
		\end{bmatrix}
		= \myvec{0}.
	\end{displaymath}
	
	However, under-sampling ensures that \(\mymat{\Theta}_{\rm ms}\) has full column rank, which is a contradiction.
	This can be deduced from \cref{tab:4_possibilities_of_msub} which says \(\mymat{\Theta}_{\rm ms}\) has \(\undelayedDMDorder+1\) columns and \(\finitekissdim + 1\) (\(\finitekissdim\)) rows in Case I (Case II).
	Since under-sampling means \(\undelayedDMDorder+1 \leq \finitekissdim\), there are at least as many rows in \(\mymat{\Theta}_{\rm ms}\) as columns.
	So, we can apply \cref{eq:vandermonde_column_rank_lemma} to deduce that \(\mymat{\Theta}_{\rm ms}\) has full column rank.
	
	Thus, \(\mymat{X}_{\rm ms}\) must have full column rank.
	
	\mysubpara{Cases III and IV }
	Now, the number \(1\) is no longer a node of \(\mymat{\Theta}_{\rm ms}\) i.e., \(1 \notin \sigma(\mymat{\Theta}_{\rm ms})\).
	
	The matrix \(\mymat{C}_{\rm ms}\) has full column rank.
	According to \cref{tab:4_possibilities_of_msub}, the matrix \(\mymat{C}_{\rm ms}\) is \(\mymat{C}\) (Case III) or \(\mymat{C}_{\not \owns 1}\) (Case IV).
	Since both matrices have full column rank, \(\mymat{C}_{\rm ms}\) will also have linearly independent columns.
	
	Similarly, the matrix \((\mymat{\Theta}_{\rm ms})_1\) also has full column rank.
	To begin with, we have just seen that \(\mymat{C}_{\rm ms}\) possesses at least \(\finitekissdim-1\) columns.
	So, by \cref{eq:KMF__of__X_ms}, \( (\mymat{\Theta}_{\rm ms})_1\) has at least \(\finitekissdim-1\) rows.
	Since it has \(\undelayedDMDorder\) columns and under-sampling means \(\undelayedDMDorder \leq \finitekissdim-1\), \cref{eq:vandermonde_column_rank_lemma} says that \( (\mymat{\Theta}_{\rm ms})_1\) has full column rank.
	
	Therefore, \(\mymat{X}_{\rm ms}\), being the product of these two matrices, inherits linearly independent columns.
	
	\mypara{Just-sampling \((\undelayedDMDorder~=~\finitekissdim)\)} 
	
	Just-sampling necessitates a case-by-case handling of the scenarios described in \cref{tab:4_possibilities_of_msub}.
	
	\mysubpara{Case I}
	Here, \(1\) is a node of the Vandermonde matrix \(\mymat{\Theta}_{\rm ms}\) but not a Koopman eigenvalue.
	
	The latter means that we need to show  \(\muDMD\) is equivalent to a DFT.
	
	To this end, we can adopt the same line of reasoning used in the under-sampled regime, with appropriate modifications in showing that \(\mymat{\Theta}_{\rm ms}\) has full column rank.
	By \cref{eq:Define__Z_ms,eq:KMF__of__Z_ms}, \(\mymat{\Theta}_{\rm ms}\)  has \(\undelayedDMDorder+1\) columns.
	Additionally, this being Case I, \cref{tab:4_possibilities_of_msub} tells us that \(\mymat{\Theta}_{\rm ms}\)  has \(\finitekissdim+1\) distinct nodes, which translates to \(\finitekissdim+1\) rows.
	Since just-sampling means \(\undelayedDMDorder ~=~ \finitekissdim\), \(\mymat{\Theta}_{\rm ms}\) still has at-least as many rows as columns.
	So, we can continue to rely on \cref{eq:vandermonde_column_rank_lemma} to deduce that \(\mymat{\Theta}_{\rm ms}\) has linearly independent columns.
	
	\mysubpara{Case II} 
	According to \cref{tab:4_possibilities_of_msub}, \(1\) is a node of \(\mymat{\Theta}_{\rm ms}\) and is also a Koopman eigenvalue.
	
	So, we need to establish non-equivalence of \(\muDMD\) and DFT.
	
	This can be achieved by drawing on \cref{thm:DMDDFT_Oversampled}, even though we are not in the over-sampled regime.
	Firstly, note that all the conditions required by \cref{thm:DMDDFT_Oversampled}, except for over-sampling, are met.
	We have a dictionary in the non-redundant span of \(\finitekissdim\) distinct KEFs and a spectrally informative initial condition \(\state_1\).
	Furthermore, the \hyperref[defn:Koopman_invariant_psi]{Koopman invariance of \(\dictionary\)}  ensures that \((\mymat{X}, \mymat{Y})\) is linearly consistent.
	Now, \cref{rem:Generality__OversamplingProof} tells us that here in Case II, \cref{thm:DMDDFT_Oversampled} holds under the weaker condition of well-sampling (\(\undelayedDMDorder \geq \finitekissdim\)).
	Since just-sampling (\(\undelayedDMDorder = \finitekissdim\)) implies well-sampling, we can apply \cref{thm:DMDDFT_Oversampled} to infer non-equivalence of \(\muDMD\) and DFT.

	\mysubpara{Case III}
	According to \cref{tab:4_possibilities_of_msub}, \(1\) is neither a node of \(\mymat{\Theta}_{\rm ms}\) nor a Koopman eigenvalue.
	
	So, we need to establish the equivalence of \(\muDMD\) and DFT.
	
	Once again, we can re-use the reasoning from Case III in the under-sampled regime, with relevant modifications to show that \((\mymat{\Theta}_{\rm ms})_1\) has full column rank.
	Firstly, by \cref{tab:4_possibilities_of_msub}, we have \(\mymat{C}_{\rm ms} ~=~ \mymat{C}\).
	This means that \(\mymat{C}_{\rm ms}\) has \(r\) columns.
	So, by \cref{eq:Define__Z_ms,eq:KMF__of__X_ms,eq:Define__X_ms},  \((\mymat{\Theta}_{\rm ms})_1\) has \(r\) rows and \(\undelayedDMDorder\) columns.
	Since just-sampling means \(\undelayedDMDorder~=~\finitekissdim\), the matrix \((\mymat{\Theta}_{\rm ms})_1\) continues to possess at-least as many rows as columns.
	Consequently, \cref{eq:vandermonde_column_rank_lemma} tells us that \((\mymat{\Theta}_{\rm ms})_1\) has full column rank.

	\mysubpara{Case IV}
	Here, \cref{tab:4_possibilities_of_msub} tells us that \(1\) is not a node of \(\mymat{\Theta}_{\rm ms}\), but is a Koopman eigenvalue.
	
	So, we need to prove that  \(\muDMD\) is not equivalent to a DFT.
	
	This can be achieved using  \cref{thm:DMDDFT_Oversampled}, despite only being just-sampled, with an appropriate ``change of coordinates''.
	In particular, the similarity of \cref{eq:KMF__of__Z_ms} to \cref{eq:KMF__of__Z} lets us view the columns of \(\mymat{Z}_{\rm ms} \) as snapshots of the same dynamical system along the same trajectory starting at \(\state_1\), but generated using a  different set of observables. 
	The new dictionary lies in the non-redundant span of the \(\finitekissdim-1\) KEFs that correspond to the eigen-values \(\KEvals \cap \{1\}^c\). 
	The initial condition \(\state_1\) continues to be spectrally informative for this modified collection of observables.
	Furthermore, note that 
	\begin{itemize}
		\item The matrix \(\mymat{C}_{\rm ms} = \mymat{C}_{\not \owns 1}\) has full column rank. 
		So, by \cref{prop:ManyFacesofLinCon}, the new dictionary is also Koopman invariant- a property that, by definition, guarantees the pertinent inputs to DMD are linearly consistent.
		\item The number of columns in \(\mymat{X}_{\rm ms}\) (\(r\)) is at least 1 more than the number of distinct KEFs (\(r-1\)) whose non-redundant span contains the new observables.
	\end{itemize}
	Applying \cref{thm:DMDDFT_Oversampled}, we find that evaluating \cref{alg:MeanSubtractedDMD} using \(\mymat{Z}_{\rm ms}\) as the input is not equivalent to performing a temporal DFT.
	However, the mean-subtracted time series \(\mymat{Z}_{\rm ms}\) has zero temporal mean.
	So, the output of \cref{alg:MeanSubtractedDMD} remains the same regardless of whether we use \(\mymat{Z}\) or \(\mymat{Z}_{\rm ms}\) as the input.
	Therefore, when we use \(\mymat{Z}\) as the input to \cref{alg:MeanSubtractedDMD} i.e., when we perform \(\muDMD\), it is not equivalent to a DFT.
	
	\manualQED
	
	\subsection{Delay embedding can generate a Koopman invariant dictionary}\label{ss:TimeDelays_Gib_KoopmanInvariance}
	Taking time delays is a simple and pragmatic way to produce a Koopman invariant dictionary from \(\dictionary\) (\cref{prop:DelaysMakeGoodObservables}).
	To substantiate this claim, consider the matrices \(\tilde{\mymat{C}}_{\rm d-delayed}\) and \(\mymat{C}_{\rm d-delayed}\) defined as the delay-embedded counterparts of \( \tilde{\mymat{C}} \) and \(\mymat{C}\) respectively.
	\begin{equation}\label{eq:Define__Ctilde_d_delayed}
		\tilde{\mymat{C}}_{\rm d-delayed} := 
		\begin{bmatrix}
			\tilde{\mymat{C}} \\ \tilde{\mymat{C}} \mymat{\Lambda} \\ \vdots \\ \tilde{\mymat{C}} \mymat{\Lambda}^{d}
		\end{bmatrix},~
		\mymat{C}_{\rm d-delayed} := 
		\begin{bmatrix}
			\mymat{C} \\ \mymat{C} \mymat{\Lambda} \\ \vdots \\ \mymat{C} \mymat{\Lambda}^{d}
		\end{bmatrix}.
	\end{equation}
	Using the definition of \(\mymat{\Theta}_d \) from \cref{eq:Vandermonde_and_submatrices} in \cref{eq:Define__Z_d_delayed}, we find:
	\begin{displaymath}
		\mymat{Z}_{\rm d-delayed}  = \mymat{C}_{\rm d-delayed} \mymat{\Theta}_d.
	\end{displaymath}
	Moreover, \cref{eq:C_definition} tells us how to connect \(\mymat{C}_{\rm d-delayed}\) and  \(\tilde{\mymat{C}}_{\rm d-delayed}\):
	\begin{displaymath}
		\mymat{C}_{\rm d-delayed} 
		= \tilde{\mymat{C}}_{\rm d-delayed} 
		\diag[\KEFVector(\state_1)].
	\end{displaymath}
	The preceding arguments, taken in light of the derivation in \Cref{ss:KMD_Compact} that gives \cref{eq:KMF__of__Z}, suggest that \(\tilde{\mymat{C}}_{\rm d-delayed}\) is the representation of \(\dictionary_{\rm d-delayed} \) in the eigen-basis \(\KEFVector \):
	\begin{displaymath}
		\dictionary_{\rm d-delayed} = \tilde{\mymat{C}}_{\rm d-delayed} \KEFVector.
	\end{displaymath}
	With this machinery in place, a classical argument from systems theory, reviewed here as \cref{lem:EVec__Test__For__Observability}, shows that \(r-1\) delays are sufficient for \(\dictionary_{\rm d-delayed}\) to be Koopman invariant.
	\begin{lemma}[Eigenvector test for Observability \cite{hespanha2018linear}]\label{lem:EVec__Test__For__Observability}
		Let \(\mymat{A}\) and \(\mymat{C}\) denote two matrices belonging to \(\mathbb{C}^{\finitekissdim \times \finitekissdim}\) and \( \mathbb{C}^{\dictionarylength \times \finitekissdim}\) respectively.
		Consider their observability matrix, which is defined thus:
		\begin{displaymath}
			\obs(\mymat{A}, \mymat{C}) ~:=~
			\begin{bmatrix}
				\mymat{C} \\
				\mymat{C} \mymat{A} \\
				\vdots \\
				\mymat{C} \mymat{A}^{\finitekissdim-1}
			\end{bmatrix}.
		\end{displaymath}
		If no eigenvector of \(\mymat{A}\) lies in the null-space of \(\mymat{C}\), then, the columns of \(\obs(\mymat{A}, \mymat{C})\) are linearly independent.
		\begin{displaymath}
			\{\myvec{v} \neq \myvec{0}~|~\exists \,\lambda ~\textrm{such that}~\mymat{A} \myvec{v} \,=\,\lambda \myvec{v} \} \,\cap\, \mathcal{N}(\mymat{C}) 
			= \{ \}
			\implies
			\obs(\mymat{A}, \mymat{C})~\textrm{has full column rank.}
		\end{displaymath}
		
	\end{lemma}
	\begin{proof}[Proof of \cref{prop:DelaysMakeGoodObservables}]
		According to \cref{prop:ManyFacesofLinCon}, Koopman invariance of \(\dictionary_{\rm d-delayed}\) can be established by showing that the columns of \(\tilde{\mymat{C}}_{\rm d-delayed}\) are linearly independent.
		To begin with, \(\dictionary\) lies in the non-redundant span of \(\finitekissdim\) distinct KEFs.
		So, \cref{eq:Every_column_in_CTilde_is_NonZero} says every column of \(\tilde{\mymat{C}}\) is non-zero.
		Furthermore, \cref{eq:Define__Ctilde_d_delayed} tells us that \(\tilde{\mymat{C}}\) comprises the first \(\dictionarylength\) rows of \(\tilde{\mymat{C}}_{\rm d-delayed}\).
		Hence, every column of \(\tilde{\mymat{C}}_{\rm d-delayed}\) is also non-zero, which means \(\dictionary_{\rm d-delayed}\) also lies in the non-redundant span of the same \(\finitekissdim\) distinct KEFs.
		Consequently, we can apply \cref{prop:ManyFacesofLinCon} to infer that Koopman invariance of \(\dictionary_{\rm d-delayed}\) is equivalent to \(\tilde{\mymat{C}}_{\rm d-delayed}\) possessing full column rank.
		
		Now, \(\fundelayscount \geq \finitekissdim-1\) means the rows of		
		\(\tilde{\mymat{C}}_{\rm d-delayed}\) are a superset of the rows of \(\obs(\mymat{\Lambda},\tilde{\mymat{C}})\).
		
		Additionally, \cref{lem:EVec__Test__For__Observability}, applied to the matrix pair \((\mymat{\Lambda}, \tilde{\mymat{C}})\), tells us that \(\obs(\mymat{\Lambda},\tilde{\mymat{C}})\) has full column rank.
		To see this, note that the eigenvectors of \(\mymat{\Lambda}\) are the standard basis vectors.
		Furthermore, we have already seen that no column of \(\tilde{\mymat{C}}\) is \(\myvec{0}\). 
		In other words, no eigenvector of \(\mymat{\Lambda}\) lies in the null-space of \(\tilde{\mymat{C}}\).
		Hence, \cref{lem:EVec__Test__For__Observability} tells us that the columns of \(\obs(\mymat{\Lambda},\tilde{\mymat{C}})\) are linearly independent.
		
		Therefore, \(\tilde{\mymat{C}}_{\rm d-delayed}\) also has full column rank.
		
	\end{proof}

	\section{Impact of low-rank regularization on the computation of \DistanceToDFT}\label{s:DisrepancyChecks}
	
	Recall, from \cref{rem:Only_computing_an_upper_bound}, that the discrepancy between \((\mymat{X}_{\rm d-delayed})_{\rm ms} \) and its low rank approximation \((\dummy{\mymat{X}}_{\rm d-delayed})_{\rm ms} \) determines whether we compute \DistanceToDFT~ or just its upper bound.
	Specifically, \DistanceToDFT~ can be computed if and only if said discrepancy is low.
	In practice, we let \(\sigma_{\rm Tail}\) denote the norm of the discrepancy/approximation error i.e.,
	\begin{displaymath}
		\sigma_{\rm Tail} ~:=~ \left\lVert (\mymat{X}_{\rm d-delayed})_{\rm ms}  - (\dummy{\mymat{X}}_{\rm d-delayed})_{\rm ms} \right \rVert,
	\end{displaymath}
	and deem the discrepancy to be low when \(\sigma_{\rm Tail} \leq 10^{-15}\).
	
	So, for each system considered in \Cref{s:Numerics}, we now scrutinize the variation of \(\sigma_{\rm Tail}\) with model order \((\theta)\) and number of time delays \((\fundelayscount)\).
	
	For all three LTI systems, \Cref{fig:sigma_tail__knownss} tells us there is low discrepancy until we are well into the over-sampled regime (\(\theta \geq 14\)), thereby validating that \cref{fig:cms_transit_knownss} indeed depicts \DistanceToDFT~ when \(\theta < 14\).
	Specifically, when we take only \(3\) or \(6\) time delays, \(\sigma_{\rm Tail}\) is negligible throughout \cref{fig:sigma_tail__knownss} and, hence, the concomitant trends in \cref{fig:cms_transit_knownss} are indeed those of \DistanceToDFT.
	However, when we excessively take \(24\) time delays, \(\sigma_{\rm Tail}\) becomes non-trivial from \(\theta~=~14,16~\textrm{and}~21\) for \(\lti_{1a},\lti_{1b}~\textrm{and}~\lti_{3}\) respectively.
	This departure from low discrepancy is conveniently reflected in \cref{fig:cms_transit_knownss} as secondary spikes.
	Therefore, the visualizations in \cref{fig:cms_transit_knownss} corresponding to \(d=24\) depict \DistanceToDFT~ only until said secondary spikes, beyond which they represent just an upper bound on \DistanceToDFT.
	\begin{figure}[]
		\centering
		\subfloat[\(\lti_{1a}\)]{\includegraphics[width = 0.71 \linewidth]{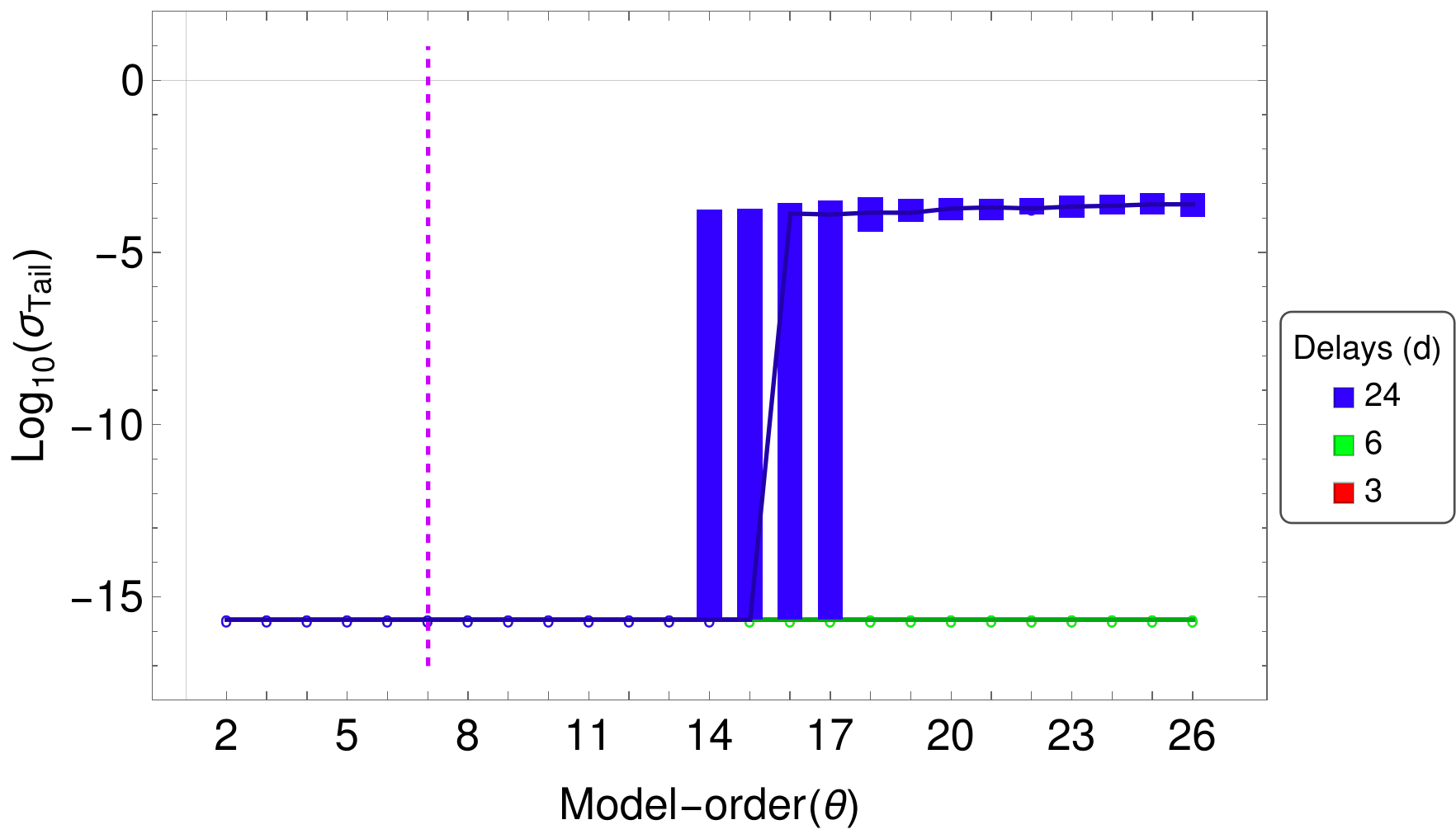}} \\
		\subfloat[\(\lti_{1b}\)]{\includegraphics[width = 0.71 \linewidth]{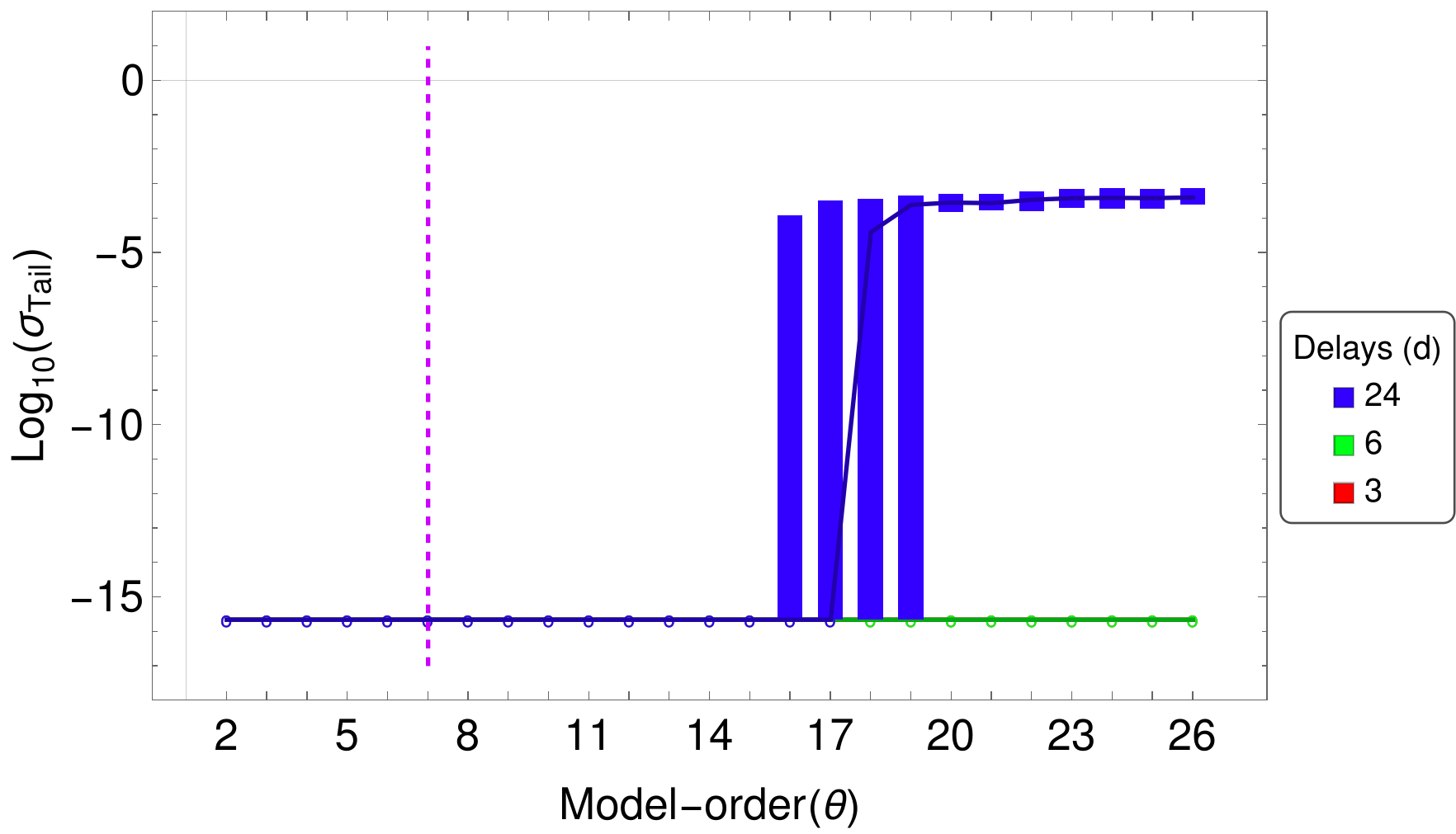}} \\
		\subfloat[\(\lti_{3}\)]{\includegraphics[width = 0.71 \linewidth]{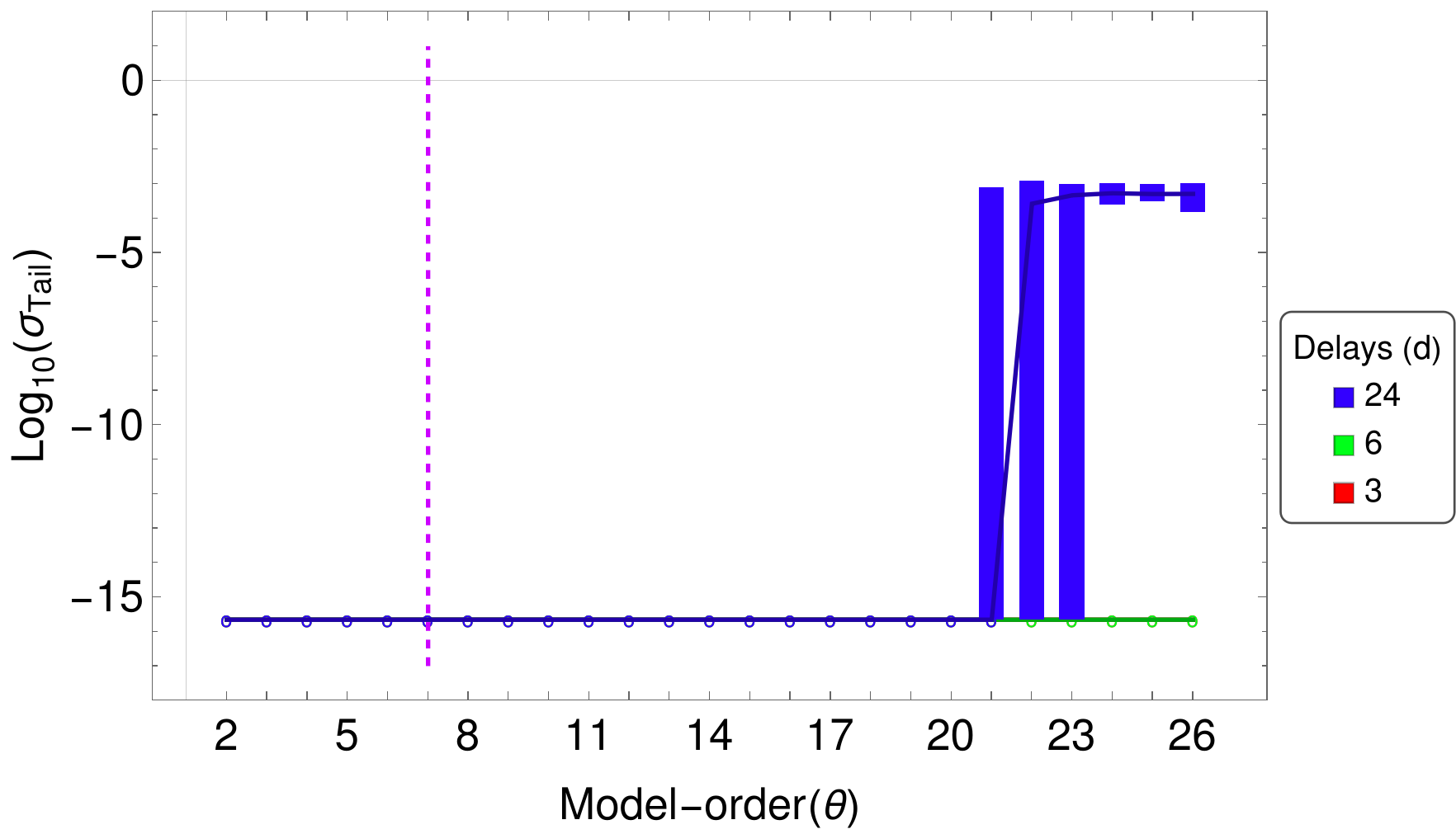}}\\
		\caption{Variation of \(\sigma_{\rm Tail}\) corresponding to the studies visualized in \cref{fig:cms_transit_knownss}.
		Given a choice of \((\theta,\fundelayscount)\), a low value of \(\sigma_{\rm Tail}\) means we have accurately computed \DistanceToDFT~ while a large value indicates that we have only calculated an upper bound.
		Hence, for all three systems,  when we only take \(3\) or \(6\) time delays, \DistanceToDFT~ is accurately computed for all model orders.
		In contrast, when \(24\) time delays are taken, \(\sigma_{\rm Tail}\) spikes at \(\theta=14,16~\textrm{and}~21\) for \(\lti_{1a},\lti_{1b}~\textrm{and}~\lti_{3}\) respectively, and remains high thereafter.
		Therefore, the concomitant computations in \cref{fig:cms_transit_knownss} for \(\fundelayscount=24\) are only accurate until the aforementioned values of \(\theta\).
		Beyond this, what is visualized in only an upper bound on \DistanceToDFT.
		}
		\label{fig:sigma_tail__knownss}
	\end{figure}

	In the Van der Pol oscillator, we see that \cref{fig:sigma_tail__VanDerPol} displays a trend similar to those seen for the LTI systems in \cref{fig:sigma_tail__knownss}.
	The value of \(\sigma_{\rm Tail}\) becomes non-trivial only when a large number of delays \((d = 25)\) are taken and the model order is at-least 21.
	Consequently, \cref{fig:DistanceToDFT__VanDerPol} is a bona fide plot of \DistanceToDFT~ for \(\fundelayscount= 3~\textrm{and}~12.\)
	Furthermore, the blue trend corresponding to \(\fundelayscount=25\) depicts \DistanceToDFT~ for \(\theta \leq 20\), and an upper bound of the same for larger values of \(\theta\).
	 \begin{figure}
		\centering
		\includegraphics[width = 0.95 \linewidth]{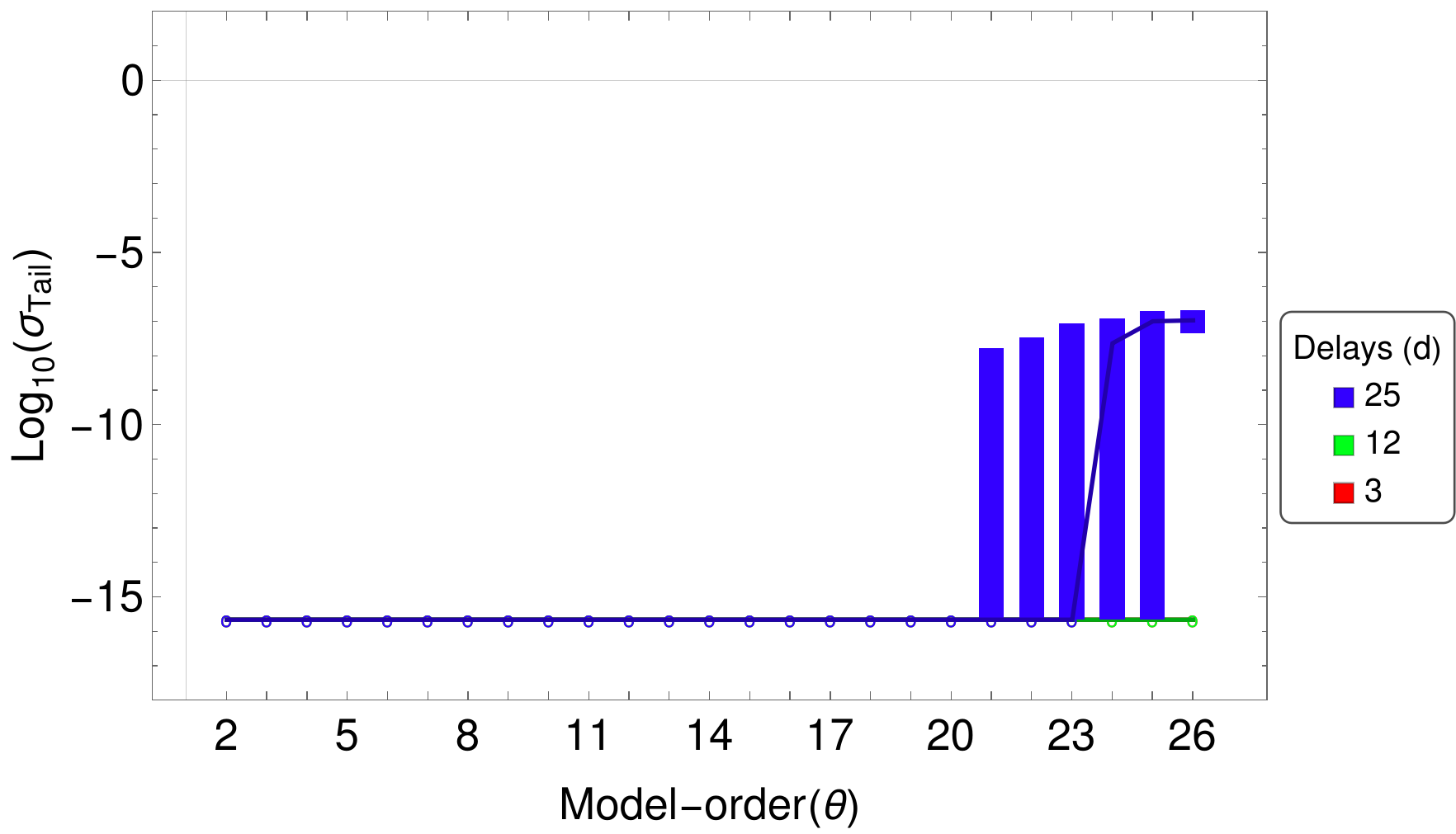}
		\caption{Variation in \(\sigma_{\rm Tail}\) for the computations with the Van der Pol oscillator in \cref{fig:DistanceToDFT__VanDerPol}.
		Unless the model order is more than \(20\) and \(\fundelayscount=25\), \(\sigma_{\rm Tail}\) remains low.
		Hence, the trends seen in \cref{fig:DistanceToDFT__VanDerPol}, particularly the spike at \(\theta=11\), are dynamically relevant and not an artifact of the low-rank approximation used in computing \DistanceToDFT.
	 }
		\label{fig:sigma_tail__VanDerPol}
	\end{figure}
	
	Finally, \cref{fig:sigma_tail__cavity} tells us that, with the exception of panel (d), \cref{fig:cms_transit_cavity} depicts only upper bounds on \DistanceToDFT.
	In panel (d) of \cref{fig:sigma_tail__cavity}, we see that \(\sigma_{\rm Tail}\) is virtually zero for all choices of model order and delay embedding dimension.
	Hence, panel (d) of \cref{fig:cms_transit_cavity} is an authentic depiction of \DistanceToDFT.
	In contrast, panels (a), (b) and (c) see \(\sigma_{\rm Tail}\) reach non-trivial values for both \(d=12\) and \(25\).
	Unfortunately, this departure from low discrepancy is the cause for panels (a) through (c) of \cref{fig:cms_transit_cavity} exhibiting a step-like trend.
	Despite this shortcoming, it is curious that said upper bounds are indicative of the growing complexity of the cavity flow.
	\begin{figure}[]
		\centering
		\subfloat[\(Re ~=~ 13\times 10^3 \) (Periodic)]{\includegraphics[width = 0.45 \linewidth]{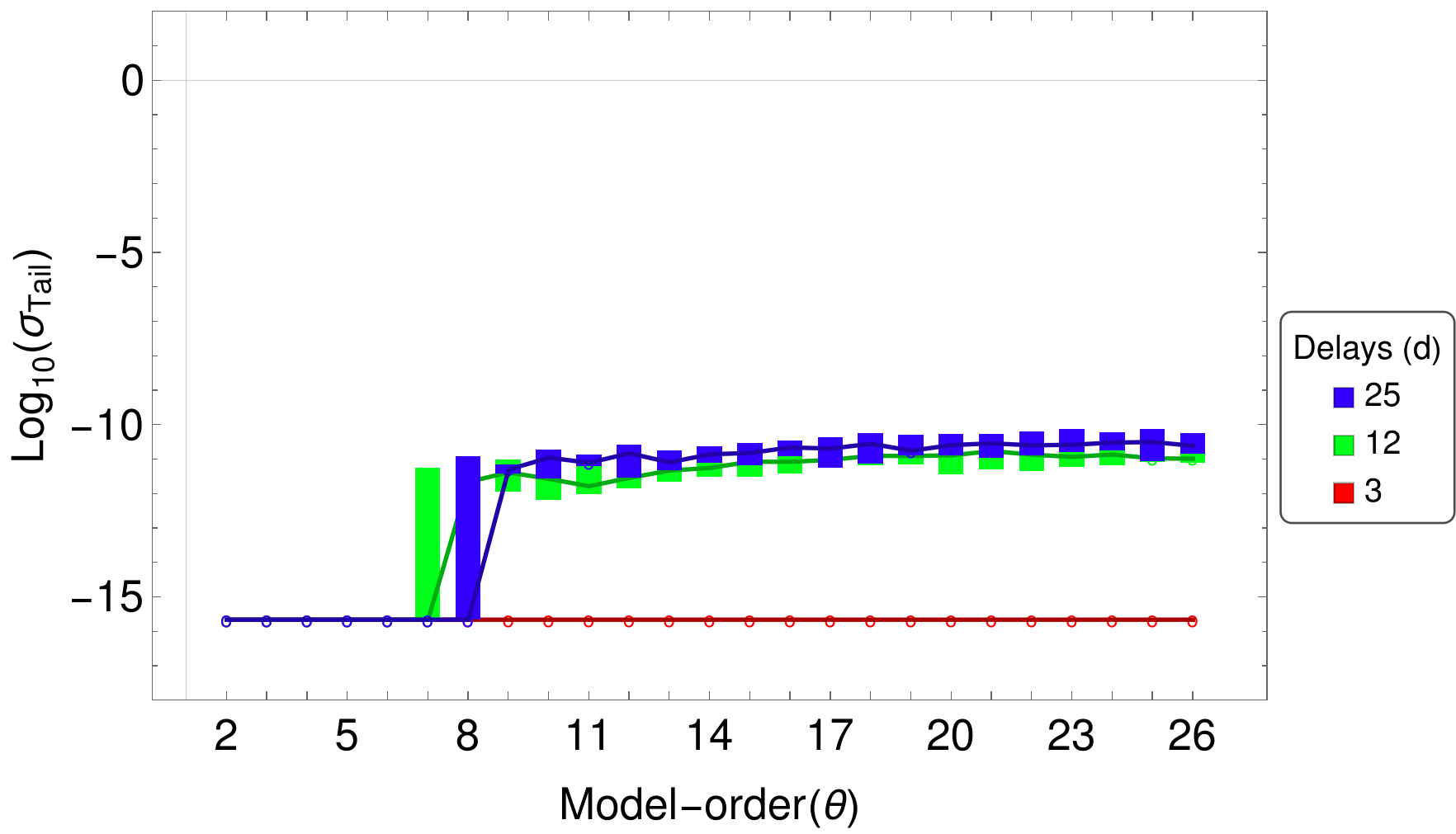}} \quad
		\subfloat[\(Re ~=~ 16\times 10^3\) (Quasi-periodic)]{\includegraphics[width = 0.45 \linewidth]{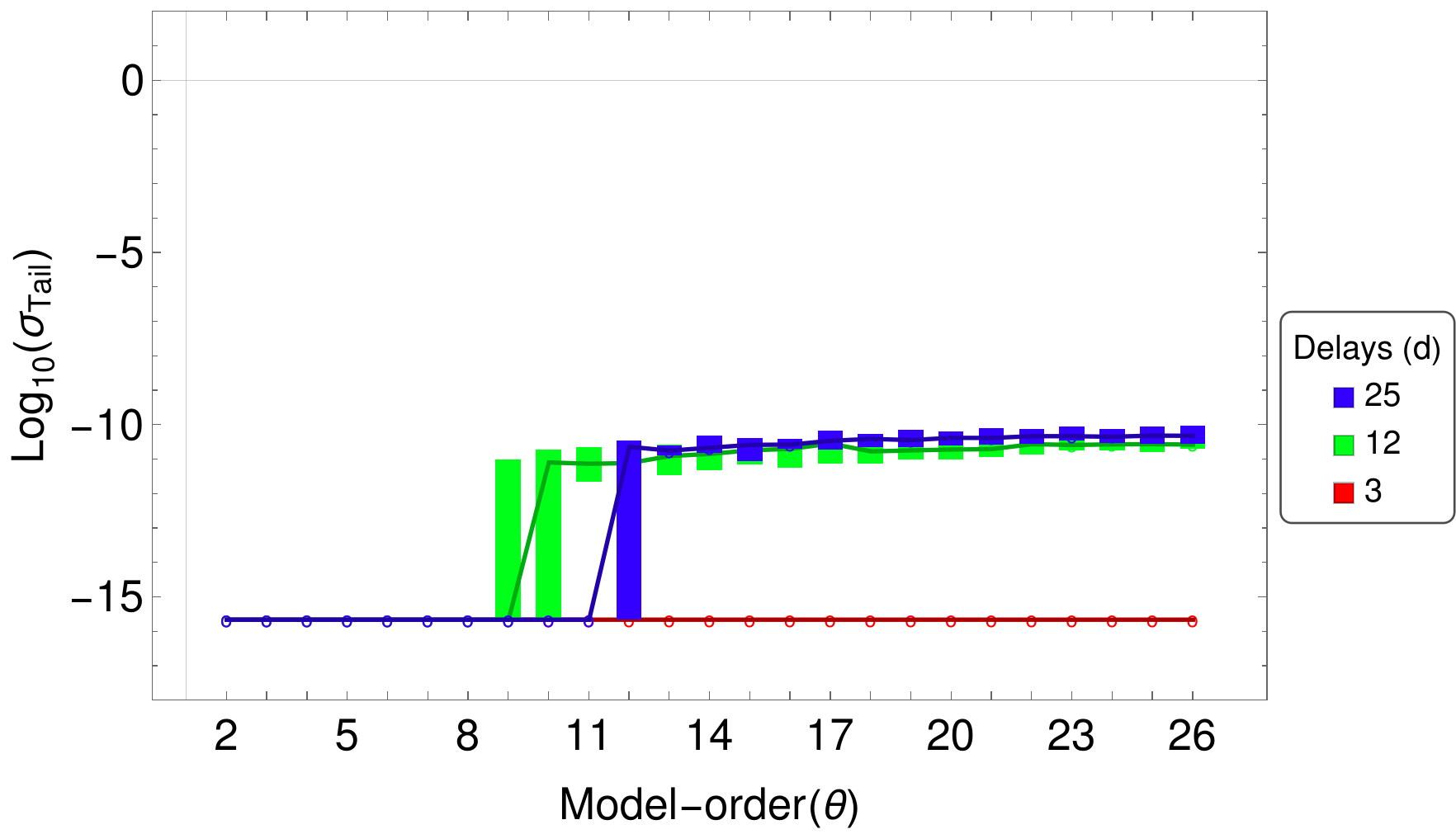}}  \\
		\subfloat[\(Re ~=~ 20\times 10^3\) (Mixed)]{\includegraphics[width = 0.45 \linewidth]{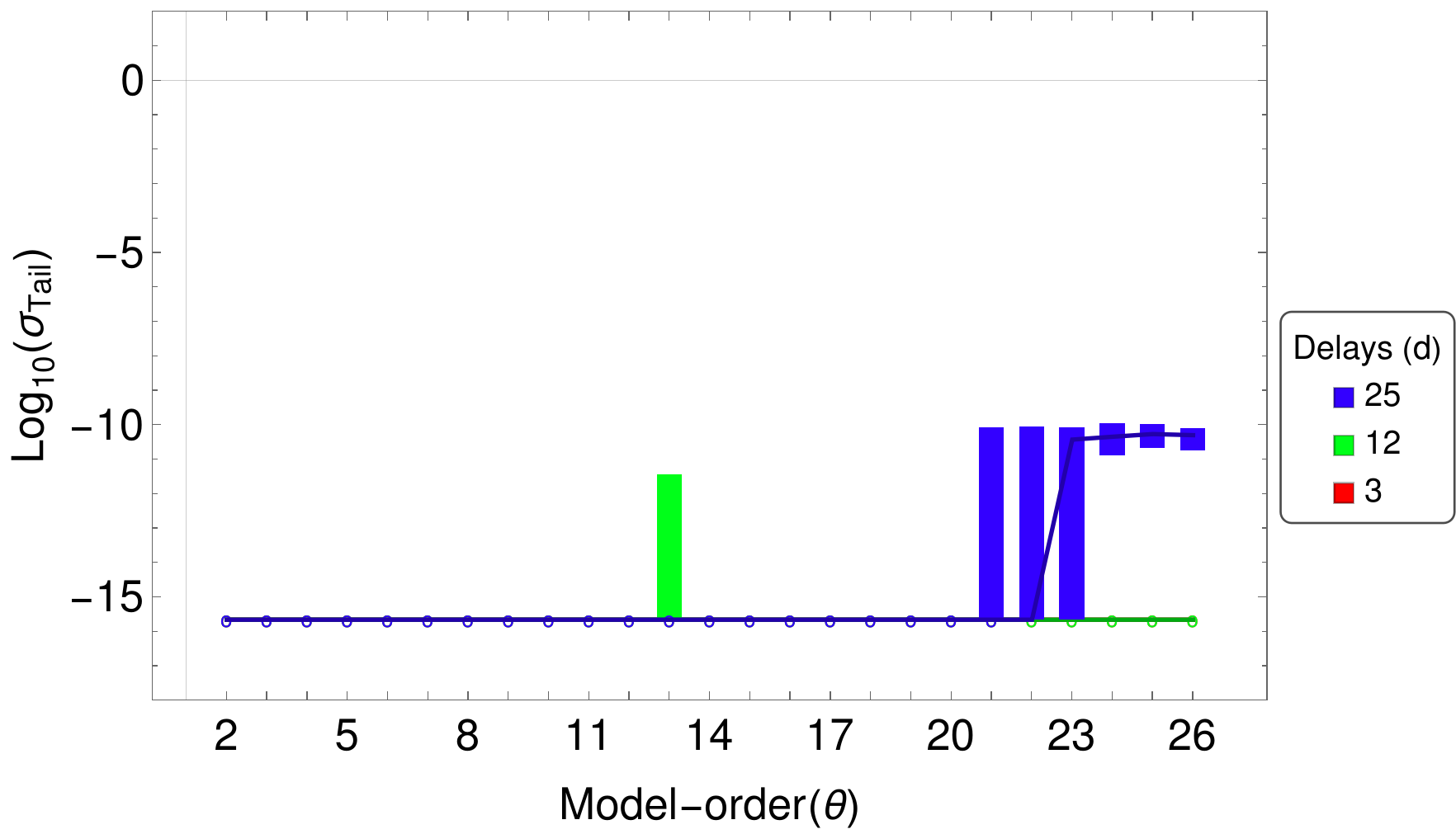}} \quad
		\subfloat[\(Re~=~ 30\times 10^3\) (Chaotic)]{\includegraphics[width = 0.45 \linewidth]{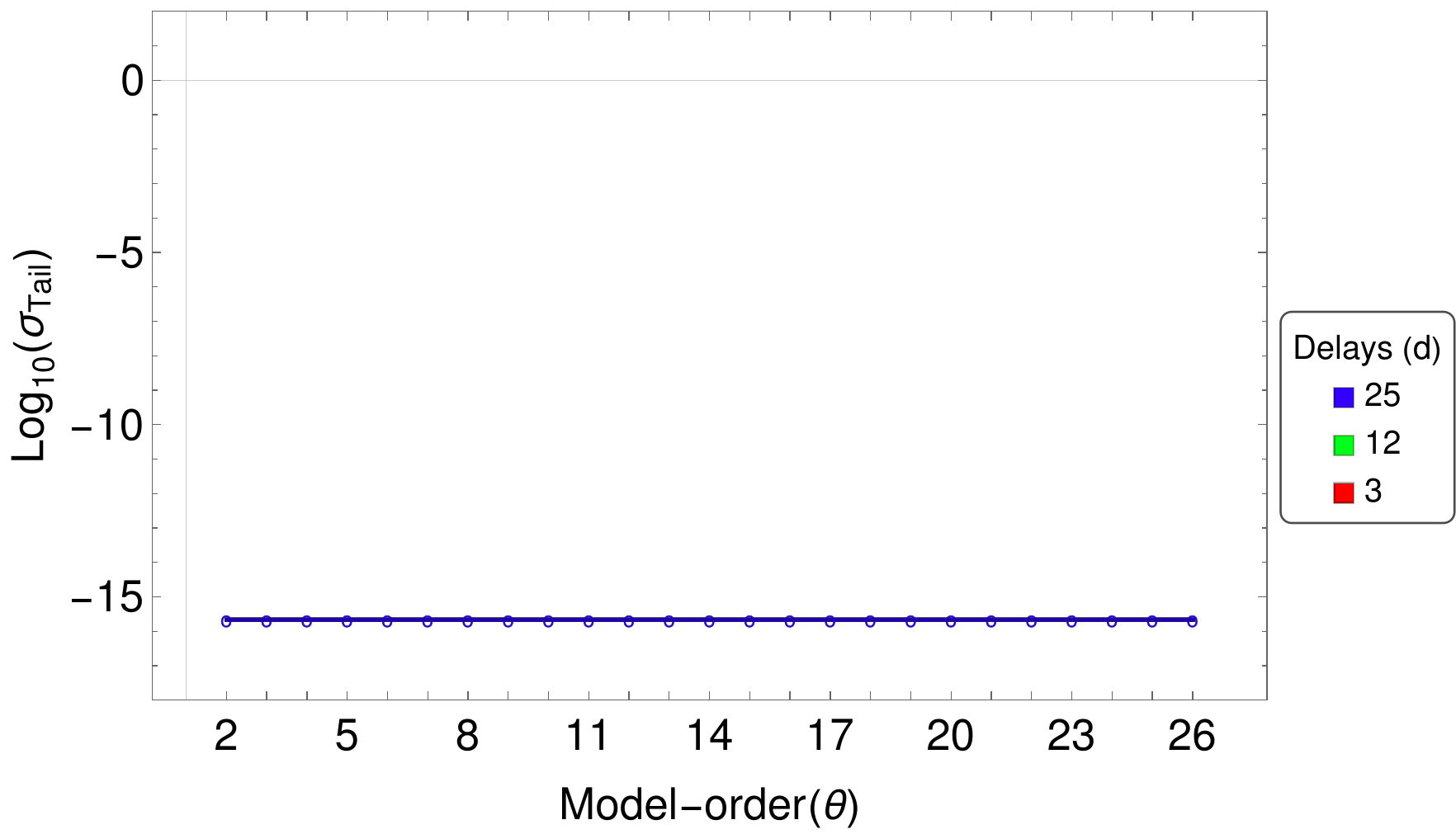}} 
		\caption{Variation of \(\sigma_{\rm Tail}\) with model-order (\(\theta\)) and delay embedding dimension (\(\fundelayscount\)) for the cavity flow studied in \cref{fig:cms_transit_cavity}.
		For Reynolds numbers between \(13 \times 10^3\) and \(20 \times 10^3\), \(\sigma_{\rm Tail}\) becomes non-trivial even if only \(12\) time delays are taken.
		Although \(\sigma_{\rm Tail}\) has a far lower plateau than seen in the preceding dynamical systems, it is safer to interpret the concomitant computations in panels (a), (b) and (c) of \cref{fig:cms_transit_cavity} as upper bounds on \DistanceToDFT.
		The case of \(Re~=~ 30\times 10^3\) is in stark contrast as \(\sigma_{\rm Tail}\) is virtually non-existent for all choices of delay-embedding dimension (\(\fundelayscount\)) and model-order (\(\theta\)).
		Hence, panel (d) in \cref{fig:cms_transit_cavity} is a faithful representation of \DistanceToDFT.
		}
		\label{fig:sigma_tail__cavity}
	\end{figure}
	
\end{document}

%% file: msps_shared.tex
\usepackage{lipsum}
\usepackage{amsfonts}
\usepackage{graphicx}
\usepackage{epstopdf}
\usepackage{stmaryrd}

\usepackage{algorithmic}
\Crefname{ALC@unique}{Line}{Lines}

\ifpdf
\DeclareGraphicsExtensions{.eps,.pdf,.png,.jpg}
\else
\DeclareGraphicsExtensions{.eps}
\fi
\graphicspath{{plots/}{./ms_plots/}}

\newsiamremark{remark}{Remark}
\newsiamremark{hypothesis}{Hypothesis}
\crefname{hypothesis}{Hypothesis}{Hypotheses}
\newsiamthm{claim}{Claim}

\newcommand{\myvec}[1]{\boldsymbol{\mathbf{#1}}}
\newcommand{\mymat}[1]{\mathbf{#1}}

\newcommand{\euler}{\mathrm{e}}
\newcommand{\ramuno}{\mathrm{i}}

\newcommand{\state}{\myvec{\upsilon}}
\newcommand{\dictionary}{\myvec{\psi}}
\newcommand{\KEFVector}{\myvec{\phi}}

\newcommand{\ssdim}{p}
\newcommand{\dictionarylength}{m}
\newcommand{\fundelayscount}{d}
\newcommand{\finitekissdim}{r}
\newcommand{\undelayedDMDorder}{n}
\newcommand{\delayedDMDorder}{\theta}

\newcommand{\koopman}{\(U\)}

\newcommand{\mean}{\myvec{\mu}}
\newcommand{\lti}{\textrm{LTI}}

\newcommand{\opt}[1]{#1^*}
\newcommand{\dummy}[1]{\hat{#1}}
\newcommand{\reduced}[1]{\tilde{#1}}

\newcommand{\KEFs}{\{\,\phi_i\,\}_{i=1}^\finitekissdim}
\newcommand{\KEvals}{ \{\,\lambda_i\,\}_{i=1}^\finitekissdim}

\newcommand{\DMDEvals}{\{\,\dummy{\lambda}_i\,\}_{i=1}^\undelayedDMDorder}
\newcommand{\DMDModes}{\{\,\dummy{\myvec{d}_i}\,\}_{i=1}^\undelayedDMDorder}

\newcommand{\muDMD}{\mean\textrm{DMD}}
\newcommand{\msubDMDEvals}{\{\,\reduced{\lambda}_i\,\}_{i=1}^\undelayedDMDorder}
\newcommand{\msubDMDmodes}{\{\,\reduced{\myvec{d}_i}\,\}_{i=1}^\undelayedDMDorder}

\newcommand{\delayedmuDMD}{\fundelayscount\muDMD}
\newcommand{\delayedmsubDMDEvals}{\{\,\check{\lambda}_i\,\}_{i=1}^\delayedDMDorder}
\newcommand{\delayedmsubDMDmodes}{\{\,\check{\myvec{d}_i}\,\}_{i=1}^\delayedDMDorder}

\newcommand{\DistanceToDFT}{\texttt{Relative distance to DFT}}
\newcommand{\obs}{\mathrm{obs}}

\newcommand{\manualQED}{\(\qquad\qquad\qquad \qquad\qquad\qquad \qquad\qquad\qquad \qquad\qquad\qquad \qquad\qquad\qquad \qquad\qquad \proofbox\)}

\newcommand{\PsiLiesInNRSpanOfrDistinctKEFs}{\hyperref[defn:NonRedundantSpanOfDistinctKEFs]{\(\dictionary\) lies in the non-redundant span of \(\finitekissdim\) distinct KEFs}}

\newcommand{\ICBeSpectrallyInformative}{\hyperref[defn:SpectrallyInformativeState]{the initial condition \(\state_1\) is spectrally informative}}

\newcommand{\SpectrallyInformative}{\hyperref[defn:SpectrallyInformativeState]{spectrally informative}}

\newcommand{\mypara}[1]{\paragraph{#1}}
\newcommand{\mysubpara}[1]{\subparagraph{#1}}

\newcommand{\EmphasiseReduction}[1]{{\color{blue} #1}}
\newcommand{\lowrankapprox}[1]{\left\llbracket#1\right\rrbracket}

\usepackage{amssymb} 
\usepackage[caption=false]{subfig}
\usepackage{easyReview}
\usepackage{centernot}
\usepackage{booktabs}
\usepackage{mathtools}

\title{Clarifying the effect of mean subtraction on Dynamic Mode Decomposition}

\author{Gowtham S Seenivasaharagavan\thanks{Department of Mechanical Engineering,
		University of California Santa Barbara} \and Milan Korda\thanks{Methods and Algorithms for Control,	LAAS-CNRS} \and Hassan Arbabi \thanks{Department of Mechanical Engineering,	Massachusetts Institute of Technology} \and Igor Mezi{\'c}\thanks{Department of Mechanical Engineering,
		University of California Santa Barbara}}

\usepackage{amsopn}
\DeclareMathOperator{\diag}{diag}
